\documentclass[ejs]{imsart}

\RequirePackage[OT1]{fontenc}
\RequirePackage{amsthm,amsmath}

\RequirePackage{natbib}
\RequirePackage[colorlinks,citecolor=blue,urlcolor=blue]{hyperref}

\usepackage{graphicx}
\usepackage{amssymb}
\usepackage{amsmath}
\usepackage{natbib}
\usepackage{amsthm}
\usepackage{aliascnt}
\usepackage{amsfonts}
\usepackage{epstopdf}
\usepackage[linesnumbered,
vlined]{algorithm2e}
\usepackage{xargs}
\usepackage{bbm}
\usepackage{mathtools}
\usepackage{stmaryrd}
\usepackage{ifthen}
\usepackage{float}
\usepackage{upgreek}
\usepackage{mathrsfs}
\usepackage[textwidth=4cm, textsize=footnotesize]{todonotes}
\usepackage{hyperref}
\usepackage{cleveref}

\pubyear{2019}
\volume{0}
\issue{0}
\firstpage{1}
\lastpage{8}

\startlocaldefs
\numberwithin{equation}{section}
\theoremstyle{plain}

\endlocaldefs

\newtheorem{theorem}{Theorem}

\newaliascnt{definition}{theorem}
\newtheorem{definition}[definition]{Definition}
\aliascntresetthe{definition}

\newaliascnt{lemma}{theorem}
\newtheorem{lemma}[lemma]{Lemma}
\aliascntresetthe{lemma}

\crefname{lemma}{lemma}{lemmas}
\Crefname{Lemma}{Lemma}{Lemmas}

\newaliascnt{remark}{theorem}

\aliascntresetthe{remark}

\newtheorem{example}{Example}

\newcommand{\A}{U}
\newcommand{\allgr}{\bar{\mathcal{G}}}
\newcommand{\gen}[1]{J_{#1}}
\newcommandx{\bk}[2][1=]{
\ifthenelse{\equal{#1}{}}
{\kernel{R}_{#2}}
{\kernel{R}_{#2} \langle #1 \rangle}
}

\newcommand{\CSF}{{CSF}}
\newcommand{\cl}{\cliqueletter}
\newcommand{\cliqueletter}{Q}
\newcommand{\clset}[1]{
\ifthenelse{\equal{#1}{}}
{\mathcal{\cliqueletter}}
{\mathcal{\cliqueletter}(#1)}
}

\newcommand{\cm}[1]{|#1|}
\newcommandx{\comb}[2][1=]{
\ifthenelse{\equal{#1}{}}
{S_{#2}}
{S_{#1 \mid #2}}
}
\newcommand{\combkernel}[1]{\kernel{\Sigma}_{#1}}
\newcommand{\combkernelpath}[1]{\bar{\kernel{\Sigma}}_{#1}}
\newcommand{\combmeas}[1]{\sigma_{#1}}
\newcommand{\combpart}[2]{\varsigma_{#1}^{#2}}
\newcommand{\combsp}[1]{\spc{S}_{#1}}
\newcommand{\cond}{\,|\,}
\newcommandx{\ct}[1][1=]{\ifthenelse{\equal{#1}{}}{\mathsf{I}}{\mathsf{I}_{#1}}}
\newcommand{\Dir}{\mathrm{Dir}}
\newcommand{\disc}{\mathsf{Pr}}
\newcommand{\dens}{f}
\newcommandx{\df}[1][1=]{\ifthenelse{\equal{#1}{}}{\delta}{\delta_{#1}}}
\newcommand{\E}{\mathbb{E}}
\newcommand{\epart}[2]{\xi_{#1}^{#2}}
\newcommand{\eqdef}{\vcentcolon=}
\newcommand{\exponent}{\mathbf{1}}
\newcommand{\G}[1]{\kernel{G}_{#1}}
\newcommandx{\graph}[1][1=]{\ifthenelse{\equal{#1}{}}{G}{G_{#1}}}
\newcommandx{\graphedgeset}[1][1=]{\ifthenelse{\equal{#1}{}}{E}{E_{#1}}}
\newcommandx{\graphfd}[1][1=]{\ifthenelse{\equal{#1}{}}{\mathcal{G}}{\mathcal{G}_{#1}}}
\newcommand{\graphmap}{\gamma}

\newcommandx{\graphnodeset}[1][1=]{\ifthenelse{\equal{#1}{}}{V}{V_{#1}}}
\newcommandx{\graphsp}[1][1=]{\ifthenelse{\equal{#1}{}}{\mathcal{G}}{\mathcal{G}_{#1}}}
\newcommandx{\graphtarg}[1][1=]{\ifthenelse{\equal{#1}{}}{\graphtargletter^\star}{\graphtargletter^\star \langle #1 \rangle}}
\newcommand{\graphtargletter}{\eta}
\newcommand{\grtr}{\mathcal{T}}
\newcommand{\hypcond}{;}
\newcommand{\hyperparamletter}{\vartheta}
\newcommandx{\hyperparam}[2][1=]{\ifthenelse{\equal{#1}{}}{\hyperparamletter_{#2}}{\hyperparamletter'_{#2}(#1)}}
\newcommandx{\I}[1][1=]{\ifthenelse{\equal{#1}{}}{I}{I^{(#1)}}}
\newcommand{\ind}[2]{I_{#1}^{#2}}
\newcommand{\intvect}[2]{\llbracket #1, #2 \rrbracket}

\newcommandx{\jtnode}[1][1=]{\ifthenelse{\equal{#1}{}}{\cliqueletter}{\cliqueletter_{#1}}}

\newcommand{\kernel}[1]{\mathbf{#1}}
\newcommand{\Lp}[1]{\mathsf{L}_{#1}}
\newcommand{\maintarg}{\eta^\ast}
\newcommand{\maintargMCMC}[1]{\eta^{\ast \N,#1}}
\newcommand{\meas}[1]{\mathbb{M}(#1)}
\newcommand{\mf}[1]{\mathbb{F}(#1)}
\newcommand{\N}{N}
\newcommand{\Nmcmc}{M}
\newcommandx{\nckernel}[2][1=]{
\ifthenelse{\equal{#1}{}}
{\boldsymbol{\Gamma}_{#2}}
{\boldsymbol{\Gamma}_{#2} \langle #1 \rangle}
}

\newcommand{\nset}{\mathbb{N}}
\newcommand{\nsetpos}{\mathbb{N}}

\newcommand{\1}{\mathbbm{1}}
\newcommand{\p}{p}
\newcommandx{\param}[1][1=]{\ifthenelse{\equal{#1}{}}{\theta}{\theta_{#1}}}
\newcommandx{\paramd}[1][1=]{\ifthenelse{\equal{#1}{}}{\theta}{\theta(#1)}}

\newcommand{\paramsp}{\Theta}
\newcommand{\paramfd}{\mathcal{P}}
\newcommandx{\partarg}[2][1=]{
    \ifthenelse{\equal{#1}{}}
    {\eta^{\N}_{#2}}
    {\eta_{#2}^{\ast, \N}}
}
\newcommandx{\partinit}[1][1=]{
       \ifthenelse{\equal{#1}{}}
       {\kappa}
       {\kappa^\ast \langle #1 \rangle}
}
\newcommandx{\perm}[2][2=]{
\ifthenelse{\equal{#2}{}}
{s_{#1}}
{s_{#1}(#2)}
}
\newcommand{\PG}[1]{\kernel{P}_{#1}^\N}

\newcommand{\potset}{\boldsymbol{\wp}}
\newcommand{\prior}{\pi}
\newcommand{\priorfun}{\varpi}
\newcommand{\precmat}{\theta}
\newcommand{\precmatset}[1]{\mathsf{M}_{#1}^+}

\newcommand{\prob}{\mathbb{P}}
\newcommand{\probmeas}[1]{\mathbb{M}_1(#1)}
\newcommandx{\prop}[2][1=]{
\ifthenelse{\equal{#1}{}}
{\kernel{\propletter}_{#2}}
{\kernel{\propletter}_{#2}^\ast \langle #1 \rangle}
}
\newcommand{\propletter}{K}
\newcommandx{\proppath}[2][1=]{
\ifthenelse{\equal{#1}{}}
{\bar{\kernel{\propletter}}_{#2}}
{\bar{\kernel{\propletter}}_{#2} \langle #1 \rangle}
}

\newcommand{\randind}[1]{\iota}

\newcommand{\refm}{\nu}
\newcommand{\retrosupp}[1]{\mathsf{S}_{#1}}

\newcommand{\rsetnonneg}{\mathbb{R}_+}
\newcommand{\rsetpos}{\mathbb{R}_+^\ast}

\newcommand{\rmd}{d}
\newcommandx{\SMCdist}[2][1=]{
\ifthenelse{\equal{#1}{}}
{\varrho_{#2}^\N}
{\varrho_{#2}^\N \langle #1 \rangle}
}
\newcommandx{\scalmat}[1][1=]{\ifthenelse{\equal{#1}{}}{v}{v_{#1}}}
\newcommand{\sepset}[1]{
\ifthenelse{\equal{#1}{}}
{\mathcal{S}}
{\mathcal{S}(#1)}
}
\newcommand{\spc}[1]{\mathsf{#1}}
\newcommand{\subnodeset}{U}
\newcommand{\supp}{\operatorname{Supp}}

\newcommandx{\targ}[3][1=, 3=]{
   \ifthenelse{\equal{#3}{}}
   {
       \ifthenelse{\equal{#1}{}}
       {\eta_{#2}}
       {\eta^\ast \langle #2 \rangle}
   }
   {
       \ifthenelse{\equal{#1}{}}
       {\eta_{#2} \langle #3 \rangle}
       {\eta_{#2}^\ast \langle #3 \rangle}
   }
}
\newcommandx{\targMCMC}[3][1=]{
\ifthenelse{\equal{#1}{}}
{\eta_{#2}^{\N, #3}}
{\eta^* \langle #2 \rangle^{\N, #3}}
}
\newcommand{\targpath}[1]{\bar{\eta}_{#1}}

\newcommand{\trans}{\intercal}
\newcommandx{\tree}[1][1=]{\ifthenelse{\equal{#1}{}}{T}{T_{#1}}}
\newcommand{\treepart}[2]{\tau_{#1}^{#2}}
\newcommandx{\treerv}[1][1=]{\ifthenelse{\equal{#1}{}}{\mathcal T}{\mathcal T_{#1}}}
\newcommandx{\tr}[2][2=]{
\ifthenelse{\equal{#2}{}}
{\tilde{x}_{#1}}
{\tilde{x}_{#1}^{#2}}
}
\newcommand{\trexfd}[1]{\mathcal{X}_{#1}}

\newcommandx{\trex}[2][2=]{
\ifthenelse{\equal{#2}{}}
{x_{#1}}
{x_{#1}^{#2}}
}

\newcommandx{\trfd}[1][1=]{\ifthenelse{\equal{#1}{}}{\mathcal{T}}{\mathcal{T}_{#1}}}
\newcommand{\trgr}{g}
\newcommandx{\trsp}[1][1=]{\ifthenelse{\equal{#1}{}}{\mathcal{T}}{\mathcal{T}_{#1}}}
\newcommandx{\trtarg}[1][1=]{\ifthenelse{\equal{#1}{}}{\eta^\ast}{\eta^\ast \langle #1 \rangle}}

\newcommandx{\uk}[2][1=]{
\ifthenelse{\equal{#1}{}}
{\kernel{Q}_{#2}}
{\kernel{Q}_{#2} \langle #1 \rangle}
}
\newcommandx{\ungraphtarg}[1][1=]{\ifthenelse{\equal{#1}{}}{\gamma^\star}{\gamma^\star \langle #1 \rangle}}
\newcommandx{\untarg}[3][1=, 3=]{
   \ifthenelse{\equal{#3}{}}
   {
       \ifthenelse{\equal{#1}{}}
       {\gamma_{#2}}
       {\gamma^\ast \langle #2 \rangle}
   }
   {
       \ifthenelse{\equal{#1}{}}
       {\gamma_{#2} \langle #3 \rangle}
       {\gamma_{#2}^\ast \langle #3 \rangle}
   }
}
\newcommandx{\untrtarg}[1][1=]{\ifthenelse{\equal{#1}{}}{\gamma^\ast}{\gamma^\ast \langle #1 \rangle}}
\newcommand{\untargpath}[1]{\bar{\gamma}_{#1}}
\newcommandx{\unpartarg}[2][1=]{
    \ifthenelse{\equal{#1}{}}
    {\gamma_{#2}^\N}
    {\gamma_{#2}^\N  \langle #1 \rangle}
}
\newcommand{\wgt}[2]{\omega_{#1}^{#2}}
\newcommandx{\wgtfunc}[2][1=]{
\ifthenelse{\equal{#1}{}}
{w_{#2}}
{w_{#2} \langle #1 \rangle}
}
\newcommand{\wgtsum}[1]{\Omega_{#1}^\N}
\newcommandx{\x}[2][2=]{
\ifthenelse{\equal{#2}{}}
{x_{#1}}
{x_{#1}^{#2}}
}
\newcommandx{\X}[2][2=]{
\ifthenelse{\equal{#2}{}}
{X_{#1}}
{X_{#1}^{#2}}
}
\newcommandx{\xt}[2][2=]{
\ifthenelse{\equal{#2}{}}
{\tilde{x}_{#1}}
{\tilde{x}_{#1}^{#2}}
}
\newcommandx{\Xt}[2][2=]{
\ifthenelse{\equal{#2}{}}
{\tilde{X}_{#1}}
{\tilde{X}_{#1}^{#2}}
}
\newcommand{\Xset}{\mathsf{X}}
\newcommand{\Xsigma}{\mathcal{X}}
\newcommand{\Yset}{\mathsf{Y}}
\newcommand{\Ysigma}{\mathcal{Y}}
\newcommand{\Zset}{\mathsf{Z}}
\newcommand{\Zsigma}{\mathcal{Z}}
\newcommand{\xfd}[1]{\mathcal{X}_{#1}}
\newcommand{\xsp}[1]{\mathcal{X}_{#1}}

\newcommandx{\y}[1][1=]{\ifthenelse{\equal{#1}{}}{y}{y_{#1}}}
\newcommandx{\Y}[1][1=]{\ifthenelse{\equal{#1}{}}{Y}{Y_{#1}}}
\newcommandx{\ysp}[1][1=]{\ifthenelse{\equal{#1}{}}{\mathsf{Y}}{\mathsf{Y}_{#1}}}
\newcommandx{\yfd}[1][1=]{\ifthenelse{\equal{#1}{}}{\mathcal{Y}}{\mathcal{Y}_{#1}}}
\newcommandx{\z}[2][2=]{
\ifthenelse{\equal{#2}{}}
{z_{#1}}
{z_{#1}^{#2}}
}
\newcommandx{\Z}[2][2=]{
\ifthenelse{\equal{#2}{}}
{Z_{#1}}
{Z_{#1}^{#2}}
}

\newcommand{\zpart}[2]{\tau_{#1}^{#2}}
\newcommand{\zsp}[1]{\mathsf{Z}_{#1}}

\begin{document}

\begin{frontmatter}
\title{Bayesian learning of weakly structural Markov graph laws using sequential Monte Carlo methods}
\runtitle{Bayesian learning of WSM graph laws using SMC methods}

\begin{aug}
\author{\fnms{Jimmy} \snm{Olsson}\thanksref{t3}\ead[label=e1]{jimmyol@math.kth.se}},
\author{\fnms{Tatjana} \snm{Pavlenko}\thanksref{t2}\ead[label=e2]{pavlenko@math.kth.se}}
\and
\author{\fnms{Felix L.} \snm{Rios}\ead[label=e3]{flrios@math.kth.se}}

\address{Department of Mathematics\\
The Royal Institute of Technology, Stockholm, Sweden\\
\printead{e1,e2,e3}}

\thankstext{t2}{TP gratefully acknowledges support by the Swedish Research Council, Grant C0595201. }
\thankstext{t3}{JO gratefully acknowledges support by the Swedish Research Council, Grant 2018-05230.}
\runauthor{J. Olsson et al.}

\affiliation{Some University and Another University}

\end{aug}

\begin{abstract}
We present a sequential sampling methodology for weakly structural Markov laws, arising naturally in a Bayesian structure learning context for decomposable graphical models.
As a key component of our suggested approach, we show that the problem of graph estimation, which in general lacks natural sequential interpretation, can be recast into a sequential setting by proposing a recursive Feynman-Kac model that generates a flow of junction tree distributions over a space of increasing dimensions.
We focus on particle McMC methods to provide samples on this space, in particular on particle Gibbs (PG), as it allows for generating McMC chains with global moves on an underlying space of decomposable graphs.
To further improve the PG mixing properties, we incorporate a systematic refreshment step implemented through direct sampling from a backward kernel.
The theoretical properties of the algorithm are investigated, showing that the proposed refreshment step improves the performance in terms of asymptotic variance of the estimated distribution.
The suggested sampling methodology is illustrated through a collection of numerical examples demonstrating high accuracy in Bayesian graph structure learning in both discrete and continuous graphical models.

\end{abstract}


\tableofcontents
\end{frontmatter}

\section{Introduction}\label{sec:intro}
Understanding the underlying dependence structure of a multivariate distribution is becoming increasingly important in modern applications when analysing complex data.
These dependencies are conveniently represented by a \emph{graphical model} (GM) in which the set of nodes represents feature variables in the model and the set of edges encodes the dependence structure.
A specific family of undirected graphical models extensively studied in the literature are those which are Markov with respect to decomposable graphs, usually referred to as \emph{decomposable graphical models} (DGMs), to which we restrict our attention in the present paper. 
For these models, joint densities factorise into products of densities over certain subsets of nodes described by \emph{cliques} and \emph{separators}. This makes such models attractive from a computation point of view, since key statistical quantities - such as likelihood ratios and prior distributions - can be calculated or specified locally and graphs can be build up sequentially; see e.g. \citet{lauritzen1996}.

Recently, the family of \emph{weakly structural Markov} (WSM) probabilistic laws for decomposable graphs was introduced by \citet{green2017structural}, providing an analogous clique-separator factorisation for the graph law as for the data distribution.
In this paper, we focus on a fully Bayesian and computational approach for inferring posterior graph laws given observed data, a process usually called \emph{structure learning}.
Specifically, we consider \emph{strong hyper-} and weakly structural Markov prior laws for the model parameters and graphs respectively, so that the resulting graph posterior also factorises over the set of cliques and separators; see e.g. \citet{dawid1993}. 

The common strategy of Bayesian structure learning i based on the class of Markov chain Monte Carlo (McMC) methods such as e.g. the Metropolis-Hastings sampling scheme.
These methods generate, by performing local perturbations on the edge set, Markov chains by either operating directly on the space of decomposable graph or their corresponding junction trees; see for example \citet{Giudici01121999, 10.2307/2673658,2005,Green01032013}.
Further pertinent approaches include e.g. \cite{stingo2015efficient} who focus on Gaussian DGMs and propose edge moves by dynamically updating the perfect sequence of the cliques in the graph. 
A completely different strategy is presented in \citet{elmasri2017decomposable, elmasri2017sub} where a node-driven McMC sampler operates on tree-dependent bipartite graphs.  

The main issue for the above-mentioned samplers as well as other McMC strategies based on local moves is the limited mobility of their corresponding Markov chains, since at each step, only a small part of the edge set is altered. 
To tackle this issue, we present a procedure for recasting the problem of structure learning in WSM laws, which in general lacks natural sequential interpretation, into a sequential setting by an auxiliary construction that we refer to as a \emph{temporal embedding}, relying partly on the methodology of \emph{sequential Monte Carlo (SMC) samplers}; see \cite{10.2307/3879283}.
Specifically, we propose a recursive Feynman-Kac model which generates a flow of junction tree distributions over a space of increasing dimensions and develop an efficient SMC sampler on this space.
The SMC algorithm is then incorporated as an inner loop of a particle Gibbs (PG) sampler \citep{andrieu2010particle}, providing global moves on the underlying graph space.
In order to reduce the variance and improve the mobility of the standard PG sampler, we further introduce a step of systematic refreshment by means of backwards sampling.

Our suggested temporal embedding of WSM laws is constructed by a four step \emph{temoralisation} procedure which can be summarised as follows.
The procedure is initiated by defining a family of laws on decomposable graph spaces defined on all subsets of the node set.
In the context of Bayesian structure learning, these laws will correspond to graph posteriors defined over the corresponding subsets of random variables.
The second step of the temporalisation is to extend each graph law to the space of junction tree representations. 
Following \citet{Green01032013}, this is carried through by rescaling of the underlying graph probabilities by the number of equivalent junction tree representations. 
In this construction, the marginal law of an underlying graph will be preserved from the first step. 
Now, in the context of sequential Bayesian structure learning the user may, by always processing the nodes in some given order, run the risk of overlooking dependence relations running counter to this specific order. 
It is hence desirable to allow the node processing order to be randomised. 
For this purpose, the third step of the temporalisation procedure augments the junction tree distributions to mixtures of junction tree distributions over different subsets of underlying graph nodes.
Finally, in the last step of the temporalisation process, we introduce a sequence of Markov kernels allowing the distributions formed in the third step to be embedded into a recursive Feynman-Kac-distribution flow. The distributions of the resulting Feynman-Kac flow, with "time parameter" given by the number of nodes of the underlying graph, can be sampled efficiently using sequential Monte Carlo methods.

A central part in the construction of any SMC algorithm is the design of a proposal distribution, which should both dominate the target of interest and preferably be computationally efficient.
In our case the junction tree representation introduced in the second step of the temporalisation is of key importance, since it enables us to fulfill these requirements through the so-called Christmas tree algorithm (CTA), presented in the companion paper \citet{cta}.
The CTA by construction defines a Markov kernel, with closed-form transition probabilities, that dominates the temporalised version of the graph law.
Up to our knowledge, the last property seems much harder to obtain by, e.g., operating directly on a path space of decomposable graphs; see e.g. \citet{randchord}.

The rest of the paper is structured as follows.
In Section~\ref{sec:prel} we introduce some notation and present standard theoretical results for decomposable graphs and the junction tree representation.
Section~\ref{sec:non-temp:FK:flows} presents the four stage temporalisation strategy procedure.
The SMC sampler is designed in Section~\ref{sec:particle:methods:temp:models} along with the standard PG and its systematic refreshment extension.
In Section~\ref{sec:application} we present two motivating examples showing how the WSM laws arise in a Bayesian inference context.
In Section~\ref{sec:numerics} we investigate numerically the performance of the suggested PG sampler for three examples of Bayesian structure learning in DGMs. 
Appendix~\ref{sec:appendix} contains some graph theoretical notations, proofs and a lemma.

\section{Preliminaries}\label{sec:prel}

\subsection*{Notational convention}

We will always assume that all random variables are well defined on a
common probability space $(\Omega, \mathcal{F}, \mathbb{P})$. We denote by
$\nsetpos$ the positive natural numbers and for any $(m, n) \in \nset^2$
we use $\intvect{m}{n}$ to denote the unordered set $\{m, \ldots, n\}$. By $\rsetnonneg$ and $\rsetpos$
we denote the non-negative and positive real numbers respectively. 

\subsubsection*{Measurable spaces}
Given some measurable space $(\Xset, \Xsigma)$, we denote by $\meas{\Xsigma}$
and $\probmeas{\Xsigma}$ the sets of measures and probability measures on
$(\Xset, \Xsigma)$, respectively. In the case where $\Xset$ is a finite set,
$\Xsigma$ is always assumed to be the power set $\potset(\Xset)$ of $\Xset$,
and we simply write $\meas{\Xset}$ and $\probmeas{\Xset}$ instead of
$\meas{\potset(\Xset)}$ and $\probmeas{\potset(\Xset)}$, respectively.
In the finite case, counting measures will be denoted by $\cm{\rmd x}$. We let $\mf{\Xsigma}$ be the set of measurable functions on $(\Xset, \Xsigma)$.

\subsubsection*{Kernel notation}
Let $\mu$ be a measure on some measurable space $(\Xset, \Xsigma)$.
Then for any $\mu$-integrable function $h$, we use the standard notation
$$
\mu h \eqdef \int h(x) \, \mu(\rmd x)
$$
to denote the Lebesgue integral of $h$ w.r.t. $\mu$.

In addition, let $(\Yset, \Ysigma)$ be some other measurable space and $\mathbf{K}$
some possibly unnormalised transition kernel $\mathbf{K} : \mathsf{X} \times \Ysigma
\rightarrow \rsetnonneg$. The kernel $\mathbf{K}$ induces two integral operators, one acting
on functions and the other on measures. More specifically, given a measure $\nu$ on
$(\Xset, \Xsigma)$ and a measurable function $h$ on $(\Yset, \Ysigma)$, we define
the measure
$$
	\nu \mathbf{K} : \Ysigma \ni A \mapsto \int \mathbf{K}(x, A) \, \nu(\rmd x)
$$
and the function
$$
	\mathbf{K} h : \Xset \ni x \mapsto \int h(y) \, \mathbf{K}(x, \rmd y)
$$
whenever these quantities are well-defined.

Finally, given a third measurable space $(\Zset, \Zsigma)$ and a second kernel
$\mathbf{L} : \Yset \times \Zsigma \rightarrow \mathbb{R}_+$ we define, with $\mathbf{K}$
as above, the \emph{product kernel}
$$
\mathbf{K} \mathbf{L} : \Xset \times \Zsigma \ni (x, B) \mapsto \int  \mathbf{L}(y, B) \, \mathbf{K}(x, \rmd y),
$$
whenever this is well-defined.

\subsubsection*{Decomposable graphs and junction trees}
The notion of decomposable graphs and junction trees are introduced below.
For general graph theoretical concepts and notations the reader is referred to~\ref{sec:graphtheory}.
A graph $\graph$ is called decomposable if and only if its cliques can be be arranged in a so-called \emph{junction tree}, i.e. a tree whose nodes are the cliques in $\graph$, and where for any pair of cliques $\cl$ and $\cl'$ in $\graph$, the intersection $\cl\cap \cl'$ is contained in each of the cliques on the unique path $\cl \sim \cl'$.
Decomposable graphs are sometimes alternatively termed \emph{chordal}
or \emph{triangulated}, as an equivalent requirement is that every cycle of length $4$ or more is
chorded, see e.g~\citet{diestel2005graph}.
Each edge $(\cl, \cl')$ in a junction tree is associated
with the intersection $S = \cl \cap \cl'$, which is referred to as a \emph{separator}.
Since all junction tree representations of a specific decomposable graph
$G$ have the same separators, it makes sense to speak about ``the separators
of a decomposable graph''. We denote by $\sepset{G}$ the \emph{multiset}
of separators formed by a graph $\graph$, where each separator has a
multiplicity. The set of equivalent junction tree representations of a
decomposable graph $\graph$ is denoted by $\grtr(\graph)$, and
$\mu(\graph) \eqdef | \grtr(\graph) |$ denotes the number of such
representations. The unique graph underlying a specific junction tree
$\tree$ is denoted by $\trgr(\tree)$.

\section{Temporal embedding of weakly structural Markov laws}\label{sec:non-temp:FK:flows}

From now on, let $\graphnodeset$ be a fixed set of $p \in \nsetpos$
distinct nodes. Without loss of generality, we let $\graphnodeset = \intvect{1}{p}$.
For $\subnodeset \subseteq \graphnodeset$, we denote by $\graphsp[\subnodeset]$
the space of decomposable graphs with nodes $\subnodeset$, i.e.,
$\graphsp[\subnodeset] \eqdef \{(\subnodeset, E) : E \subseteq
\subnodeset \times \subnodeset \}$. In particular, set $\graphsp
\eqdef \graphsp[V]$. In addition, let $\allgr \eqdef \cup_{\subnodeset \subseteq \graphnodeset}
\graphsp[\subnodeset]$ be the space of all decomposable graphs with nodes
given by $\graphnodeset$ or some subset of the same.

\begin{definition}\label{def:CFS}
A positive function $\graphmap$ on $\allgr$ is said to satisfy the
\emph{clique-separator factorisation} (\CSF) if for all $G \in \allgr$,
$$
	\graphmap(G) = \frac{\prod_{\cl \in \clset{G}} \graphmap(\cl)}
	{\prod_{S \in \sepset{G}} \graphmap(S)}.
$$
\end{definition}

For some given function $\graphmap$ satisfying the \CSF, the aim
of this paper is to develop a strategy for sampling from the family of so-called \emph{weakly structural Markov} laws on $\probmeas{\graphsp}$ \citep{green2017structural}, which assuming full support on \(\graphsp\), are characterised as
\begin{equation} \label{eq:graph:target}
	\graphtarg(\rmd \graph) = \frac{\ungraphtarg(\rmd \graph)}
	{\ungraphtarg \1_{\graphsp}},
\end{equation}
where
$$
	\ungraphtarg(\rmd \graph) \eqdef
	\graphmap |_{\graphsp} (\graph) \, \cm{\rmd \graph},
$$
with $\graphmap |_{\graphsp}$ denoting the restriction of $\graphmap$
to $\graphsp$ and $\cm{\rmd \graph}$ the counting measure on $\graphsp$.
The normalising constant $\ungraphtarg \1_{\graphsp} =
\sum_{\graph \in \graphsp} \ungraphtarg(\graph)$ will be considered
as intractable, as computing the same requires the summation of over
the whole space $\graphsp$, which is impractical as the cardinality of
$\graphsp$ is immense already for moderate $p$.

Our goal is now to develop an efficient strategy for sampling from distributions of form~\eqref{eq:graph:target}.
As mentioned in the introduction, particle McMC methods are appealing as
these allow McMC chains with ``global'' moves to be defined also on large spaces.
However, unlike our setting, SMC methods sample from \emph{sequences}
of distributions, and a key ingredient of our developments is hence to provide
an auxiliary, sequential reformulation of the sampling problem under consideration.
This construction, which we will refer to as \emph{temporalisation}, comprise four
steps described in the following.
\subsubsection*{Step I}
Using the function $\graphmap$ inducing the target~\eqref{eq:graph:target} of interest, define, for each
$\A \subseteq \graphnodeset$, the measure
$$
	\graphtarg[\A](\rmd \graph)
	= \frac{\ungraphtarg[\A](\rmd \graph)}
	{\ungraphtarg[\A] \1_{\graphsp[\A]}}
$$
in $\probmeas{\graphsp[\A]}$, where
$$
	\ungraphtarg[\A](\rmd \graph)
	\eqdef \graphmap |_{\graphsp[\A]} (\graph)
	\, \cm{\rmd \graph},
$$
with $\graphmap |_{\graphsp[\A]}$ denoting the restriction of
$\graphmap$ to $\graphsp[\A]$ and $\cm{\rmd \graph}$ the
counting measure on $\graphsp[\A]$. Note that
$\graphtarg[\graphnodeset]$ coincides with $\graphtarg$,
the target of interest. As usual, we will let the same symbols $\ungraphtarg[\A]$
and $\graphtarg[\A]$ denote the probability functions of these measures.

\subsubsection*{Step II}
Extend each distribution $\graphtarg[\A]$ to
a distribution $\trtarg[\A]$ on $\trsp[\A] \eqdef
\cup_{\graph \in \graphsp[\A]} \grtr(\graph)$, the space of junction
tree representations of graphs in $\graphsp[\A]$. 
Following \citet{Green01032013}, one way of carrying through this extension is to define, for each $\A \subseteq \graphnodeset$, the measure
\begin{equation} \label{def:extended:tree:meas}
	\trtarg[\A](\rmd \tree)
	\eqdef \frac{\untrtarg[\A](\rmd \tree)}
	{\untrtarg[\A] \1_{\trsp[\A]}}
\end{equation}
in $\probmeas{\trfd[\A]}$, where
$$
\untrtarg[\A](\rmd \tree) \eqdef
\frac{\ungraphtarg[\A] \circ \trgr(\tree)}
{\mu \circ \trgr(\tree)} \, \cm{\rmd \tree},
$$
with $\cm{\rmd \tree}$ denoting the counting measure
on $\trfd[\A]$. In particular, we set $\untrtarg = \untrtarg[\graphnodeset]$
and $\trtarg = \trtarg[\graphnodeset]$.

\subsubsection*{Step III}
Let, for all $m \in \intvect{1}{p}$,
$\combsp{m}$ be the space of all $m$-combinations of elements in
$\intvect{1}{p}$. 
An element $\comb{m} \in \combsp{m}$ is of form $\comb{m}
= (\comb[1]{m}, \ldots, \comb[m]{m})$ where
$\{\comb[\ell]{m} \}_{\ell = 1}^m \subseteq \intvect{1}{p}$ are distinct.
In particular, $\combsp{p} = \{(1,\ldots, \p) \}$.
For $(\ell, \ell') \in \intvect{1}{m}^2$ such that $\ell \leq \ell'$,
we denote $\comb[\ell:\ell']{m} \eqdef (\comb[\ell]{m}, \ldots, \comb[\ell']{m})$.
In addition, we define, for all $m \in \intvect{1}{p}$,
the extended state spaces
$$
	\xsp{m} \eqdef \bigcup_{\comb{m} \in \combsp{m}}
	\left( \{\comb{m} \} \times \trsp[\comb{m}] \right),
$$
and, for some given discrete probability distribution
$\combmeas{m}$ on $\combsp{m}$, extended target distributions
$$
	\targ{m}(\rmd \x{m})
	=  \frac{\untarg{m}(\rmd \x{m})}{\untarg{m} \1_{\xsp{m}}},
$$
in $\probmeas{\xfd{m}}$, where
$$
	\untarg{m}(\rmd \x{m})
	= \untarg{m}(\rmd \comb{m}, \rmd \tree[m])
	\eqdef \untarg[marg]{\comb{m}}(\rmd \tree[m])
	\, \combmeas{m}(\rmd \comb{m}).
$$
Here we have chosen to write $\tree[m]$ instead of
$\tree[\comb{m}]$ in order to avoid double subscript notation.
The measures $\{\combmeas{m} \}_{m = 1}^p$
are supposed to satisfy the recursion
$$
	\combmeas{m + 1} = \combmeas{m} \combkernelpath{m},
$$
where
\begin{equation} \label{eq:def:comb:kernel:path}
	\combkernelpath{m}(\comb{m}, \rmd \comb{m + 1})
	\eqdef \delta_{\comb{m}}(\rmd \comb[1:m]{m + 1})
	\, \combkernel{m}(\comb[1:m]{m + 1}, \rmd \comb[m + 1]{m + 1}),
\end{equation}
with $\combkernel{m}$ being a Markov transition kernel from
$\combsp{m}$ to $\intvect{1}{p}$ such that $\combkernel{m}(\comb{m}, j) = 0$
for all $j \in \comb{m}$. 

\subsubsection*{Step IV}
Let $\{ \bk{m} \}_{m = 1}^{p - 1}$ be a sequence of Markov transition kernels acting in the
\emph{reversed} direction, i.e., for each $m$, $\bk{m} : \xsp{m + 1}
\times \potset(\xsp{m}) \rightarrow [0, 1]$, and define, following
\citet{10.2307/3879283}, for all $m \in \intvect{1}{p}$,
\begin{equation} \label{eq:def:extended:target}
    \untargpath{m}(\rmd \x{1:m}) \eqdef \untarg{m}(\rmd \x{m})
    \prod_{\ell = 1}^{m - 1} \bk{\ell}(\x{\ell + 1}, \rmd \x{\ell})
\end{equation}
and
\begin{equation} \label{eq:def:extended:untarget}
	\targpath{m}(\rmd \x{1:m})
	\eqdef \frac{\untargpath{m}(\rmd \x{1:m})}
	{\untargpath{m} \1_{\xsp{1:m}}}
	= \frac{\untargpath{m}(\rmd \x{1:m})}
	{\untarg{m} \1_{\xsp{m}}},
\end{equation}
where $\xsp{1:m} \eqdef \prod_{\ell = 1}^m \xsp{\ell}$.\footnote{Here and
in the following, we put a bar on top of a measure, kernel, function, etc.,
in order to indicate that the quantity is defined on a path space.}

\bigskip
Trivially, $\targpath{m}$ allows $\targ{m}$ as a marginal distribution
with respect to the last component $\x{m}$, therefore we regard $\eqref{eq:def:extended:untarget}$ as a \emph{temporal embedding} of $\eqref{eq:graph:target}$. 
We conclude this section by some remarks and comments on the steps of the above described procedure.
We first note that, in step II, by Lemma~\ref{lemma:preservation:marginal} (see \ref{sec:proofslemmas}), 
for $\graph \in \graphsp[\A]$,
\begin{equation*} \label{eq:T:given:G}
	\prob_{\trtarg[\A]} \left( \tau = \tree \cond \trgr(\tau) = \graph \right)
	= \frac{\prob_{\trtarg[\A]} \left( \tau = \tree, \trgr(\tau) = \graph \right)}
	{\prob_{\trtarg[\A]} \left(\trgr(\tau) = \graph \right)}
	= \frac{\trtarg[\A](\tree)}{\graphtarg[\A](\graph)}
	\1_{\{ \graph = \trgr(\tree) \}}.
\end{equation*}
Moreover, using~\eqref{eq:trtarg:alt:representation},
the right hand side 
can be expressed as
\begin{align*}
	 \frac{\trtarg[\A](\tree)}{\graphtarg[\A](\graph)} \1_{\{ \graph = \trgr(\tree) \}}
	 = \frac{\graphtarg[\A] \circ \trgr(\tree)}{\graphtarg[\A](\graph) \mu \circ \trgr(\tree)
	 \1_{\{ \graph = \trgr(\tree) \}}} = \frac{1}{\mu(\graph)} \1_{\trsp(\graph)}(\tree),
\end{align*}
i.e., under $\trtarg[\A]$, conditionally on the event
$\{ \trgr(\tau) = \graph \}$, the tree $\tau$ is \emph{uniformly}
distributed over the set $\trsp(\graph)$ (recall that $\mu(\graph)$
is the cardinality of $\trsp(\graph)$). In other words, a draw from
$\trtarg[\A]$ can be generated by drawing a graph according to
$\graphtarg[\A]$ and then drawing a tree uniformly over
all junction tree representations of that graph.

In step III, 
each $\untarg{m}(\rmd \x{m})$ has a density
$\untarg{m}(\x{m})=\untarg[marg]{\comb{m}}(\tree[m])
\combmeas{m}(\comb{m})$ (by abuse of notation, we reuse
the same symbol) w.r.t. $\cm{\rmd \x{m}}$, the counting measure
on $\xsp{m}$. Moreover, since $\combmeas{p} = \delta_{\intvect{1}{p}}$,
$\maintarg$ is the marginal of $\targ{p}$ with respect to the
$\tree[p]$ component. 
Further we note that $\combkernelpath{m}$ is a
Markov transition kernel from $\combsp{m}$ to $\combsp{m + 1}$.
In other words, $\combkernelpath{m}$ transforms a given $m$-combination
$\comb{m}$ into an $(m + 1)$-combination $\comb{m + 1}$ by
selecting randomly an element $s^*$ from the (non-empty) set
$\intvect{1}{p} \setminus \comb{m}$ according to
$\combkernel{m}(\comb{m}, \cdot)$ and adding the same to
$\comb{m}$. 
When selecting $s^*$, several approaches are
possible; $s^*$ can, e.g., be selected randomly from the set
$\{s \in \intvect{1}{p} : \min_{s' \in \comb{m}} |s - s'| \leq \delta\}$
for some prespecified distance $\delta \in \intvect{1}{p}$.
Also the initial distribution $\combmeas{1}$ can be designed freely,
e.g., as the uniform distribution over $\intvect{1}{p}$.

In step IV, as the reversed kernels are assumed to be Markovian and known to the user,
each extended target distribution $\targpath{m}$ is known up to the same
normalising constant $\untarg{m} \1_{\xsp{m}}$ as its marginal $\targ{m}$.
The algorithm that we propose is based on the observation that the
distribution flow $\{ \targ{m} \}_{m = 1}^p$ satisfies the recursive
\emph{Feynman-Kac model}
\begin{equation} \label{eq:Feynman:Kac:model}
    \targ{m + 1}(\rmd \x{m + 1}) = \frac{\targ{m} \uk{m}(\rmd \x{m + 1})}
    {\targ{m} \uk{m} \1_{\xsp{m + 1}}} \quad (m \in \intvect{1}{p - 1}),
\end{equation}
where we have defined the un-normalised transition kernel
\[
	\uk{m}(\x{m}, \rmd \x{m + 1}) \eqdef
	\begin{cases}
	    \displaystyle \frac{\untarg{m + 1}(\rmd \x{m + 1})
	    \, \bk{m}(\x{m + 1}, \x{m})}{\untarg{m}(\x{m})},
	    & \x{m} \in \supp(\untarg{m}), \\
	    0, & \mbox{otherwise}.
	\end{cases}
\]

In the SMC sampler framework of~\citet{10.2307/3879283},
focus is set on sampling from a \emph{sequence} of probability densities
known up to normalising constants and defined on the \emph{same} state
space. In this context, the authors propose to transform the given distribution
sequence into a sequence of distributions over state spaces of increasing
dimension (given by powers of the original space) by means an auxiliary
Markovian transition kernel. In this construction, each extended distribution
is of form (\ref{eq:def:extended:target}), with $\xsp{m} \equiv \xsp{}$ for all
$m$, and allows the original density of interest as a marginal with respect
to the last component $\x{m}$. Having access to such a flow of distributions
over spaces of increasing dimensions, standard SMC methods provide
numerically stable online approximation of the marginals, the latter satisfying
a Feynman-Kac recursion of form (\ref{eq:Feynman:Kac:model}).

In our case, we arrive at the recursion (\ref{eq:Feynman:Kac:model})
from an entirely different direction, i.e., by starting off with a \emph{single}
distribution defined on a possibly high-dimensional space and constructing
an auxiliary sequence of increasingly complex distributions used for
directing an SMC particle sample towards the distribution of interest
(see the next section).

\section{Particle approximation of temporalised weakly structural Markov laws}\label{sec:particle:methods:temp:models}
In the following we discuss how to obtain a particle interpretation of the recursion \eqref{eq:Feynman:Kac:model}. 
Assume for the moment that we have at hand a sequence $\{ \prop{m} \}_{m = 1}^{p - 1}$
of proposal kernels such that $\uk{m}(\x{m}, \cdot) \ll \prop{m}(\x{m}, \cdot)$
for all $m \in \intvect{1}{p - 1}$ and all $\x{m} \in \xsp{m}$. In our applications,
we will let these proposal kernels correspond to the so-called \emph{Christmas tree algorithm} (CTA) proposed in the companion paper \citet{cta} 
and overviewed in Section~\ref{sec:junction:tree:expanders}.
\subsection{Sequential Monte Carlo approximation} 
\label{sub:sequential_monte_carlo_approximation}


We proceed recursively and assume that we are given a sample
$\{ (\epart{m}{i}, \wgt{m}{i}) \}_{i = 1}^\N$ of particles, each particle
$\epart{m}{i} = (\combpart{m}{i}, \treepart{m}{i})$ being a random
draw in $\xsp{m}$ (more specifically, $\combpart{m}{i}$ is a random
$m$-combination in $\intvect{1}{p}$ and $\treepart{m}{i}$ a random
draw in $\zsp{\combpart{m}{i}}$), with associated importance weights
(the $\wgt{m}{i}$'s) approximating $\targ{m}$ in the sense that for all
$h \in \mf{\xfd{m}}$,
$$
    \partarg{m} h \simeq \targ{m} h \quad \mbox{as $\N \rightarrow \infty$,}
$$
where
$$
    \partarg{m}(\rmd \x{m}) \eqdef \sum_{i = 1}^\N
    \frac{\wgt{m}{i}}{\wgtsum{m}} \delta_{\epart{m}{i}}(\rmd \x{m}),
$$
with $\wgtsum{m} \eqdef \sum_{i = 1}^\N \wgt{m}{i}$, denotes the weighted
empirical measure associated with the particle sample. 

In order to produce
an updated particle sample $\{ (\epart{m + 1}{i}, \wgt{m + 1}{i}) \}_{i = 1}^\N$
approximating $\partarg{m + 1}$, we plug $\partarg{m}$ into the recursion
\eqref{eq:Feynman:Kac:model} and sample from the resulting distribution
$$
    \frac{\partarg{m} \uk{m}(\rmd \x{m + 1})}
    {\partarg{m} \uk{m} \1_{\xsp{m + 1}}}
    = \sum_{i = 1}^\N \frac{\wgt{m}{i} \uk{m}(\epart{m}{i}, \rmd \x{m + 1})}
    {\sum_{\ell = 1}^\N \wgt{m}{\ell} \uk{m} \1_{\xsp{m + 1}} (\epart{m}{\ell})}
$$
by means of importance sampling.  For this purpose we first extend
the previous measure to the index component, yielding the mixture
$$
    \check{\eta}_{m + 1}^\N(\rmd i, \rmd \x{m + 1})
    \eqdef \frac{\wgt{m}{i} \uk{m}(\epart{m}{i}, \rmd \x{m + 1})}
    {\sum_{\ell = 1}^\N \wgt{m}{\ell} \uk{m} \1_{\xsp{m + 1}}
    (\epart{m}{\ell})} \, \cm{\rmd i}
$$
on the product space $\intvect{1}{\N} \times \xsp{m + 1}$,
and sample from the latter by drawing i.i.d. samples
$\{ (\ind{m + 1}{i}, \epart{m + 1}{i}) \}_{i = 1}^\N$ from
the proposal distribution
$$
	\rho_{m + 1}^\N (\rmd i, \rmd \x{m + 1}) \eqdef \frac{\wgt{m}{i}}{\wgtsum{m}}
	\prop{m}(\epart{m}{i}, \rmd \x{m + 1}) \, \cm{\rmd i}.
$$
Each draw $(\ind{m + 1}{i}, \epart{m + 1}{i})$ is assigned
an importance weight
$$
	\wgt{m + 1}{i} \eqdef \wgtfunc{m}(\epart{m}{{\ind{m + 1}{i}}}, \epart{m + 1}{i})
	\propto \frac{\rmd \check{\eta}_{m + 1}^\N}{\rmd \rho_{m + 1}^\N}
	    (\ind{m + 1}{i}, \epart{m + 1}{i}),
$$
where we have defined the importance weight function
\begin{equation} \label{eq:importance:weights}
	\wgtfunc{m}(\x{m}, \x{m + 1}) \eqdef \frac{\rmd \uk{m}(\x{m}, \cdot)}
	{\rmd \prop{m}(\x{m}, \cdot)}(\x{m + 1})
	= \frac{\untarg{m + 1}(\x{m + 1}) \bk{m}(\x{m + 1}, \x{m})}
	    {\untarg{m}(\x{m}) \prop{m}(\x{m}, \x{m + 1})}.
\end{equation}
Finally, the weighted empirical measure
$$
	\partarg{m + 1}(\rmd \x{m + 1}) \eqdef \sum_{i = 1}^\N
	\frac{\wgt{m + 1}{i}}{\wgtsum{m + 1}}
	\delta_{\epart{m + 1}{i}}(\rmd \x{m + 1})
$$
is returned as an approximation of $\targ{m + 1}$.

We will always assume
that the proposal kernel $\prop{m}$ is of form
\begin{equation} \label{eq:proposal:special:form}
	\prop{m}(\x{m}, \rmd \x{m + 1})
	= \combkernelpath{m}(\comb{m}, \rmd \comb{m + 1}) \,
	\prop[\comb{m}, \comb{m + 1}]{m}(\tree[m], \rmd \tree[m + 1]),
\end{equation}
where $ \combkernelpath{m}$ is defined in \eqref{eq:def:comb:kernel:path}
and for all $(\comb{m}, \comb{m + 1}) \in \combsp{m} \times \combsp{m + 1}$,
$\prop[\comb{m}, \comb{m + 1}]{m}$ is a Markov transition kernel from
$\xsp{\comb{m}}$ to $\xsp{\comb{m + 1}}$. Each law
$\prop[\comb{m}, \comb{m + 1}]{m}(\tree[m], \cdot)$, $\tree[m] \in \trsp[\comb{m}]$,
has a probability function, which we denote by the same symbol. Note
that the assumption \eqref{eq:proposal:special:form} implies that for
all $i \in \intvect{1}{N}$,
$$
	\combpart{1:m | m + 1}{i} = \combpart{m}{\ind{m + 1}{i}},
$$
and, consequently, by \eqref{eq:def:comb:kernel:path},
$$
	\combmeas{m + 1}(\combpart{m + 1}{i})
	= \combmeas{m} \combkernelpath{m}(\combpart{m + 1}{i})
	=  \combmeas{m}(\combpart{m}{\ind{m + 1}{i}})
	\combkernelpath{m}(\combpart{m}{\ind{m + 1}{i}}, \combpart{m + 1}{i}).
$$
Thus, the importance weight \eqref{eq:importance:weights}
simplifies according to
\begin{equation*}
\begin{split}
\lefteqn{\wgtfunc{m}(\epart{m}{{\ind{m + 1}{i}}}, \epart{m + 1}{i})}
\hspace{3mm} \\
&= \frac{\untarg[marg]{\combpart{m + 1}{i}}(\zpart{m + 1}{i})
\combmeas{m + 1}(\combpart{m + 1}{i}) \bk{m}(\epart{m}{\ind{m + 1}{i}}, \epart{m + 1}{i})}
{\untarg[marg]{\combpart{m}{\ind{m + 1}{i}}}(\zpart{m}{\ind{m + 1}{i}})
\combmeas{m}(\combpart{m}{\ind{m + 1}{i}})
\combkernelpath{m}(\combpart{m}{\ind{m + 1}{i}}, \combpart{m + 1}{i})
\prop[\combpart{m}{\ind{m + 1}{i}}, \combpart{m + 1}{i}]{m}
(\treepart{m}{\ind{m + 1}{i}}, \treepart{m + 1}{i})} \\
&=\frac{\untarg[marg]{\combpart{m + 1}{i}}(\zpart{m + 1}{i})
\bk{m}(\epart{m}{\ind{m + 1}{i}}, \epart{m + 1}{i})}
{\untarg[marg]{\combpart{m}{\ind{m + 1}{i}}}(\treepart{m}{\ind{m + 1}{i}})
\prop[\combpart{m}{\ind{m + 1}{i}}, \combpart{m + 1}{i}]{m}
(\treepart{m}{\ind{m + 1}{i}}, \treepart{m + 1}{i})}.
\end{split}
\end{equation*}

Further, we have the identity
\begin{align}
	\frac{\untarg[marg]{\combpart{m + 1}{i}}(\treepart{m + 1}{i})}
	{\untarg[marg]{\combpart{m}{\ind{m + 1}{i}}}(\zpart{m}{\ind{m + 1}{i}})}
	&= \frac{\mu \circ \trgr(\treepart{m}{\ind{m + 1}{i}})}{\mu \circ \trgr(\treepart{m + 1}{i})}\times\frac{
	\prod_{\cl \in \clset{\trgr(\treepart{m + 1}{i})}} \graphmap(\cl)
	\prod_{\cl \in \clset{\trgr(\treepart{m}{\ind{m + 1}{i}})}} \graphmap(\cl)^{-1}}
	{ \prod_{S \in \sepset{\trgr(\treepart{m + 1}{i})}}
	\graphmap(S) \prod_{S \in \sepset{\trgr(\treepart{m}{\ind{m + 1}{i}})}}
	\graphmap(S)^{-1} } \nonumber \\
	&= \frac{\mu \circ \trgr(\treepart{m}{\ind{m + 1}{i}})}{\mu \circ \trgr(\treepart{m + 1}{i})}\times\frac{
	\prod_{\cl \in \clset{\trgr(\treepart{m + 1}{i})} \triangle
	\clset{\trgr(\treepart{m}{\ind{m + 1}{i}})}}
	\graphmap(\cl)^{\exponent \langle \treepart{m + 1}{i} \rangle(\cl)}}
	{
	\prod_{S \in \sepset{\trgr(\treepart{m + 1}{i})} \triangle
	\sepset{\trgr(\treepart{m}{\ind{m + 1}{i}})}}
	\graphmap(S)^{\exponent \langle \treepart{m + 1}{i} \rangle(S)}},
	\label{eq:local:weight:update}
\end{align}
where $\triangle$ denotes symmetric difference and
$$
	\exponent \langle \treepart{m + 1}{i} \rangle (\cl)
	\eqdef 2 \1_{\clset{\trgr(\treepart{m + 1}{i})}}(\cl) - 1
$$
($\exponent \langle \treepart{m + 1}{i} \rangle (S)$ is defined similarly).

The computational burden involved in computing the first factor of (\ref{eq:local:weight:update}) can be substantially reduced by exploiting the factorisation presented in \citet[Theorem 7]{cta} and restated below.
Let $\graph[m+1]\in \graphsp[m+1]$ be a graph expanded from a graph $\graph[m]\in \graphsp[m]$ in the sense that $\graph[m+1][\{1,\dots,m\}]=\graph[m]$, then we can define the set $\mathcal S^\star\subset\sepset{\graph[m+1]}$ consisting of the separators created by the expansion.
The factorisation is then given as
\begin{align*}
    \frac{\mu (\graph[m])}{\mu(\graph[m+1])} =\frac{\prod \limits_{s \in \mathcal U_1}\nu_\mathsmaller{\graph[]}(s)}{\prod\limits_{s \in \mathcal U_2}{\nu_\mathsmaller{\graph[m+1]}(s)} },
\end{align*}
where
$\mathcal U_1 = \{s\in \sepset{\graph[m]} : \exists s^\prime \in \mathcal S^\star  \text{, such that } s\subset s^\prime\}$
and
$\mathcal U_2 = \{ s \in \sepset{\graph[m+1]} : \exists s^\prime \in \mathcal S^\star \text{, such that } s\subset s^\prime\}$
are the set of separators in $\graph[m]$ and $\graph[m+1]$ respectively, contained in some separator in $\mathcal S^\star$.
The function $\nu_\mathsmaller{\graph}(s)$ denotes the number of equivalent junction trees that can be obtained by randomizing a junction tree for the graph $\graph$ at the separator $s$.
For a more detailed presentation see~\citet{cta}.
The sets  $\clset{\trgr(\treepart{m + 1}{i})} \triangle \clset{\trgr(\treepart{m}{\ind{m + 1}{i}})}$ and $\sepset{\trgr(\treepart{m + 1}{i})} \triangle \sepset{\trgr(\treepart{m}{\ind{m + 1}{i}})}$ in the second factor might be composed by only a few cliques and separators, respectively, and computing the products in the numerator and denominator of \eqref{eq:local:weight:update} will in that case be an easy operation. 
This is the case for the CTA described in Section~\ref{sec:junction:tree:expanders} below.

In summary the identity \eqref{eq:local:weight:update} suggests that the first part of the importance weights may, in principle, be computed with a complexity that does not increase with the iteration index $m$ as long as the proposal kernel $\prop{m}^\ast$ only modifies and extends \emph{locally} the junction tree (and, consequently, the underlying graph).

The SMC update described above is summarised in Algorithm~\ref{alg:SMC:update}.
Here and in the following, we let $\disc(\{ a_\ell \}_{\ell = 1}^\N)$ denote the categorical probability distribution induced by a set $\{ a_\ell \}_{\ell = 1}^\N$ of positive
(possibly unnormalised) numbers; thus, writing $W \sim \disc(\{ a_\ell \}_{\ell = 1}^\N)$
means that the variable $W$ takes the value $\ell \in \intvect{1}{\N}$ with
probability $a_\ell / \sum_{\ell' = 1}^\N a_{\ell'}$. 

\bigskip
\begin{algorithm}[H] \label{alg:SMC:update}
     \KwData{$\{ (\epart{m}{i}, \wgt{m}{i}) \}_{i = 1}^\N$}
     \KwResult{$\{ (\epart{m + 1}{i}, \wgt{m + 1}{i}) \}_{i = 1}^\N$}
     \For{$i \gets 1, \ldots, \N$}{
         draw $\ind{m + 1}{i} \sim \disc( \{ \wgt{m}{\ell} \}_{\ell = 1}^\N)$\;
         draw $\combpart{m + 1}{i} \sim \combkernelpath{m}
         (\combpart{m}{\ind{m + 1}{i}}, \rmd \comb{m + 1})$\;
         draw $\treepart{m + 1}{i} \sim \prop[\combpart{m}{\ind{m + 1}{i}}, \combpart{m + 1}{i}]{m}
         (\treepart{m}{\ind{m + 1}{i}}, \rmd \tree[m + 1])$\;
         set $\epart{m + 1}{i} \gets (\combpart{m + 1}{i}, \treepart{m + 1}{i})$\;
         set $\displaystyle \wgt{m + 1}{i} \gets \frac{\untarg[marg]{\combpart{m + 1}{i}}
         (\zpart{m + 1}{i}) \bk{m}(\epart{m}{\ind{m + 1}{i}}, \epart{m + 1}{i})}
         {\untarg[marg]{\combpart{m}{\ind{m + 1}{i}}}(\zpart{m}{\ind{m + 1}{i}})
         \prop[\combpart{m}{\ind{m + 1}{i}}, \combpart{m + 1}{i}]{m}
         (\zpart{m}{\ind{m + 1}{i}}, \zpart{m + 1}{i})}$\;
     }
     \caption{SMC update}
\end{algorithm}
\bigskip

Naturally, the SMC algorithm is initialised by drawing i.i.d.~draws
$(\epart{1}{i})_{i = 1}^\N$ from some initial distribution
$\partinit \in \probmeas{\xfd{1}}$ and letting $\wgt{1}{i} =
\untarg{1}(\epart{1}{i}) / \partinit(\epart{1}{i})$ for all $i$, where
the density (with respect to $\rmd \x{1}$) of $\partinit$ is
denoted by the same symbol. In addition, letting $\partinit$
be of form
$$
	\partinit(\rmd \x{1}) = \combmeas{1}(\rmd \comb{1})
	\partinit[\comb{1}](\rmd \tree[1])
$$
yields the weights $\wgt{1}{i} = \untarg[marg]{\combpart{1}{i}}(\zpart{1}{i})
/ \partinit[\combpart{1}{i}](\zpart{1}{i})$.

As a by-product, Algorithm~\ref{alg:SMC:update} provides, for all
$m \in \intvect{1}{p}$ and $h \in \mf{\xfd{p}}$, unbiased estimators
$$
	\unpartarg{m} h \eqdef \frac{1}{\N^m} \left( \prod_{\ell = 1}^{m - 1}
	\wgtsum{\ell} \right) \sum_{i = 1}^\N \wgt{m}{i} h(\epart{m}{i})
$$
of $\untarg{m} h$. In particular,
$$
\unpartarg{p} \1_{\xsp{p}} = \frac{1}{\N^p} \prod_{\ell = 1}^p \wgtsum{\ell}
$$
is an unbiased estimator of the normalising constant $\untarg{p} \1_{\xsp{p}}
= \untrtarg \1_{\trsp[p]}$ of the distribution of interest. 

\subsubsection{Design of retrospective dynamics} 
\label{sub:design_of_retrospective_dynamics}
As we will see, the reversed kernels $\{\bk{m} \}_{m = 1}^{p - 1}$ will
typically be designed on the basis of the forward proposal kernels
$\{ \prop{m} \}_{m = 1}^{p - 1}$. It is clear that for all $m \in \intvect{1}{p - 1}$,
the constraint that $\uk{m}(\x{m}, \cdot) \ll \prop{m}(\x{m}, \cdot)$
for all $\x{m} \in \xsp{m}$ is satisfied as soon as the retrospective
kernel $\bk{m}$ is such that $\supp (\bk{m}(\x{m + 1}, \cdot))
\subseteq \supp (\prop{m}(\cdot, \x{m + 1}))$ for all $\x{m + 1}
\in \supp(\untarg{m + 1})$. Consequently, if for all $m \in \intvect{1}{p - 1}$,
\begin{equation} \label{eq:support:condition:proposal}
    \supp(\targ{1} \prop{1} \cdots \prop{m}) = \supp(\untarg{m + 1}),
\end{equation}
one may, e.g., construct each retrospective kernel $\bk{m}$
by identifying, for all $\x{m + 1} \in \supp(\untarg{m + 1})$, a
nonempty set $$\retrosupp{m}(\x{m + 1}) \subseteq
\supp (\prop{m}(\cdot, \x{m + 1})) \cap \supp(\untarg{m}),$$
and letting
\begin{align}
    \label{eq:support:condition:bk}
    \bk{m}(\x{m + 1}, \x{m}) \eqdef |\retrosupp{m}(\x{m + 1})|^{-1}
    \1_{\retrosupp{m}(\x{m + 1})}(\x{m})
    \quad (\x{m + 1} \in \supp(\untarg{m + 1})),
\end{align}
i.e., $\bk{m}(\x{m + 1}, \rmd \x{m})$ is the uniform distribution
over $\retrosupp{m}(\x{m + 1})$. The existence of such a nonempty
set is guaranteed by (\ref{eq:support:condition:proposal}). Indeed,
let $\x{m + 1} \in \supp(\untarg{m + 1})$; then, by
(\ref{eq:support:condition:proposal}),
$$
	\sum_{\x{m} \in \xsp{m}} \targ{1} \prop{1} \cdots
	\prop{m}(\x{m}) \prop{m}(\x{m}, \x{m + 1}) > 0,
$$
i.e., there exists at least one $x_m^\ast \in \xsp{m}$ such that
\(\targ{1} \prop{1} \cdots \prop{m}(x_m^\ast) > 0 \) and
$\prop{m}(x_m^\ast, \x{m + 1}) > 0$. Thus, again by
(\ref{eq:support:condition:proposal}), $x_m^\ast \in
\supp(\prop{m}(\cdot, \x{m + 1})) \cap \supp(\untarg{m})$,
which is hence nonempty. For $\x{m + 1} \notin
\supp(\untarg{m + 1})$, $\bk{m}(\x{m + 1}, \rmd \x{m})$
may be defined arbitrarily. As we will see next, the
property~\eqref{eq:support:condition:proposal} is satisfied
by the junction tree expanders used by us.

\subsubsection{The Christmas tree algorithm}\label{sec:junction:tree:expanders}
Following the presentation of~\citet{cta} we disregard, without loss of generality, the permutations of the nodes for the underlying graphs specified by $\trexfd{m}$.
This implies that we consider a fixed set of ordered nodes $\comb{m}=(1,\dots,m) \in \combsp{m}$ and by $\trsp[m]$ we mean $\trsp[\comb{m}]$.

As previously mentioned $\{\prop{m}\}_{m=1}^{\p-1}$ and $\{\bk{m}\}_{m=1}^{\p-1}$ will here correspond to the kernels induced by the CTA and its reversed version, respectively.
The CTA kernel takes as input a junction tree $\tree[m] \in \trsp[m]$ and expands it into a new junction tree $\tree[m+1]\in \trsp[m+1]$ according to $\prop{m}(\tree[m],\rmd\tree[m+1])$ by adding the internal node $m+1$ to the underlying graph $\trgr(\tree[m])$ in such a way that $\trgr(\tree[m+1])[\{1,\dots,m\}] = \trgr(\tree[m])$.
It requires two input parameters $(\alpha,\beta)\in (0,1)^2$ jointly controlling the sparsity of the produced underlying graph.
Specifically, at the initial step of the algorithm, a Bernoulli trial with parameter $\beta$ is performed in order to determine whether or not the internal node $m+1$ is being isolated in the underlying graph of the produced tree.
If $m+1$ is not isolated, a high value of the parameter $\alpha$ controls the number of cliques in $\tree[m+1]$ that will contain $m+1$.
In this sense, $\prop{m}(\tree[m],\rmd \tree[m+1])$ is a mixture distribution with weight parameter $\beta$.

\subsection{Particle Gibbs sampling}\label{sec:particle:Gibbs:sampling}

In the following, we discuss how to sample from the extended target
$\targ{1:p}$, having the distribution $\maintarg$ of interest
as a marginal distribution, using \emph{Markov chain Monte Carlo}
(McMC) methods. A \emph{particle Gibbs} (PG)
\emph{sampler} constructs, using SMC, a Markov
kernel $\PG{p}$ leaving $\targ{1:p}$ invariant.
Algorithmically, the more or less only difference between the PG
kernel and the standard SMC algorithm is that the PG kernel,
which is described in detail in Algorithm~\ref{alg:particle:Gibbs:kernel},
evolves the particle cloud \emph{conditionally} on a fixed reference
trajectory specified \emph{a priori}; this \emph{conditional SMC}
algorithm is constituted by Lines~1--16 in Algorithm~\ref{alg:particle:Gibbs:kernel}.
After having evolved, for $p$ time steps, the particles of the conditional
SMC algorithm, the PG kernel draws randomly a particle from the last
generation (Lines~17--19), traces the genealogical history of the selected
particle back to the first generation (Lines~20--22), and returns the
traced path (Line~23).

As as established in \cite[Proposition~8]{chopin:singh:2015}, $\PG{p}$
is $\targ{1:p}$-reversible and thus leaves $\targ{1:p}$ invariant. Interestingly,
reversibility holds true for any particle sample size $\N \in \nsetpos \setminus \{ 1 \}$.
Thus, on the basis of $\PG{p}$, the PG sampler generates
(after possible burn-in) a Markov chain $\{ \X{1:p}[\ell] \}_{\ell \in \nsetpos}$
 according to
$$
    \X{1:p}[1] \stackrel{\PG{p}}{\longrightarrow}
    \X{1:p}[2] \stackrel{\PG{p}}{\longrightarrow}
    \X{1:p}[3] \stackrel{\PG{p}}{\longrightarrow}
    \X{1:p}[4] \rightarrow \cdots
$$
and returns $\sum_{\ell = 1}^\Nmcmc h(\X{1:p}[\ell]) / \Nmcmc$ as
an estimate of $\targ{1:p} h$ for any $\targ{1:p}$-integrable objective
function $h \in \mf{\xfd{1:p}}$. Here $\Nmcmc \in \nsetpos$ denotes the
McMC sample size. In particular, in the case where the objective
function $h$ depends on the argument $\tree[p]$ only, we obtain the
estimator
\begin{align}
    \sum_{\ell = 1}^\Nmcmc h(\Z{p}[\ell]) / \Nmcmc
    \label{eq:mcmc_estimator}
\end{align}
 of $\maintarg h$,
where each $\Z{p}[\ell]$ variable is extracted, on Line~18, at iteration
$\ell - 1$. of Algorithm~\ref{alg:particle:Gibbs:kernel}.
\begin{algorithm}[ht]\label{alg:particle:Gibbs:kernel}
     \KwData{a reference trajectory $\x{1:p} \in \xsp{1:p}$}
     \KwResult{a draw $\X{1:p}$ from $\PG{p}(\x{1:p}, \rmd \x{1:p}')$}
     \For{$i \gets 1, \ldots, \N - 1$}{
         draw $\combpart{1}{i} \sim \combmeas{1}(\rmd \comb{1})$\;
         draw $\zpart{1}{i} \sim \partinit[\combpart{1}{i}](\rmd \tree[1])$\;
         set $\epart{1}{i} \gets (\combpart{1}{i}, \zpart{1}{i})$\;
     }
     set $\epart{1}{\N} \gets \x{1}$\;
     \For{$i \gets 1, \ldots, \N$}{
         set $\wgt{1}{i} \gets \untarg[marg]{\combpart{1}{i}}(\zpart{1}{i})
         / \partinit[\combpart{1}{i}](\zpart{1}{i})$\;
     }
     \For{$m \gets 1, \ldots, p - 1$}{
         \For{$i \gets 1, \ldots, \N - 1$}{
             draw $\ind{m + 1}{i} \sim \disc( \{ \wgt{m}{\ell} \}_{\ell = 1}^\N)$\;
             draw $\combpart{m + 1}{i} \sim \combkernelpath{m}
             (\combpart{m}{\ind{m + 1}{i}}, \rmd \comb{m + 1})$\;
             draw $\zpart{m + 1}{i} \sim \prop[\combpart{m}{\ind{m + 1}{i}}, \combpart{m + 1}{i}]{m}
             (\zpart{m}{\ind{m + 1}{i}}, \rmd \tree[m + 1])$\;
             set $\epart{m + 1}{i} \gets (\combpart{m + 1}{i}, \zpart{m + 1}{i})$\;
         }
         set $\epart{m + 1}{\N} \gets \x{m + 1}$\;
         \For{$i \gets 1, \ldots, \N$}{
             set $\displaystyle \wgt{m + 1}{i} \gets \frac{\untarg[marg]{\combpart{m + 1}{i}}
             (\zpart{m + 1}{i}) \bk{m}(\epart{m}{\ind{m + 1}{i}}, \epart{m + 1}{i})}
             {\untarg[marg]{\combpart{m}{\ind{m + 1}{i}}}(\zpart{m}{\ind{m + 1}{i}})
             \prop[\combpart{m}{\ind{m + 1}{i}}, \combpart{m + 1}{i}]{m}
             (\zpart{m}{\ind{m + 1}{i}}, \zpart{m + 1}{i})}$\;
         }
     }
     draw $\gen{p} \sim \disc( \{ \wgt{p}{\ell} \}_{\ell = 1}^\N )$\;
     set $\Z{p} \gets \zpart{p}{\gen{p}}$\;
     set $\X{p} \gets (\intvect{1}{p}, \Z{p})$\;
     \For{$m \gets p - 1, \ldots, 1$}{
         set $\gen{m} \gets \ind{m}{\gen{m + 1}}$\;
         set $\X{m} \gets \epart{m}{\gen{m + 1}}$\;
     }
     set $\X{1:p} \gets (\X{1}, \ldots, \X{p})$\;
     \Return{$\X{1:p}$}
     \caption{One transition of PG.}
\end{algorithm}

\subsection{Particle Gibbs with systematic refreshment}\label{sec:particle:Gibbs:refreshment}

For the graph-oriented applications of interest in the present paper,
the naive implementation of the PG sampler will suffer from bad mixing,
even though the distribution of interest, $\maintarg$, is defined only on
the marginal space $\trsp[p]$. Thus, we will modify slightly
the standard PG sampler by inserting an intermediate \emph{refreshment step}
in between the PG iterations. More specifically, define
$$
    \G{p}(\x{1:p}, \rmd \x{1:p}')
    = \delta_{\x{p}}(\rmd \x{p}') \prod_{m = 1}^{p - 1}
    \bk{m}(\x{m + 1}', \rmd \x{m}') \quad (\x{1:p} \in \xsp{1:p}).
$$
Given $\x{1:p}$, drawing $\X{1:p} \sim \G{p}(\x{1:p}, \rmd \x{1:p}')$
amounts to setting, deterministically, $\X{p} = \x{p}$ and simulating
$\X{1:p - 1}$ according to the Markovian retrospective dynamics
induced by the kernels $\{ \bk{m} \}_{m = 1}^{p - 1}$.  Note that
each distribution $\G{p}(\x{1:p}, \rmd \x{1:p}')$ depends exclusively
on $\x{p}$. Describing a standard Gibbs substep for sampling from
$\targ{1:p}$, $\G{p}$ is $\targ{1:p}$-reversible; see, e.g.,
\cite[Proposition~6.2.14]{cappe:moulines:ryden:2005}.
Consequently, also the product kernel $\PG{p} \G{p}$ is
$\targ{1:p}$-invariant. Unlike standard PG, the McMC sampling
scheme that we propose, which is summarised in
Algorithm~\ref{alg:particle:Gibbs:refreshment}, generates
(after possible burn-in) a Markov chain
$\{ \X{1:p}[\ell] \}_{\ell \in \nsetpos}$ according to
$$
    \X{1:p}[1] \stackrel{\PG{p} \G{p}}{\longrightarrow}
    \X{1:p}[2] \stackrel{\PG{p} \G{p}}{\longrightarrow}
    \X{1:p}[3] \stackrel{\PG{p} \G{p}}{\longrightarrow}
    \X{1:p}[4] \rightarrow \cdots
$$
and returns
$$
    \targMCMC{1:p}{\Nmcmc} h \eqdef \frac{1}{\Nmcmc}
    \sum_{\ell = 1}^\Nmcmc h(\X{1:p}[\ell])
$$
as an estimator of $\targ{1:p} h$ for any $\targ{1:p}$-integrable
function $h$. In addition, as previously, in the case where the
objective functions $h$ depends on the argument $\tree[p]$ only,
we obtain the estimator
\begin{align}
    \maintargMCMC{\Nmcmc} h \eqdef \frac{1}{\Nmcmc}
    \sum_{\ell = 1}^\Nmcmc h(\Z{p}[\ell])
    \label{eq:mcmc_estimator2}
\end{align}
of $\maintarg h$, where each $\Z{p}[\ell]$ variable is extracted,
on Line~2, at iteration $\ell - 1$ of Algorithm~\ref{alg:particle:Gibbs:refreshment}.
\bigskip

\begin{algorithm}[H] \label{alg:particle:Gibbs:refreshment}
     \KwData{a reference trajectory $\x{1:p} \in \xsp{1:p}$}
     \KwResult{a draw $\X{1:p}$ from $\PG{p}\G{p}(\x{1:p}, \rmd \x{1:p}')$}
     draw, using Algorithm~\ref{alg:particle:Gibbs:kernel},
     $\X{1:p}' \sim \PG{p}(\x{1:p}, \rmd \x{1:p}')$\;
     set $\X{p} = (\intvect{1}{p}, \Z{p}) \gets \X{p}'$\;
     \For{$m \gets p - 1, \ldots, 1$}{
         draw $\X{m} \sim \bk{m}(\X{m + 1}, \rmd \x{m})$\;
     }
     set $\X{1:p} \gets (\X{1}, \ldots, \X{p})$\;
     \Return{$\X{1:p}$}
     \caption{One transition
     of PG with systematic refreshment.}
\end{algorithm}

\bigskip

\subsubsection{Particle Gibbs with systematic refreshment vs.~standard particle Gibbs}
\label{sec:theoretical:results}
As established by the following theorem, the systematic refreshment
step improves indeed the mixing of the algorithm. For any functions
$g$ and $h$ in $\Lp{2}(\targ{1:p})$ we define the \emph{scalar product}
$\langle g, h \rangle \eqdef \targ{1:p}(g h)$. Moreover, for all $\targ{1:p}$-invariant
Markov kernels $\kernel{M}$ on $(\xsp{1:p}, \xfd{1:p})$ and functions
$h \in \Lp{2}(\targ{1:p})$ such that
\begin{equation} \label{eq:summation:condition}
	\sum_{\ell = 1}^\infty | \langle h, \kernel{M}^\ell h \rangle | < \infty,
\end{equation}
we define the asymptotic variance
\begin{equation} \label{eq:asymptotic:variance}
	v(h, \kernel{M}) \eqdef \lim_{\Nmcmc \rightarrow \infty}
	\frac{1}{\Nmcmc} \operatorname{Var}
	\left( \sum_{\ell = 1}^\Nmcmc h(\X{1:p}[\ell]) \right),
\end{equation}
where $\{ \X{1:p}[\ell] \}_{\ell = 1}^\infty$ is a Markov chain with initial
distribution $\targ{1:p}$ and transition kernel $\kernel{M}$. (The assumption
\eqref{eq:summation:condition} can be shown to imply the existence of
the limit \eqref{eq:asymptotic:variance}). In the case where the latter
Markov chain satisfies a central limit theorem for the objective function
$h$, the corresponding asymptotic variance is given by
\eqref{eq:asymptotic:variance}. As established by the following result,
whose proof relies on asymptotic theory for inhomogeneous
Markov chains developed in \cite{maire:douc:olsson:2014},
the improved mixing implied by systematic refreshment of the
trajectories implies a decrease of asymptotic variance w.r.t. standard PG.

\begin{theorem}
\label{thm:asymvar}
	For all $\N \in \nsetpos$ and all functions $h^\ast \in \Lp{2}(\maintarg)$
	such that both $\PG{p}\G{p}$ and $\PG{p}$ satisfy the summation condition
	\eqref{eq:summation:condition} with $h \eqdef \1_{\xsp{1:p - 1}} \varotimes h^\ast$,
	it holds that
	$$
		v(h, \PG{p}\G{p}) \leq v(h, \PG{p}).
	$$
\end{theorem}
\begin{proof}
See \ref{proof:asymvar}.
\end{proof}

\section{Application to decomposable graphical models}\label{sec:application}
In this section we show in more detail how distributions of form \eqref{eq:graph:target}
appear in Bayesian analysis of \emph{graphical models}.
To rigorously describe the setting, we shall need some further notations.
Let $\{ (\ysp[m], \yfd[m]) \}_{m = 1}^p$, $p \in \nsetpos$, be a sequence
of measurable spaces and define $\ysp \eqdef \prod_{m = 1}^p \ysp[m]$
and $\yfd \eqdef \bigotimes_{m = 1}^p \yfd[m]$. Let
$\Y = (\Y[1], \ldots, \Y[p]) : \Omega \rightarrow \ysp$ be a random element.
We consider a fully dominated model where the distribution of $\Y$ has a
density $\dens$ on $\ysp$ with respect to some reference measure
$\refm \eqdef  \bigotimes_{m = 1}^p \refm_m$ on $(\ysp, \yfd)$, where each
$\refm_m$ belongs to  $\meas{\yfd[m]}$. For some subset $\{a_1, \ldots, a_m \}
\subseteq \intvect{1}{p}$ with $a_1 \leq \ldots \leq a_m$, we let $\Y[A] \eqdef
(\Y[a_1], \ldots, \Y[a_m])$ and define $\ysp[A] = \prod_{\ell = 1}^m \ysp[a_\ell]$
and $\yfd[A] = \bigotimes_{\ell = 1}^m \yfd[a_\ell]$. By slight abuse of notation,
we denote by $\dens(\y[A])$ the marginal density of $\Y[A]$ with respect to
$\nu_A \eqdef \bigotimes_{\ell = 1}^m \refm_{a_\ell}$. For disjoint subsets
$A$, $B$, and $S$ of $\intvect{1}{p}$, we say, following \cite{lauritzen1996},
that $\Y[A]$ and $\Y[B]$ are \emph{conditionally independent given} $\Y[S]$,
denoted $\Y[A] \perp \Y[B] \cond \Y[S]$, if it holds that
\begin{equation} \label{eq:cond:indep}
	f(\y[A \cup B] \cond \y[S]) = \dens(\y[A] \cond \y[S]) \dens(\y[B] \cond \y[S]),
	\quad \mbox{for all } \y[A] \in \Y[A], \ \y[B] \in \Y[B], \ \y[S] \in \Y[S],
\end{equation}
where the conditional densities are defined as $\dens(\y[A] \cond \y[S])
\eqdef \dens(\y[A \cup S]) / f(\y[S])$. The distribution of $\Y$ is said to
be \emph{globally Markov} w.r.t. the undirected graph $G = (V, E)$, with
$\graphnodeset = \intvect{1}{p}$ and $E \subseteq V \times V$, if for disjoint
subsets $A$, $B$, and $S$ of $V$ it holds that
$$
	A \perp_G B \cond S \Rightarrow \Y[A] \perp \Y[B] \mid \Y[S].
$$
We call the distribution governed by $\dens$ a \emph{decomposable model}
if it is globally Markov w.r.t. a decomposable graph. Then, by repeated use of
\eqref{eq:cond:indep}, it is easily shown that the density of a decomposable model
satisfies the \CSF-type identity
\begin{equation} \label{eq:DM:fixed:G}
	\dens(\y) = \frac{\prod_{\cl \in \clset{\graph}} \dens(\y[\cl])}
	{\prod_{S \in \sepset{\graph}} \dens(\y[S])},
\end{equation}
where we, in order to justify the notation $\y[\cl]$ and $\y[S]$, by slight abuse
of notation identify the cliques $\cl$ and the separators $S$ (which are complete
graphs) with the corresponding subsets of $\graphnodeset$.

In the following we consider the dependence structure $\graph$ as \emph{unknown},
and take a Bayesian approach to the estimation of the same on the basis of a
given, fixed data record $\y \in \ysp$. For this purpose, we assign a prior distribution
\begin{equation} \label{eq:prior:CSF:form}
	\prior(\rmd \graph) \eqdef \frac{\priorfun^\star(\rmd \graph)}
	{\priorfun^\star \1_{\graphsp}}
\end{equation}
in $\probmeas{\graphfd}$ to $\graph$, where
$$
	\priorfun^\star(\rmd \graph)
	\eqdef \priorfun |_{\graphsp}(\graph) \, \cm{\rmd \graph}
$$
and $\priorfun : \allgr \rightarrow \rsetpos$ is a function satisfying the
\CSF~in Definition~\ref{def:CFS}. For instance, in the completely uninformative
case, $\priorfun \equiv 1$; in the presence of prior information concerning the
maximal clique size of the underlying graph, one may let $\priorfun(\graph)
= \1 {\{\vee_{\cl \in \clset{\graph}} |\cl| \leq M\}}$ for some $M \in \nset$ controlling
the sizes of the cliques. In both cases, the \CSF~is immediately checked, see e.g \cite{bornn2011bayesian}.
We let the same symbol $\prior$ denote also the corresponding probability function.

In this Bayesian setting, focus is set on the \emph{posterior} distribution
$\graphtarg$ of the graph $\graph$ given the available data $\y$, which is,
by Bayes' formula, obtained via \eqref{eq:graph:target} with
$\ungraphtarg$ induced by
$$
	\graphmap(\graph) = \frac{\prod_{\cl \in \clset{\graph}} \dens(\y[\cl]) \priorfun(\graph[\cl])}
	{\prod_{S \in \sepset{\graph}} \dens(\y[S]) \priorfun(\graph[S])}, \quad \graph \in \allgr.
$$
The problem of computing the posterior may consequently be perfectly
cast into the setting of Section~\ref{sec:non-temp:FK:flows}.

The model will in general comprise additional unknown parameters
collected in a vector $\param$,  which is assumed to belong to some
measurable parameter space $(\paramsp_{\graph}, \paramfd_{\graph})$
depending on the graph $\graph$. We add $\param$ and $\graph$ to the
notation of the likelihood, which is assumed to be of form
\begin{equation} \label{eq:likelihood:full:model}
	\dens(\y \cond \param, \graph)
	= \frac{\prod_{\cl \in \clset{\graph}} \dens(\y[\cl] \cond \param[\cl])}
	{\prod_{S \in \sepset{\graph}} \dens(\y[S]\cond \param[S])}.
\end{equation}
Our Bayesian approach calls for a prior also on
$\param = \{ \param[\cl] , \param[S] : \cl \in \clset{\graph}, S \in \sepset{\graph} \}$,
and we will always assume that this is \emph{hyper Markov} w.r.t. the underlying
graph $\graph$. More specifically, we assume that the conditional distribution of
$\param$ given $\graph$ has a density w.r.t. some reference measure, denoted
$\rmd \param$ for simplicity, on $(\paramsp, \paramfd)$. This density is assumed
to be of form
\begin{equation} \label{eq:hyper:Markov:form}
	\prior(\param \cond \graph \hypcond \hyperparamletter)
	= \frac{\prod_{\cl \in \clset{\graph}} \prior(\param[\cl] \hypcond \hyperparam{\cl})}
	{\prod_{S \in \sepset{\graph}} \prior(\hyperparam{S} \hypcond \vartheta_S)},
\end{equation}
where $\hyperparamletter = \{ \hyperparam{\cl} , \hyperparam{S} : \cl \in \clset{\graph},
S \in \sepset{\graph} \}$ is a set of hyperparameters and each factor
$\prior(\param[\cl] \hypcond \hyperparam{\cl})$ (and $\prior(\hyperparam{S} \hypcond \vartheta_S)$)
is a probability density $\prior(\param[\cl] \hypcond \hyperparam{\cl})
= z(\param[\cl] \hypcond \hyperparam{\cl}) / I(\hyperparam{\cl})$ with
$I(\hyperparam{\cl}) = \int z(\param[\cl] \hypcond \hyperparam{\cl}) \, \rmd \param[\cl]$
being a normalising constant.

In the case where each $\prior(\param[\cl] \hypcond \hyperparam{\cl})$
is a \emph{conjugate prior} for the corresponding likelihood factor
$f(\y[\cl] \cond \param[\cl])$ it holds that
\begin{equation} \label{eq:conjugacy}
	\dens(\y[\cl] \cond \param[\cl]) z(\param[\cl] ; \hyperparam{\cl})
	= c^{|\cl|} z(\param[\cl] ; \hyperparam[{\y[\cl]}]{\cl}),
\end{equation}
for some updated hyperparameter $\hyperparam[{\y[\cl]}]{\cl}$ and
some constant $c > 0$. If the normalising constants $I(\hyperparam{\cl})$
are tractable, we may marginalise out the parameter and consider directly
the posterior of $\graph$ given data $\y$. Indeed, since for all
hyperparameters,
$$
	\int \frac{\prod_{\cl \in \clset{\graph}} z(\param[\cl] \hypcond \hyperparam{\cl})}
	{\prod_{S \in \sepset{\graph}} z(\hyperparam{S} \hypcond \vartheta_S)}
	\, \rmd \theta = \frac{\prod_{\cl \in \clset{\graph}} I(\hyperparam{\cl})}
	{\prod_{S \in \sepset{\graph}} I(\hyperparam{S})},
$$
the marginalised likelihood is obtained as
\[
	\begin{split}
		\dens(\y \cond \graph)
		&= \int \frac{\prod_{\cl \in \clset{\graph}} \dens(\y[\cl] \cond \param[\cl])
		\prior(\param[\cl] \hypcond \hyperparam{\cl})}
		{\prod_{S \in \sepset{\graph}} \dens(\y[S] \cond \param[S])
		\prior(\param[S] \hypcond \hyperparam{S})} \, \rmd \param \\
		&= c^p \int \frac{\prod_{\cl \in \clset{\graph}}
		z(\param[\cl] \hypcond \hyperparam[{\y[\cl]}]{\cl}) / I(\hyperparam{\cl})}
		{\prod_{S \in \sepset{\graph}} z(\param[S] \hypcond \hyperparam[{\y[S]}]{S})
		/ I(\hyperparam{S})} \, \rmd \param \\
		&= c^p \frac{\prod_{\cl \in \clset{\graph}} I(\hyperparam[{\y[\cl]}]{\cl})
		/ I(\hyperparam{\cl})}{\prod_{S \in \sepset{\graph}}  I(\hyperparam[{\y[S]}]{S})
		/ I(\hyperparam{S})},
	\end{split}
\]
Thus, by Bayes' formula, the marginal posterior $\graphtarg$ of $\graph$
given the available data $\y$ can be expressed by \eqref{eq:graph:target}
with $\ungraphtarg$ induced by
\begin{equation} \label{eq:gamma:conjugate:priors}
	\graphmap(\graph) = \frac{\prod_{\cl \in \clset{\graph}}
	\priorfun(\graph[\cl]) I(\hyperparam[{\y[\cl]}]{\cl}) / I(\hyperparam{\cl})}
	{\prod_{S \in \sepset{\graph}} \priorfun(\graph[S]) I(\hyperparam[{\y[S]}]{S})
	/ I(\hyperparam{S})}, \quad \graph \in \allgr.
\end{equation}

\begin{example}[discrete log-linear models]\label{ex:multinomial:graphical:models}
    Let $V$ be a set of $p$ criteria defining a contingency table.
    Without loss of generality, we let $\graphnodeset = \intvect{1}{p}$ and
    denote the table by $\ct = \ct[1] \times \cdots \times \ct[p]$, where each
    $\ct[m]$ is a finite set. An element $i \in \ct$ is referred to as a \emph{cell}.
    In this setting, $\I = (\I_1, \ldots, \I_p)$ is a discrete-valued random
    vector whose distribution $\param$ is assumed to be globally Markov
    w.r.t. some decomposable graph $\graph = (\graphnodeset, E)$ with
    $E \subseteq \graphnodeset \times \graphnodeset$, i.e.,
    \begin{align}
        \paramd[i] = \prob \left( \I = i \right) = \frac{\prod_{\cl \in \clset{\graph}}
        \paramd[i_\cl]}{\prod_{S \in \sepset{\graph}} \paramd[i_S]}, \quad i \in \ct.
        \label{eq:hyper_markov_param}
    \end{align}
    The vector $\I$ may, e.g., characterise a randomly selected individual
    w.r.t. the table $\ct$. Given $\graph$, the parameter space of the model
    is determined by the clique and separator marginal probability tables
    $\paramd[i_\cl]$ and $\paramd[i_S]$; more specifically,
    \begin{equation*}
        \Theta_{\graph} = \left \{ \paramd(i_\cl) \in (0,1), \paramd(i_S) \in (0,1)
        : i \in \I, \cl \in \clset{\graph}, S \in \sepset{\graph},
         \sum_{i \in \ct} \paramd[i] = 1 \right \}.
    \end{equation*}
    Let $Y$ be a collection of $n \in \nset$ i.i.d. observations from the model;
    e.g., $Y$ is an $n \times p$ matrix where each row corresponds to an
    observation of $\I$. Then also $Y$ forms a DGM
    with state space $\ysp = \ct[1]^n \times \cdots \times \ct[p]^n$ and
    probability function $\dens(y \cond \param, \graph)$ given by
    \eqref{eq:likelihood:full:model} with
    $$
    \dens(y_\cl \cond \param[\cl]) = \prod_{i_\cl \in \ct[\cl]} \paramd[i_\cl]^{n(i_\cl)}
    $$
    (and similarly for $\dens(y_S \cond \param[S])$), where $\ct[\cl] = \prod_{m \in \cl} \ct[m]$,
    $\param[\cl] \eqdef \{\param(i_\cl) \}_{i_\cl \in \ct[\cl]}$, and $n(i_\cl)$ counts the
    number of elements of $y_\cl$ belonging to the marginal cell $i_\cl$.

    The problem of estimating the dependence structure $\graph$ is complicated
    further by the fact that also the probabilities $\param$ are unknown in general.
    When assigning a prior $\prior(\param \cond \graph \hypcond \hyperparamletter)$
    to the latter conditionally on the former, we follow \citet{dawid1993} and let the prior
    $\prior(\param[\cl] \hypcond \hyperparam{\cl})$ of each  $\param[\cl]$ be a standard
    \emph{Dirichlet distribution}, $\Dir(\hyperparam{\cl})$, where $\hyperparam{\cl} =
    \{\hyperparam{\cl}(i_\cl) \}_{i_\cl \in \ct[\cl]}$ are hyper parameters often referred to
    as \emph{pseudo counts}. Under the assumption that the collection
    $\{\prior(\param[\cl] \hypcond \hyperparam{\cl}) \}_{\cl \in \clset{\graph}}$ is
    pairwise \emph{hyper consistent} in the sense that for all $(\cl, \cl') \in \clset{\graph}^2$
    such that $\cl \cap \cl' \neq \varnothing$, $\prior(\param[\cl] \hypcond \hyperparam{\cl})$
    and $\prior(\param[\cl'] \hypcond \hyperparam{\cl'})$ induce the same law on
    $\param[\cl \cap \cl']$, which in this case is implied by the condition
    \begin{align*}
        \hyperparam{\cl}(i_{\cl \cap \cl'})
        &\eqdef \sum_{j_{\cl} \in \ct[\cl] : j_{\cl \cap \cl'}
        = i_{\cl \cap \cl'}} \hyperparam{\cl}(j_\cl) \\
        &= \sum_{j_{\cl'} \in \ct[\cl'] : j_{\cl \cap \cl'}
        = i_{\cl \cap \cl'}} \hyperparam{\cl'}(j_{\cl'})
        = \hyperparam{\cl'}(i_{\cl \cap \cl'}),
    \end{align*}
    \cite[Theorem~3.9]{dawid1993} implies the existence of a unique \emph{hyper Dirichlet}
    law of the form  \eqref{eq:hyper:Markov:form}. Thus, $z(\param[\cl] \hypcond \hyperparam{\cl})
    = \prod_{i_\cl \in \ct[\cl]} {\paramd[i_\cl]}^{\hyperparam{\cl}(i_\cl)}$, $I(\hyperparam{\cl})
    = B(\hyperparam{\cl}) \eqdef \prod_{i_\cl \in \ct[\cl]} \Gamma(\hyperparam{\cl}(i_\cl))
    / \Gamma(\sum_{i_\cl \in \ct[\cl]} \hyperparam{\cl}(i_\cl))$ (the beta function),
    and the conjugacy \eqref{eq:conjugacy} holds with $c = 1$ and $\hyperparam[{\y[\cl]}]{\cl}
    = \{ \hyperparam[i_\cl]{\cl}(\y[\cl]) \}_{i_\cl \in \ct[\cl]}$, where $\hyperparam[i_\cl]{\cl}(\y[\cl])
    = \hyperparam{\cl}(i_\cl) + n(i_\cl)$. Then, putting a prior of form \eqref{eq:prior:CSF:form}
    on the graph,  \eqref{eq:gamma:conjugate:priors} implies that the marginal posterior of
    $\graph$ given data $\y$ is obtained through \eqref{eq:graph:target} with $\ungraphtarg$
    induced by
    $$
        \graphmap(\graph) = \frac{\prod_{\cl \in \clset{\graph}}
        \priorfun(\cl) B(\hyperparam[{\y[\cl]}]{\cl}) / B(\hyperparam{\cl})}
        {\prod_{S \in \sepset{\graph}} \priorfun(S) B(\hyperparam[{\y[S]}]{S})
        / B(\hyperparam{S})}, \quad \graph \in \allgr.
    $$
\end{example}

\begin{example}[Gaussian graphical models]
    \label{ex:Gaussian:graphical:graph:models}
    A $p$-dimensional Gaussian random vector forms a
    \emph{Gaussian graphical model} if it is globally Markov w.r.t.
    some graph $\graph = (\graphnodeset, E)$ with $\graphnodeset
    = \intvect{1}{p}$ and $E \subseteq \graphnodeset \times \graphnodeset$.
    In the following we assume that the model has zero mean (for simplicity)
    and is, given $\graph$, parameterised by its \emph{precision}
    (inverse covariance) \emph{matrix} belonging to the set
    $$
        \Theta_{\graph} = \{\param \in \precmatset{p} :
        \param[ij] = 0 \mbox{ for all } (i,j) \notin E \},
    $$
    where $ \precmatset{p}$ denotes the space of $p \times p$ positive definite matrices.
    It is well known that in this model, a zero in the precision matrix, $\param[ij] = 0$,
    is equivalent to conditional independence of the $i$th and $j$th variables given
    the rest of the variables, see \citet{10.2307/2241271}. In addition, when $\graph$ is
    decomposable, a model with $\theta \in \Theta_{\graph}$ is globally Markov w.r.t.
    $\graph$. In the following, for any matrix $p \times p$ matrix $M$ and
    $A \subseteq \intvect{1}{p}$, denote by $M_A$ the $|A|\times |A|$ matrix
    obtained by extracting the elements $(M_{ij})_{(i, j) \in A^2}$ from $M$.
    Suppose that $\graph$ is decomposable and that are we have access to
    $n$ independent observations from the model. The observations are stored
    in an $n \times p$ data matrix $\Y$, whose likelihood $\dens(\y \cond \param, \graph)$
    is, as a consequence of the global Markov property,
    given by \eqref{eq:likelihood:full:model} with
    $$
        \dens(\y[\cl] \cond \param[\cl])
        = \frac{1}{(2 \pi)^{|\cl|}}|\param[\cl]|^{n / 2}
        \exp\left( - \mathrm{tr}(\param[\cl] s_{\cl}) / 2 \right),
    $$
    (and similarly for $\dens(\y[S] \cond \param[S])$) where $s = y^\trans y$,
    $|\cl|$ is the cardinality of $\cl$, and $|\param[\cl]|$ is the determinant
    of $\param[\cl]$.

    For Bayesian inference on $\param$, we follow~\citet{dawid1993} and furnish, given $\graph$, $\param$ with a
    \emph{hyper Wishart} prior $\prior(\param \cond \graph)$ of form~\eqref{eq:hyper:Markov:form}, with each
    $\prior(\param[\cl] \hypcond \hyperparam{\cl})$
    being proportional to
    $$
        z(\param[\cl] \hypcond \hyperparam{\cl})
        = |\param[\cl]|^{\beta_\cl}
        \exp \left(- \mathrm{tr}(\param[\cl] \scalmat[\cl]) / 2\right),
    $$
    where $\beta_\cl \eqdef (\df + |\cl| - 1) / 2$, $\hyperparam{\cl}
    = (\df, \scalmat[\cl])$ with $\scalmat \in \precmatset{p}$ being a
    scale matrix and $\df > \vee_{\cl \in \clset{\graph}} |\cl| - 1$
    the number of degrees of freedom, and normalising constant
    $$
        I(\hyperparam{\cl})
        = 2^{\df |\cl| / 2} \frac{\Gamma_{|\cl|}(\beta_\cl)}
        {|\scalmat[\cl]|^{\beta_\cl}},
    $$
    where $\Gamma_p$ denotes the multivariate gamma function.
    Since all hyperparameters $\hyperparam{\cl}$ are extracted from
    the \emph{same} scale matrix $\scalmat$, the collection of priors
    $\{ \prior(\param[\cl] \hypcond \hyperparam{\cl}) \}_{\cl \in \clset{\graph}}$
    is automatically pairwise hyper consistent, and the existence of the
    (unique) hyper Wishart prior is guaranteed by \citet[Theorem~3.9]{dawid1993}.
    As
    $$
        \dens(\y[\cl] \cond \param[\cl]) z(\param[\cl] \hypcond \hyperparam{\cl})
        = \frac{1}{(2 \pi)^{|\cl|}}|\param[\cl]|^{\alpha_\cl}
        \exp \left( - \mathrm{tr}\{\param[\cl] (s_{\cl} + \scalmat[\cl]) \} / 2 \right),
    $$
    where $\alpha_\cl \eqdef (\df + n + |\cl| - 1) / 2$, we conclude that the
    conjugacy condition \eqref{eq:conjugacy} holds for $c = 1/(2 \pi)$ and
    $\hyperparam[{\y[\cl]}]{\cl} = (\df[\cl]', \scalmat[\cl]')$ with $\df[\cl]' = \df + n$
    and $\scalmat[\cl]' = s_{\cl} + \scalmat[\cl]$ (and similarly for factors
    corresponding to separators). Consequently, assigning also a prior of
    form \eqref{eq:prior:CSF:form} to the graph, \eqref{eq:gamma:conjugate:priors}
    implies that the marginal posterior of $\graph$ given data $\y$ is, in this case,
    obtained through \eqref{eq:graph:target} with $\ungraphtarg$ induced by
    $$
        \graphmap(\graph) = \frac{\prod_{\cl \in \clset{\graph}} \priorfun(\cl) \rho(\cl)}
        {\prod_{S \in \sepset{\graph}} \priorfun(S) \rho(S)}, \quad \graph \in \allgr,
    $$
    with
    $$
        \rho(\cl) \eqdef
        \frac{|\scalmat[\cl]|^{\alpha_\cl}}{|\scalmat[\cl] + s_\cl|^{\beta_\cl}}
        \frac{\Gamma_{|\cl|}(\alpha_\cl)}{\Gamma_{|\cl|}(\beta_\cl)},
    $$
    and $\rho(S)$ defined analogously.
\end{example}

\section{Numerical study}\label{sec:numerics}
\newcommand{\subtreecont}{\alpha}
\newcommand{\subtreenonempty}{\beta}

In this section we investigate numerically the performance of the suggested PG algorithm for three example datasets.
The first example treats the classical Czech autoworkers dataset found in e.g. \cite{10.2307/2336086}.
The second one considers simulated data generated from the discrete $\p=15$ nodes structure, introduced in \cite{2005}.
The third example investigates a continuous dataset simulated from a Gaussian DGM of dimensionality $\p=50$ with a time-varying dependence structure.

The proposal and backward kernels $\{\prop{m}\}_{m=1}^{\p-1}$ and $\{\bk{m}\}_{m=1}^{\p-1}$ are given by the CTA and its reversed version, respectively, provided in Section 3 and 4 of the companion paper \cite{cta}.
The transition kernels $\{\combkernel{m}\}_{m=1}^{\p-1}$ introduced in \eqref{eq:def:comb:kernel:path} are defined by selecting $s^*$ uniformly at random from the set $\{s \in \intvect{1}{p} : \min_{s' \in \comb{m}} |s - s'| \leq \delta\}$ as suggested in Section \ref{sec:non-temp:FK:flows}.

We assign the uniform prior for the graph structure in each of the examples.
The estimated graph posteriors are summarized in terms of marginal edge distributions presented as heatmaps, where the  probability of an edge $(a,b)$ is estimated according to (\ref{eq:mcmc_estimator2}) by letting $h(\Z{})=\1_{(a,b) \in E}$, where $E$ here denotes the edge set for $\trgr(\Z{})$. 

We study the number of edges in the graph (\emph{graph size}) in order to evaluate the mixing properties. 
Since in many practical situations the aim is to select one particular model that best represents the underlying dependence structure, we also present the maximum aposteriori (MAP) graph for each of the examples.

All the experiments were performed on the Tegnér cluster at PDC having one intel $2\times 12$ core Intel E5-2690v3 Haswell processor per node.
The Python program used to generate the examples is part of the \emph{trilearn} library available at \url{https://github.com/felixleopoldo/trilearn}.

\subsection{Czech autoworkers data}

This dataset, previously analyzed many times in the literature comprises 1841 men cross-classified with respect to six potential risk factors for coronary thrombosis randomly selected from a population of Czech autoworkers: ($Y_1$) smoking, ($Y_2$) strenuous mental work, ($Y_3$) strenuous physical work, ($Y_4$) systolic blood pressure, ($Y_5$) ratio of beta and alpha lipoproteins and ($Y_6$) family anamnesis of coronary heart disease.
In absence of any prior information we assume that the data are generated from the discrete log-linear DGM presented in Section 5.
Each of the $64$ cells in the contingency table is assigned a pseudo count of $1/64$, which in turn induces hyper parameters in the conjugate prior Hyp-Dir$(\hyperparamletter)$ defined by $\{\pi(\theta_Q;\hyperparamletter_Q)\}_{Q\in \mathcal Q(\graph)}$ where $\hyperparamletter_Q=\{\hyperparamletter_{Q(i_Q)}\}_{i\in \mathsf I_Q }$ and $\hyperparamletter_{Q(i_Q)} = |\mathsf I_{V\setminus Q}| / 64$.
This type of low dimensional model is suitable for evaluation purposes since it is possible to exactly compute the posterior distribution.
Specifically, the total number of decomposable graphs with six nodes is equal to 18154, allowing for full computation of the posterior distribution.

All the estimators are based on $\N=100$ particles and averaging is performed over $\Nmcmc=10000$ PG-runs according to equation (\ref{eq:mcmc_estimator2}).
Due to the absence of a time-dependent dynamic, we set the bandwidth $\delta$ to $\p$.

The heatmaps for the exact and estimated posterior distributions are displayed in Figure~\ref{fig:autoworkers}, along with the estimated auto-correlation.
A visual inspection of the marginal edge probabilities in the heatmaps indicates a good agreements between the distributions.
From the fast decay from one to zero in the auto-correlation plot we deduce that the PG-sampler exhibit very good mixing properties for this problem.

Table~\ref{tab:autoworkers} summarizes the edge sets for the top five graphs on both the exact and estimated posterior distribution along with their corresponding probabilities.
It is important to note that the top five graphs are exactly the same for these two distributions and the estimated probabilities are in a good agreement with the exact ones.
Our findings are also consistent with the results obtained by \cite{10.2307/25662199} and \cite{10.2307/2291017}.
Specifically, our top highest posterior probability graphs are the same as those identified by \cite{10.2307/25662199}, see Table~2, case $\alpha=1.0$ in that paper.

Finally, after evaluating a range of different combinations, we conclude that our results obtained in this example appear to be insensitive to the choice of CTA parameters $\alpha$ and $\beta$ for this small scale problem.

 \begin{figure}
    \centering
    \begin{minipage}[b]{0.33\textwidth}
        \includegraphics[width=\textwidth]{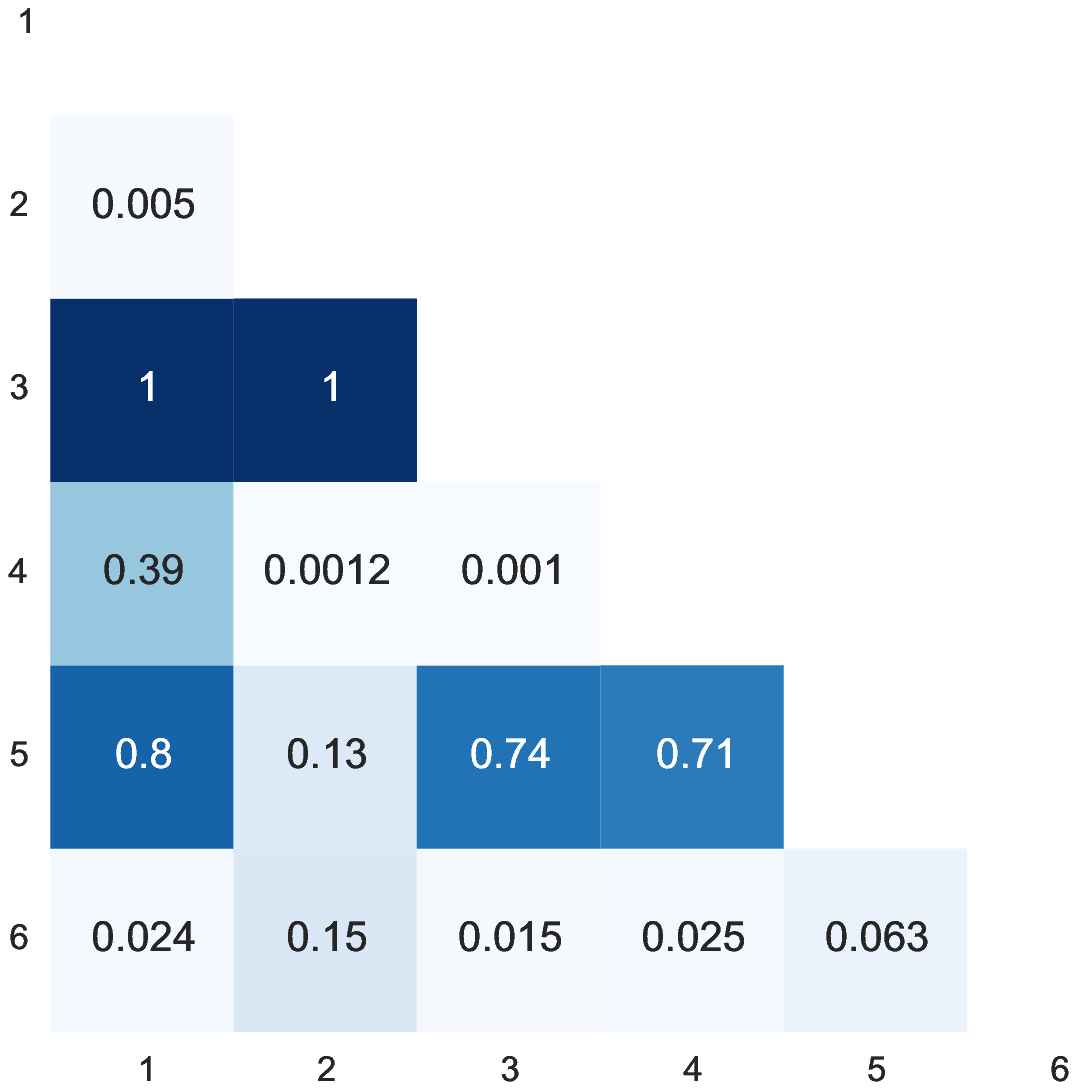}
    \end{minipage}
    ~
    \begin{minipage}[b]{0.33\textwidth}
        \includegraphics[width=\textwidth]{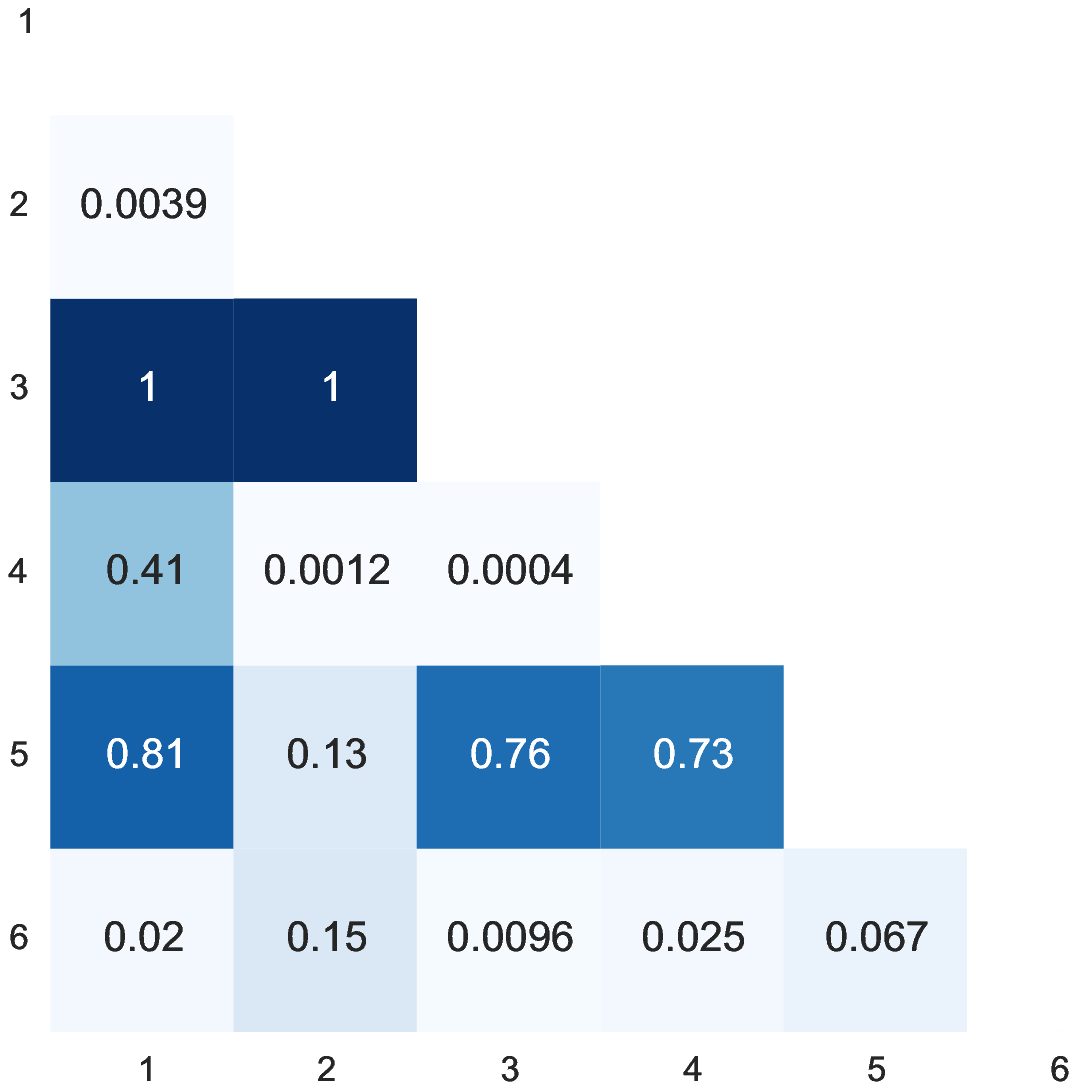}
    \end{minipage}
        ~
    \begin{minipage}[b]{0.33\textwidth}
        \includegraphics[width=\textwidth]{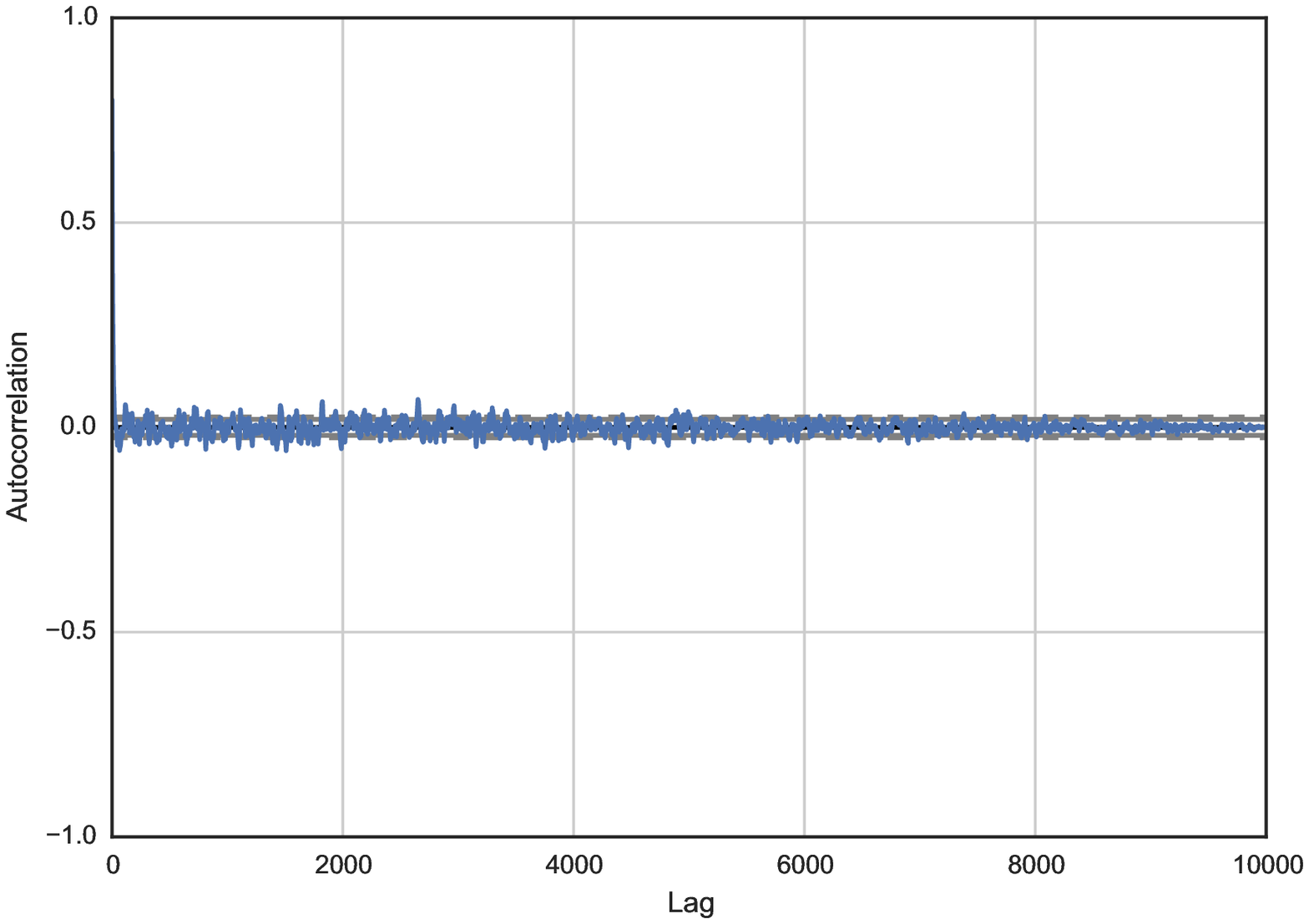}
    \end{minipage}
    \caption{True edge heatmap (left), estimated heatmap (middle) and auto-correlation of the number of edges in the trajectory of graph (right).}
    \label{fig:autoworkers}
\end{figure}

\begin{table}[h]
\begin{center}
    \begin{tabular}{ c | c | c  }
    Edge set & Exact & Estimated \\
    \hline
    $(1, 3), (1, 5), (2, 3), (3, 5), (4, 5)$ & 0.248 & 0.263 \\
    $(1, 3), (1, 4), (1, 5), (2, 3), (3, 5), (4, 5)$&0.104 & 0.115 \\
     $(1, 3), (1, 4), (1, 5), (2, 3), (3, 5)$& 0.101& 0.103\\
     $(1, 3), (2, 3), (2, 5), (4, 5)$& 0.059 & 0.062\\
     $(1, 3), (1, 5), (2, 3), (2, 6), (3, 5), (4, 5)$&0.051 & 0.051\\
    \end{tabular}
    \caption{The estimated graph probabilities are compared to the true posterior probabilities for the five graphs with the highest posterior probabilities.}
    \label{tab:autoworkers}
\end{center}
\end{table}

\subsection{Discrete data with $p=15$}
 \begin{figure}
    \centering
    \begin{minipage}[b]{0.4\textwidth}
        \includegraphics[width=\textwidth]{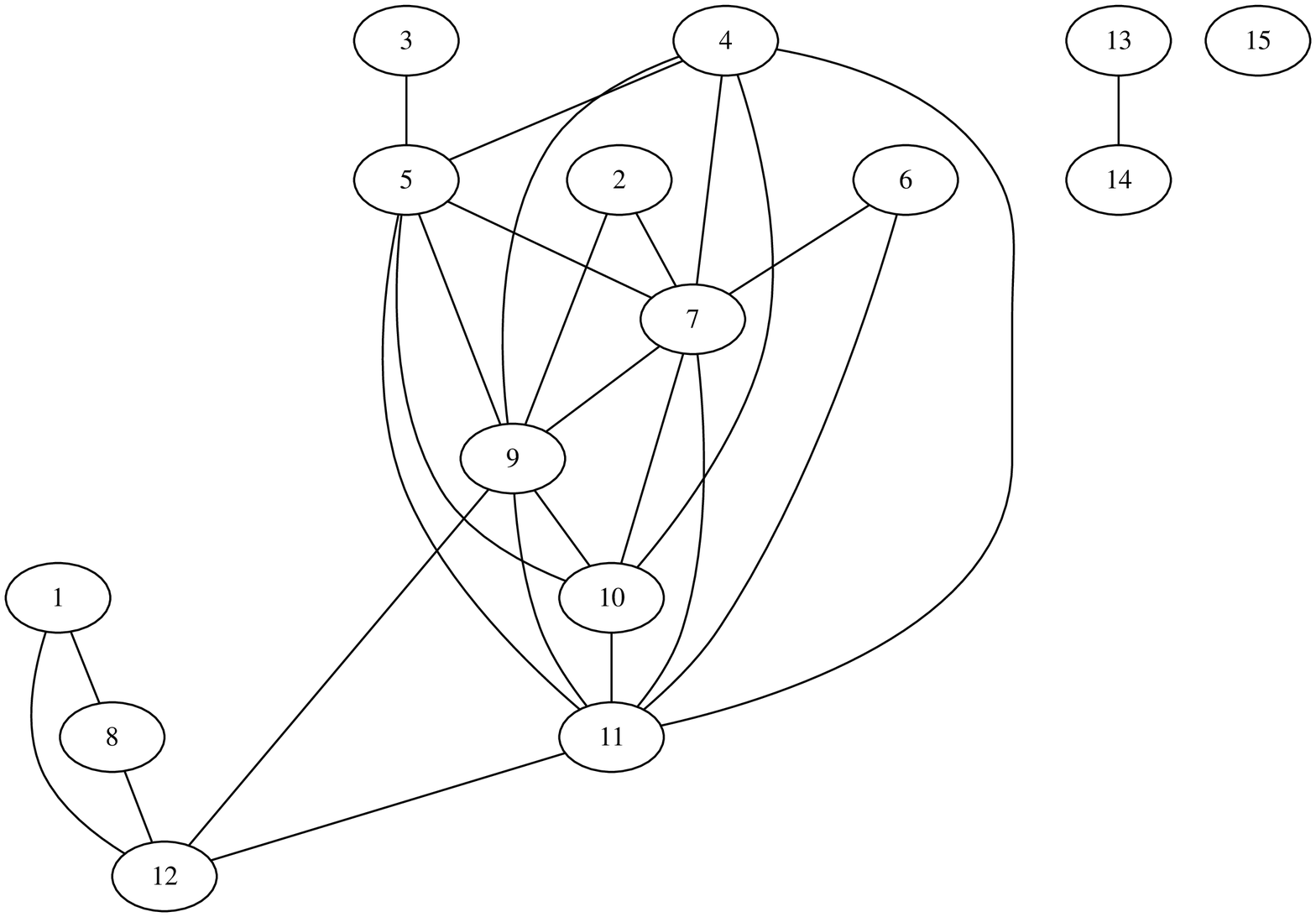}
    \end{minipage}
    ~
    \begin{minipage}[b]{0.4\textwidth}
        \includegraphics[width=\textwidth]{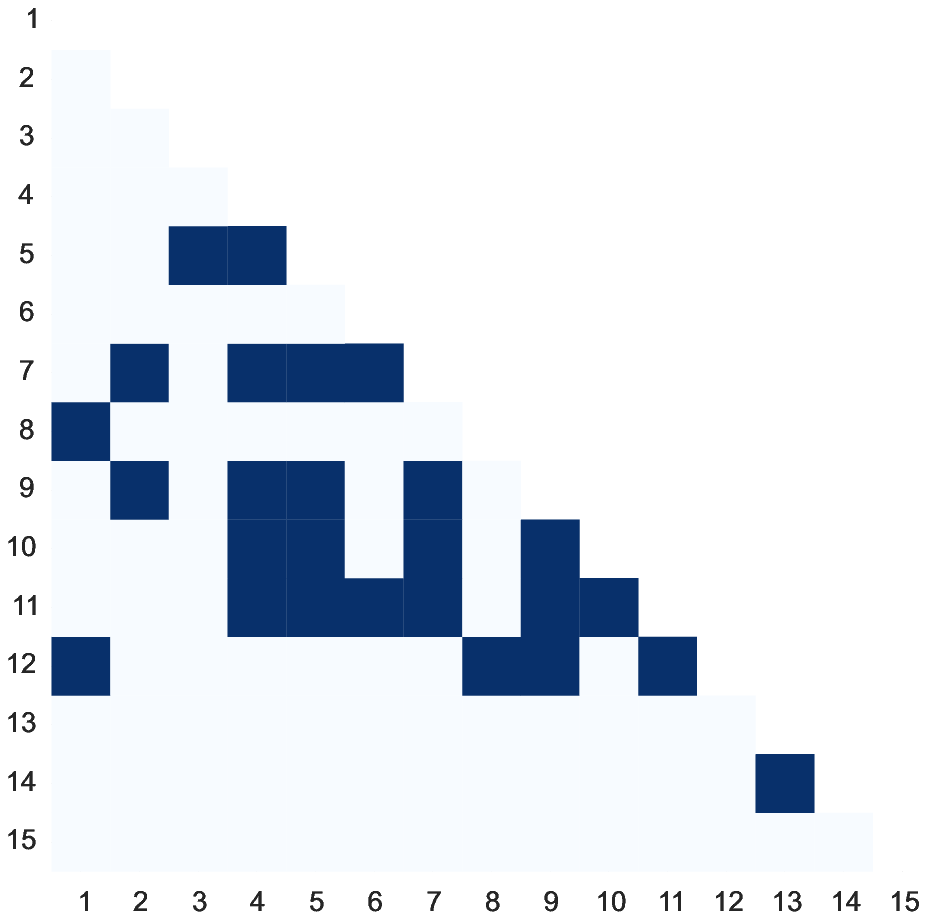}
    \end{minipage}
    \caption{The true underlying decomposable graph on $\p = 15$ nodes along with its adjacency matrix.}
    \label{fig:carvalho}
\end{figure}
In this example we study a discrete log-linear DGM with $\p=15$ nodes and the dependence structure displayed in Figure~\ref{fig:carvalho}, presented in \cite[Figure~1]{2005}.
The parameters were selected to satisfy (\ref{eq:hyper_markov_param}) thereby ensuring that the distribution $\paramsp_{\graph}$ specified in Example \ref{ex:multinomial:graphical:models} will be Markov with respect to $\graph$.
Analogously to the previous example, the we use the Hyp-Dir$(\hyperparamletter)$ prior and assign to each cells in the contingency table a pseudo count of $1/2^{15}$.
Due to the absence of any time interpretation of the model, the bandwidth parameter $\delta$ is selected as $\p$.
We used the CTA parameters $\subtreecont=0.2$ and $\subtreenonempty=0.8$, obtained as the parameter setting giving best mixing properties within all the possible combination on the grid $\subtreecont,\subtreenonempty=0.2,0.5,0.8$.

To evaluate how the estimation accuracy is affected by the number of particles, we sampled $n=1000$ data vectors and estimated both the graph posterior and the auto-correlation function with $\N=20$ and $\N=100$.
By comparing the true underlying graph in Figure~\ref{fig:carvalho} with the heatmaps and the MAP in Figure~\ref{fig:carvalho_est}, we observe that increasing $\N$ from $20$ to $100$ gives a slightly better agreement with the true adjacency matrix.
This effect can be further explained by the behavior of the estimated auto-correlation function; by increasing $\N$ a clear reduction of the auto-correlation can be noted.
Qualitatively we conclude that the mobility of the PG-sampler is improved when increasing $\N$.

 \begin{figure}
    \centering
    \begin{minipage}[b]{0.4\textwidth}
        \includegraphics[width=\textwidth]{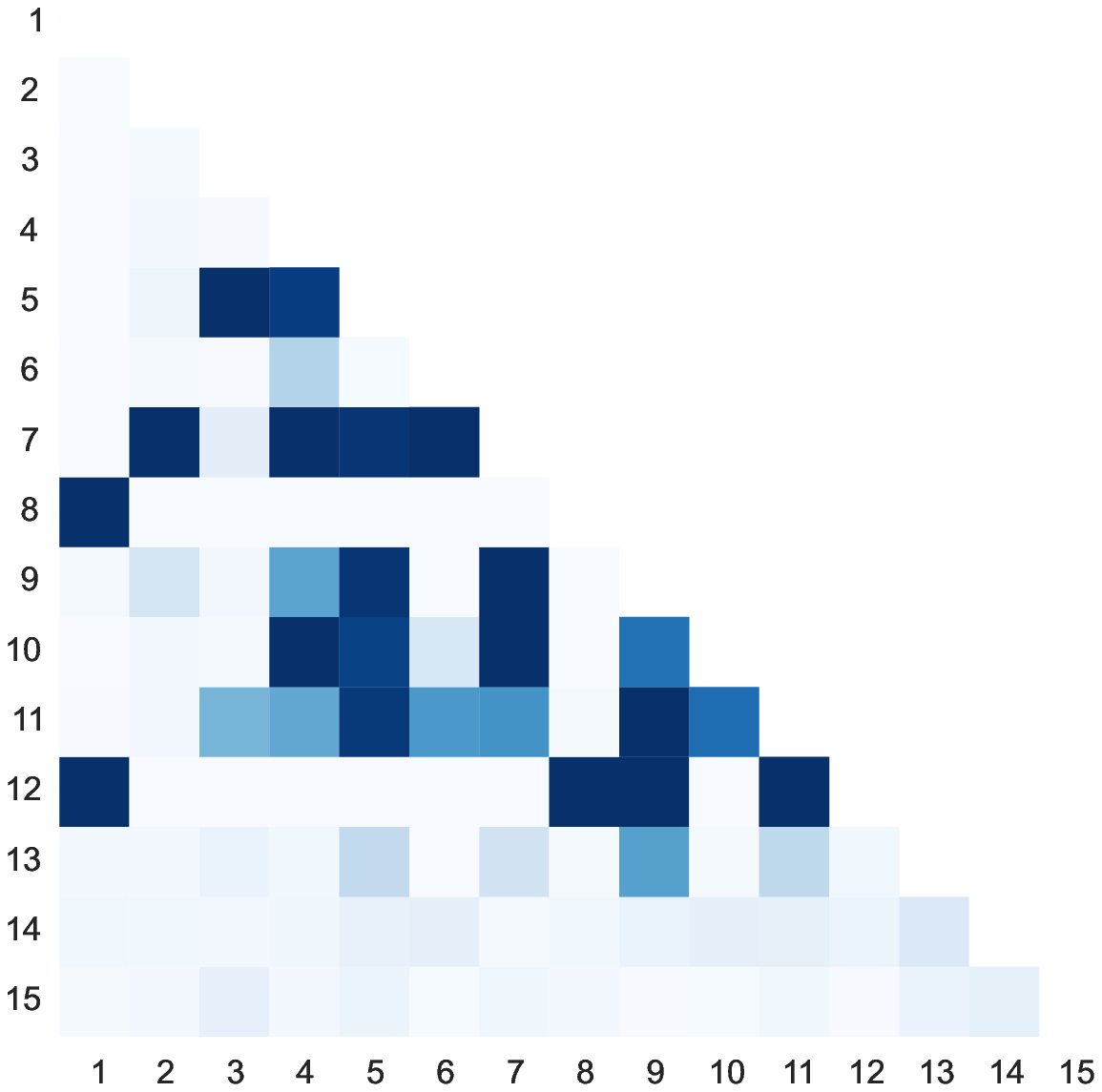}
    \end{minipage}
    ~
    \begin{minipage}[b]{0.45\textwidth}
        \includegraphics[width=\textwidth]{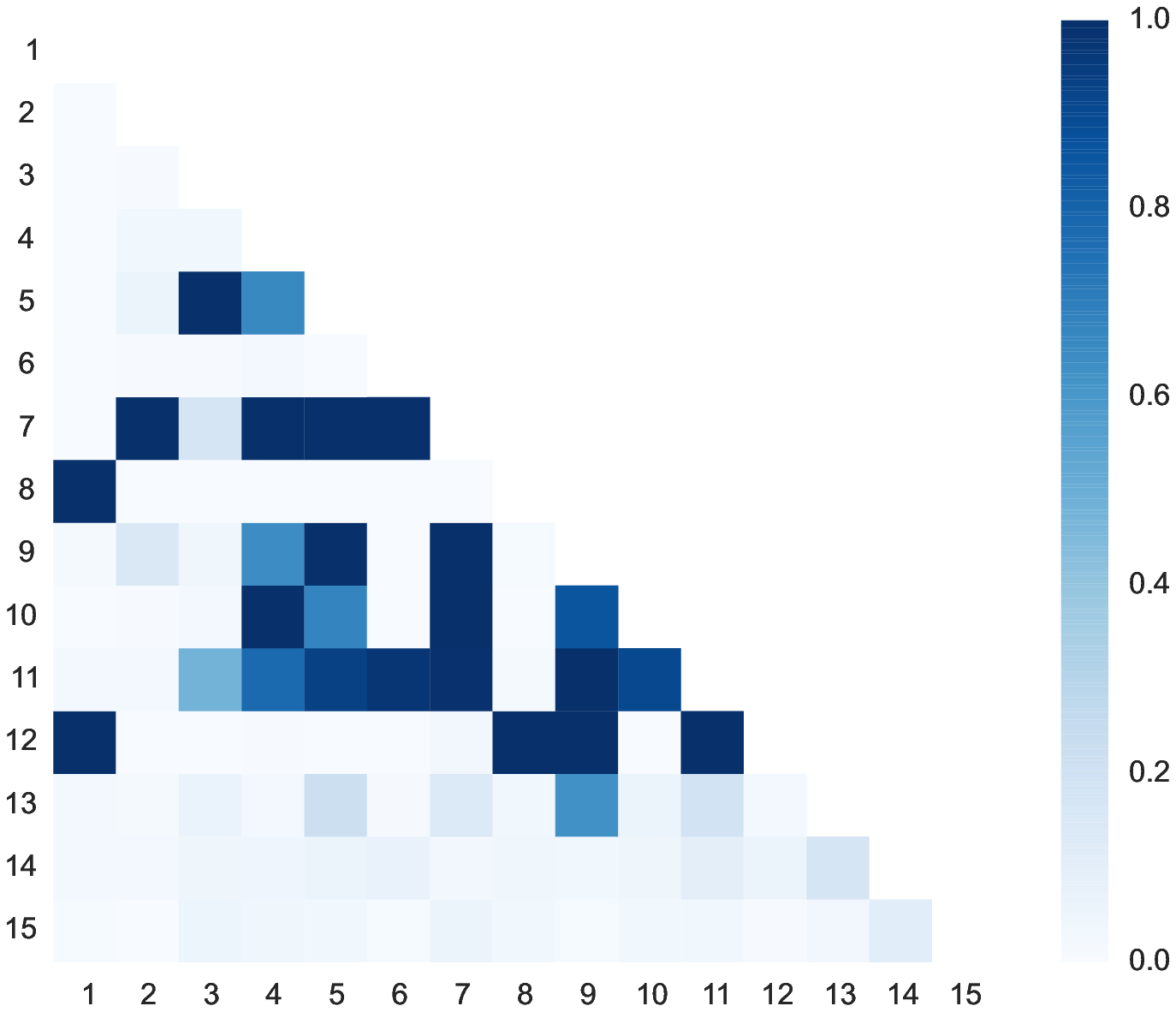}
    \end{minipage}

    \begin{minipage}[t]{0.4\textwidth}
        \includegraphics[width=\textwidth]{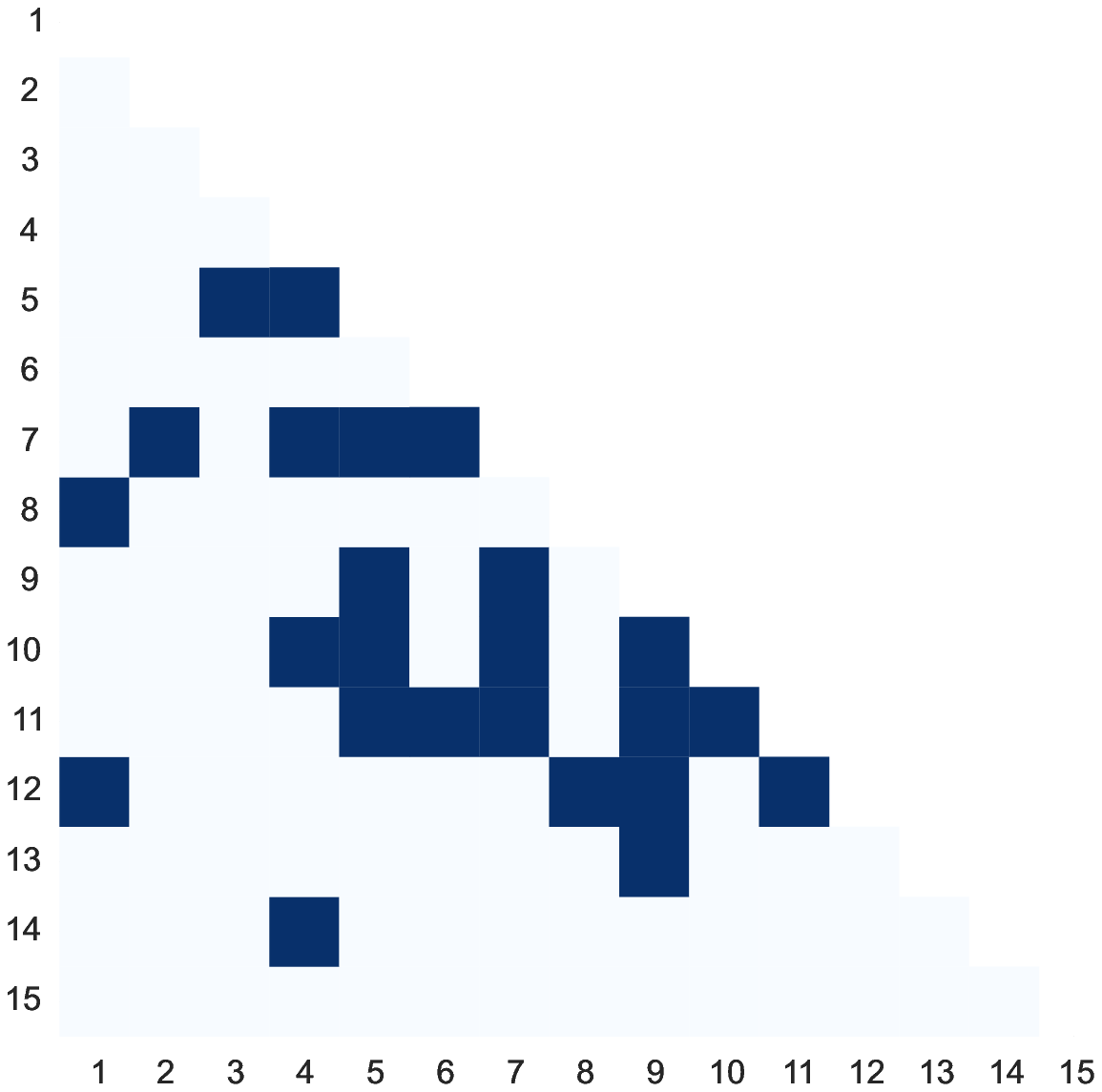}
    \end{minipage}
    ~
    \begin{minipage}[t]{0.45\textwidth}
        \includegraphics[width=\textwidth]{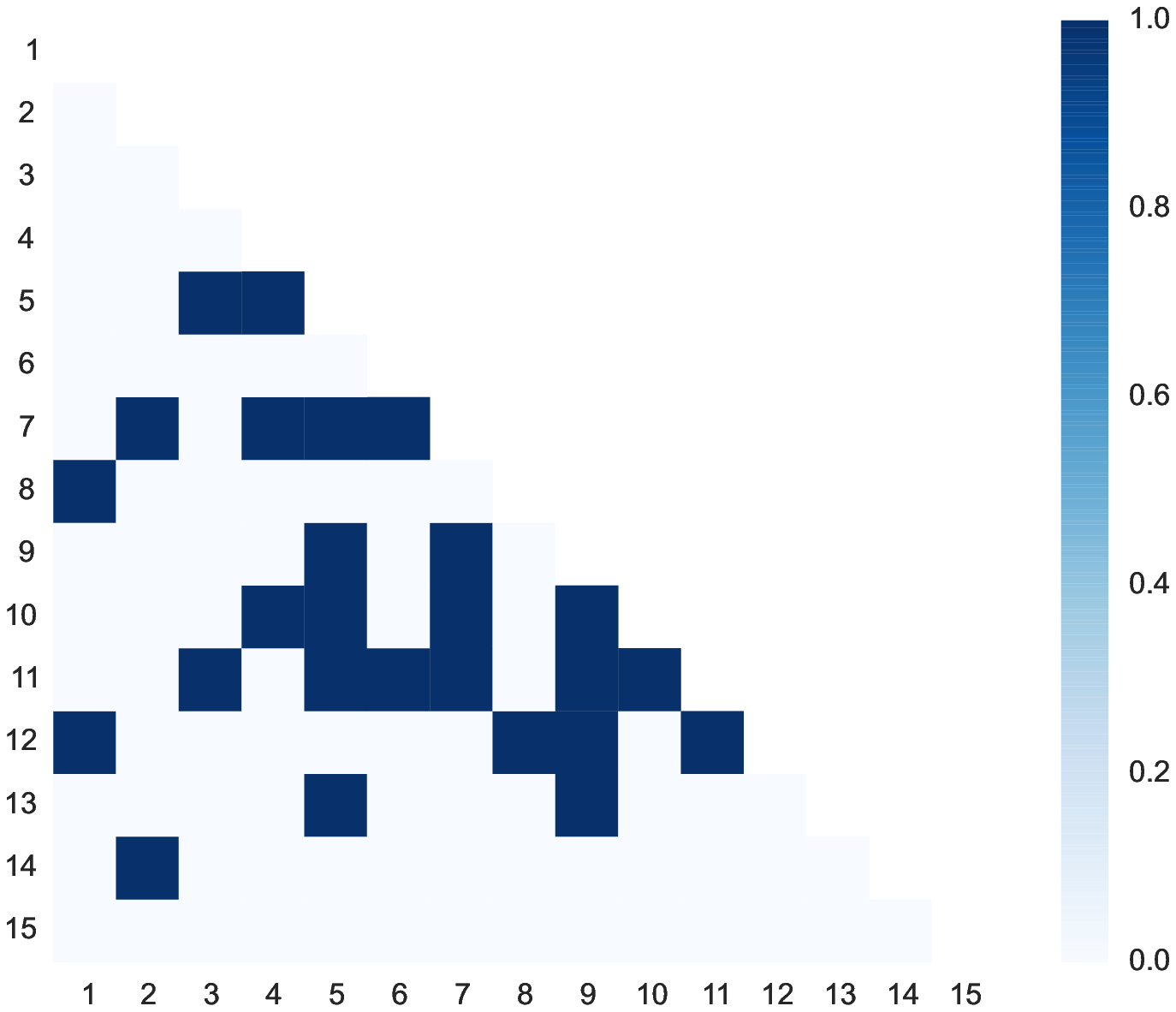}
    \end{minipage}

    \begin{minipage}[t]{0.45\textwidth}
        \includegraphics[width=\textwidth]{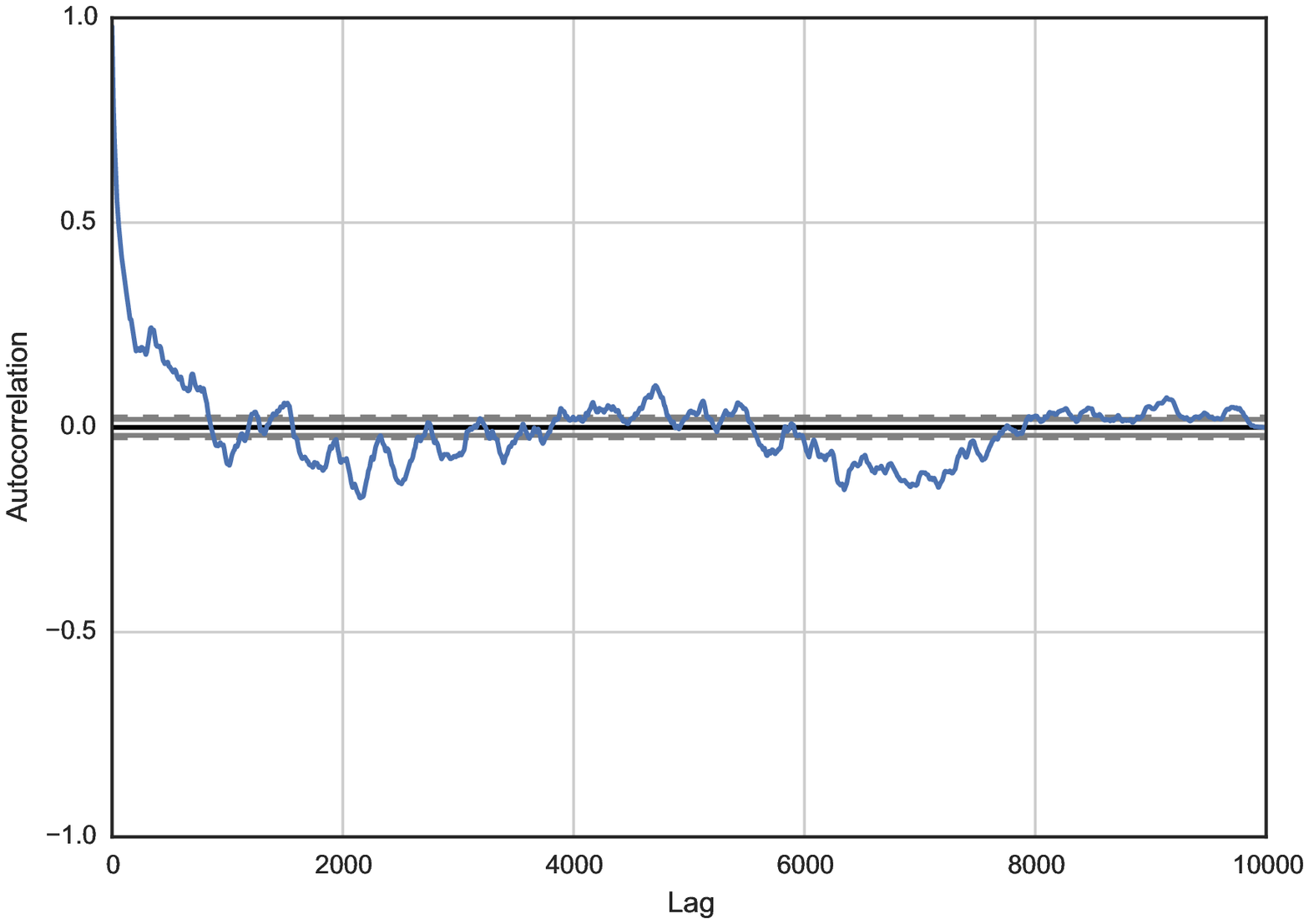}
    \end{minipage}
    ~
    \begin{minipage}[t]{0.45\textwidth}
        \includegraphics[width=\textwidth]{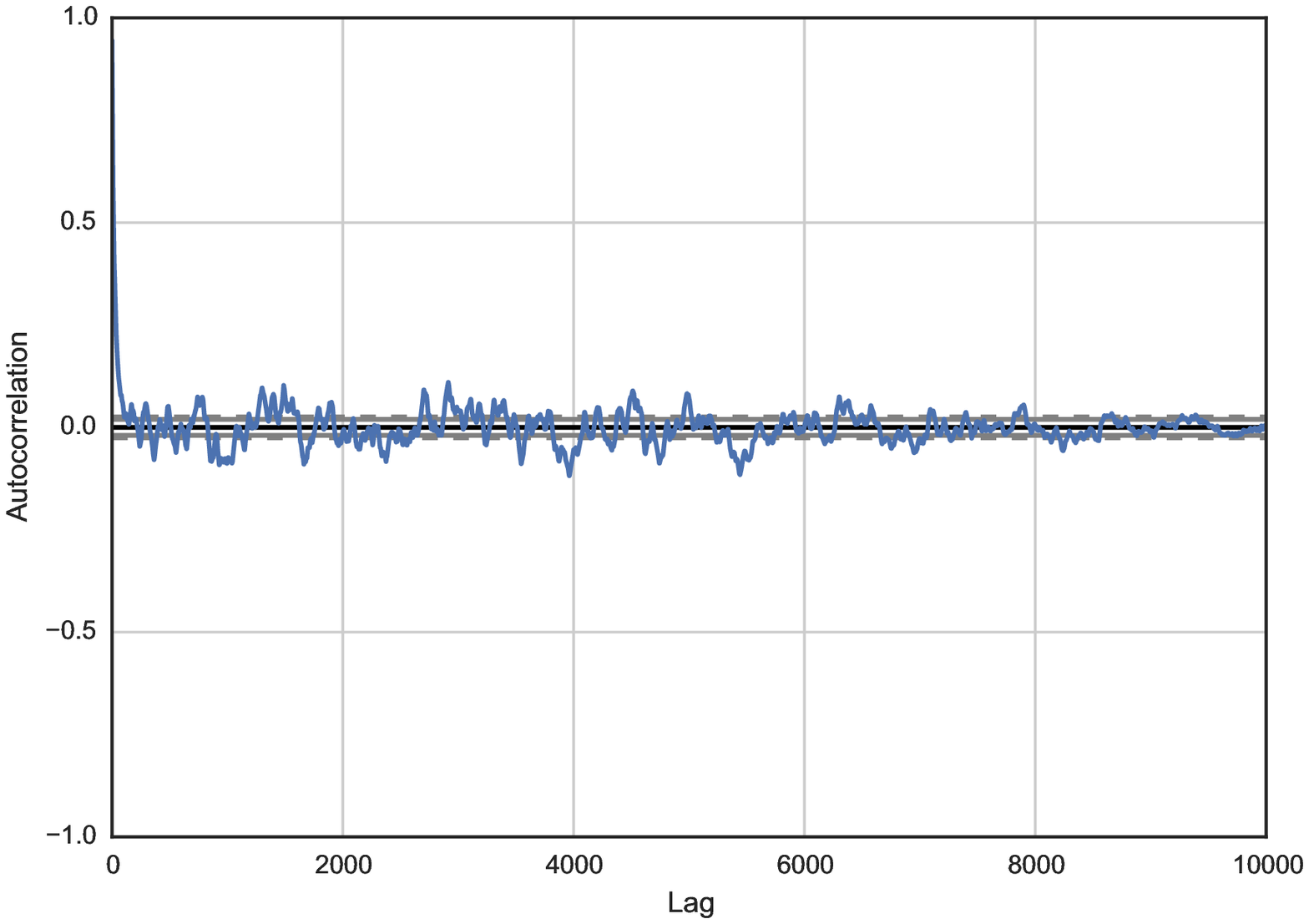}
    \end{minipage}

    \caption{Estimation of the graph posterior for the log-linear model with $\p=15$ and $n=100$. The CTA parameters are $\subtreecont=0.2,\, \subtreenonempty=0.8$ and $\delta=15$.
    The number of McMC sweeps $\Nmcmc$ is set to $10000$.
    The left and right panel correspond to $\N=20$ and $\N=100$, respectively. For both panels from top to bottom, the first figure presents the estimated edge heatmap, the estimated MAP graph and estimated auto-correlation of the number of edges in the graph.}
    \label{fig:carvalho_est}
\end{figure}

\subsection{Continuous data with temporal dependence}
\label{sec:time_dep}

In this example, we study a Gaussian DGM where a temporal interpretation of the underlying dependence structure is suitable.
The graph structure along with its adjacency matrix are displayed in Figure~\ref{fig:AR1-5} and can be interpreted as an AR-process with lag varying between 1 and 5.
The model to be considered is represented by a Gaussian distribution with zero mean and covariance matrix $\precmat^{-1}$ defined as
\begin{align*}
    (\precmat^{-1})_{ij} = \begin{cases}
        \sigma^2, &\text{ if } i=j\\
        \rho\sigma^2, &\text{ if } (i,j) \in \graph \\
    \end{cases}
\end{align*}
and $\precmat_{ij} =0$ if $(i,j) \notin \graph$.
This is a modification of the second order intra-class structure considered in \citet{Green01032013}, where the bandwidth is varying.
We have sampled $n=100$ data vectors from this model where the variance $\sigma^2$, and correlation coefficient $\rho$, were set to $1.0$ and $0.9$, respectively. 

Following Example \ref{ex:Gaussian:graphical:graph:models}, $\precmat$ is assigned a hyper Wishart prior, where for each clique $\jtnode$ the degrees of freedom is set to be equal to $\p=50$, and the scale matrix is set to be the identity matrix of dimension $|\jtnode|$. 

In this example, 10 PG trajectories of length $M=10000$ were sampled, of which the first 3000 samples were removed as burn-in period.
The heatmaps and MAP graph estimates most similar to the true graph are displayed in Figure~\ref{fig:AR1-5_est}.
The temporal interpretation of the structure of this graph is particularly suited for investigating the role of $\delta$ as a tuning parameter.
Results for $\delta=5$ and $\delta=50$ are diplayed in Figure~\ref{fig:AR1-5_est} and Figure~\ref{fig:AR1-5_est_traj} respectively.
By comparing the two heatmaps one can notice that the dependence structure can be better captured by selecting a value of $\delta$ which corresponds to the maximal bandwidth size of 5 for the true graph.
In addition, by letting $\subtreecont=0.8$ and $\subtreenonempty=0.5$ we are able to express a priority for connected graphs, we obtain a heatmap pattern which better mimics the true one.
This effect is also reflected by the log-likelihood trajectories in the bottom row of Figure~\ref{fig:AR1-5_est_traj}.
The graph size auto-correlation (after burn-in) shown in the middle row of Figure~\ref{fig:AR1-5_est_traj}, decays slightly faster by a smaller $\delta$ and the mobility of the size trajectory, top row of Figure~\ref{fig:AR1-5_est_traj} is improved in this example.

\subsection{Comparison to the Metropolis-Hastings algorithm}\label{sub:comparison_to_cite}

We have compared the PG-sampler with the Metropolis-Hastings (MH) algorithm proposed in~\cite{Green01032013} for the Gaussian example in Section~\ref{sec:time_dep}.
Following the suggestions from that paper we randomise the junction tree every $\lambda$ iteration, we executed their algorithm for both $\lambda=100$ and $\lambda=1000$. 
We can confirm the suggestion in that paper that a more frequent junction tree randomisation has an improving effect on recovering the underlying model.
Therefore, results only for $\lambda=100$ are demonstrated here.

The main advantage of the PG-sampler as compared to the MH-sampler is its mixing properties.
Figure~\ref{fig:green_stat} shows the 10 trajectories of the MH-sampler after 350000 iterations (in total 500000 graphs were sampled), out of which 4 trajectories seem to have reached stationarity; see the size- and log-likelihood trajectories displayed in green, orange, gray and brown.
In the middle panel of Figure~\ref{fig:AR1-5_est} and Figure~\ref{fig:green_stat}, it is seen that the estimated auto-correlation of the MH-sampler is substantially stronger than that for the PG-sampler for both choices of $\delta$, being on average about $20000$ for the MH-sampler as compared to about $500$ for the PG-samples with $\delta=5$. 
Figure~\ref{fig:green_hm} shows the heatmap and MAP estimate corresponding to the green trajectory in Figure~\ref{fig:green_stat}.

On the other hand, the MH-sampler is superior in speed, which is at cost of the slower mixing seemingly inherited by the local moves.
Using the Java implementation from~\citet{Green01032013}, the MH-sampler with randomising interval \(\lambda=100 \) were able to sample about $20000$ samples per second, while each sample took about $3$ seconds for the PG-sampler.
Note that the implementation by~\cite{Green01032013} considers the intra-class model introduced Section~\ref{sec:time_dep} so that $\precmat$ is defined by merely two parameters $\sigma^2$ and $\rho$, and independent two priors are assigned for these instead of the hyper Wishart distribution as presented in Example~\ref{ex:Gaussian:graphical:graph:models}.
However, a gain in sample time for the MH-sampler is expected since, in each PG iteration, the conditional SMC procedure generates $p(N-1)$ junction trees with $\mathcal O(p^2)$ internal nodes per PG sample.
Also, when a new junction tree is proposed in the SMC algorithm, the previous junction tree, which it stems from is copied since it could potentially be an ancestor for other trees as well due to the re-sampling step.
As scope for future research, investigating new data structures for junction trees which are tailored to sequential sampling is of great interest.
We also expect that the speed of the PG-sampler could be improved by, for example parallelizing the SMC-updates and by improved caching strategies.
\begin{figure}[htbp]
    \centering
    \begin{minipage}[b]{0.4\textwidth}
        \includegraphics[width=\textwidth]{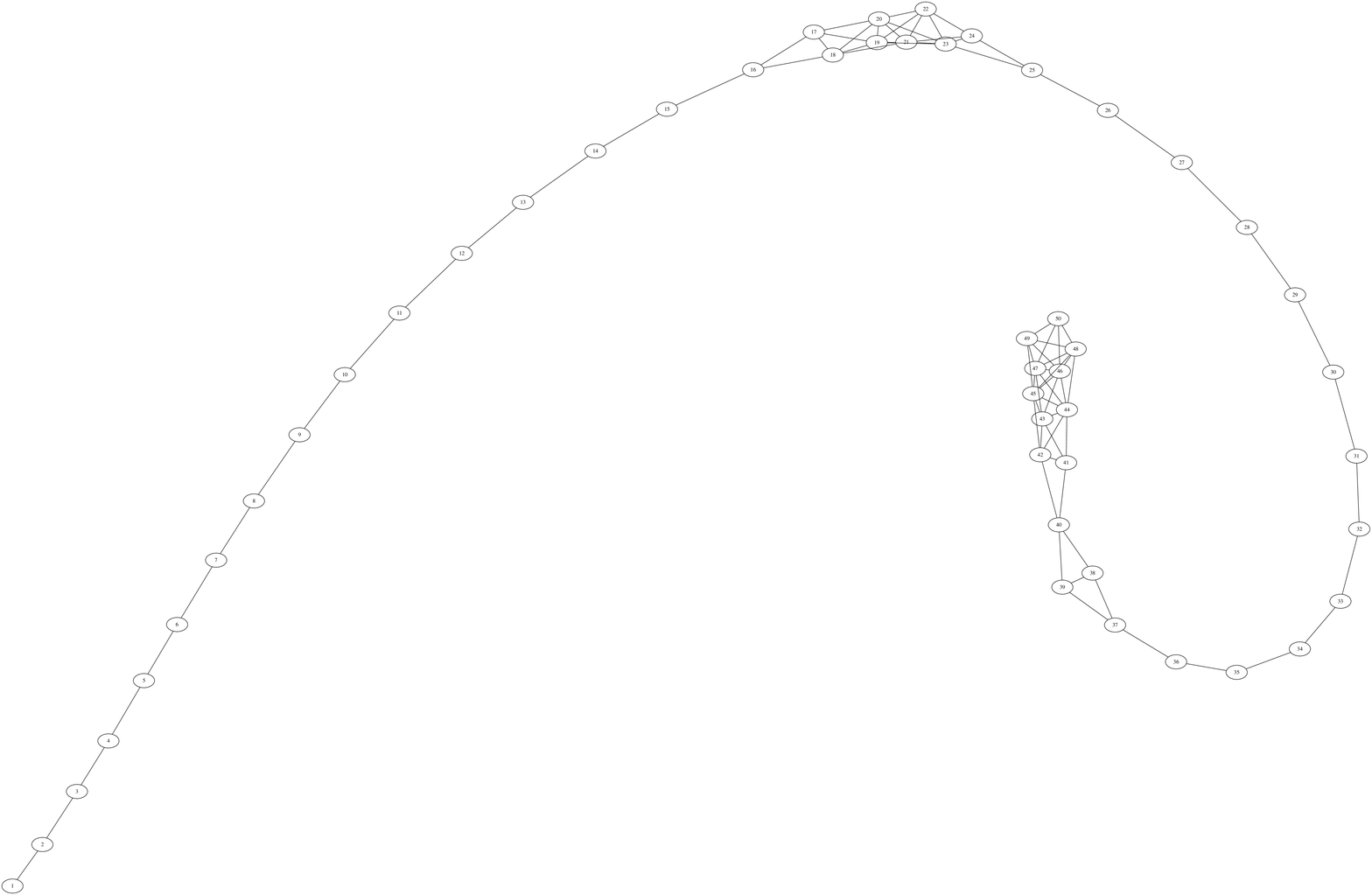}
    \end{minipage}
    ~
    \begin{minipage}[b]{0.4\textwidth}
        \includegraphics[width=\textwidth]{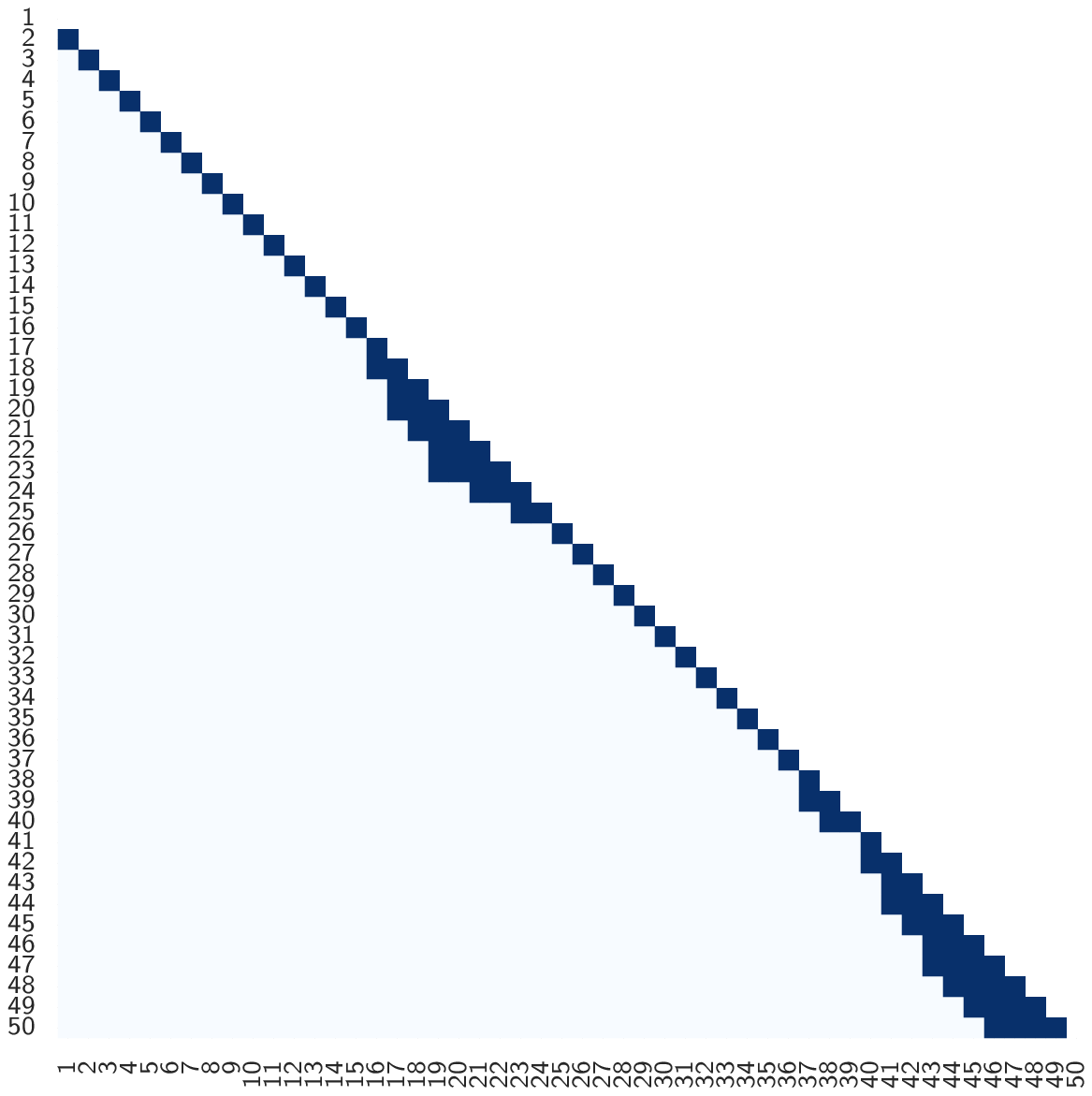}
    \end{minipage}
    \caption{The true underlying decomposable graph on $\p = 50$ nodes along with its adjacency matrix.}
    \label{fig:AR1-5}
\end{figure}
\begin{figure}[htbp]
    \centering

    \begin{minipage}[b]{0.38\textwidth}
        \includegraphics[width=\textwidth,trim={3cm 0 3cm 0},clip]{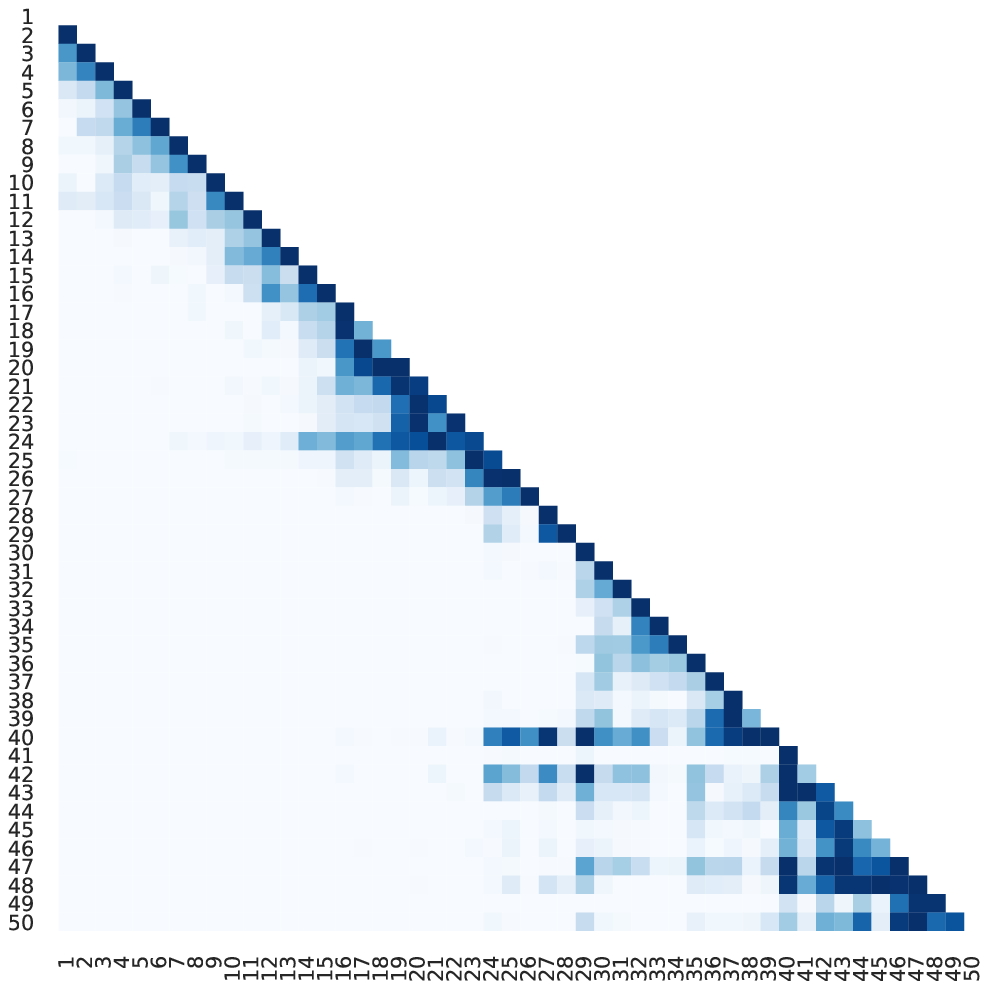}
    \end{minipage}
    ~
    \begin{minipage}[b]{0.45\textwidth}
        \includegraphics[width=\textwidth,trim={2cm 0 2cm 0},clip]{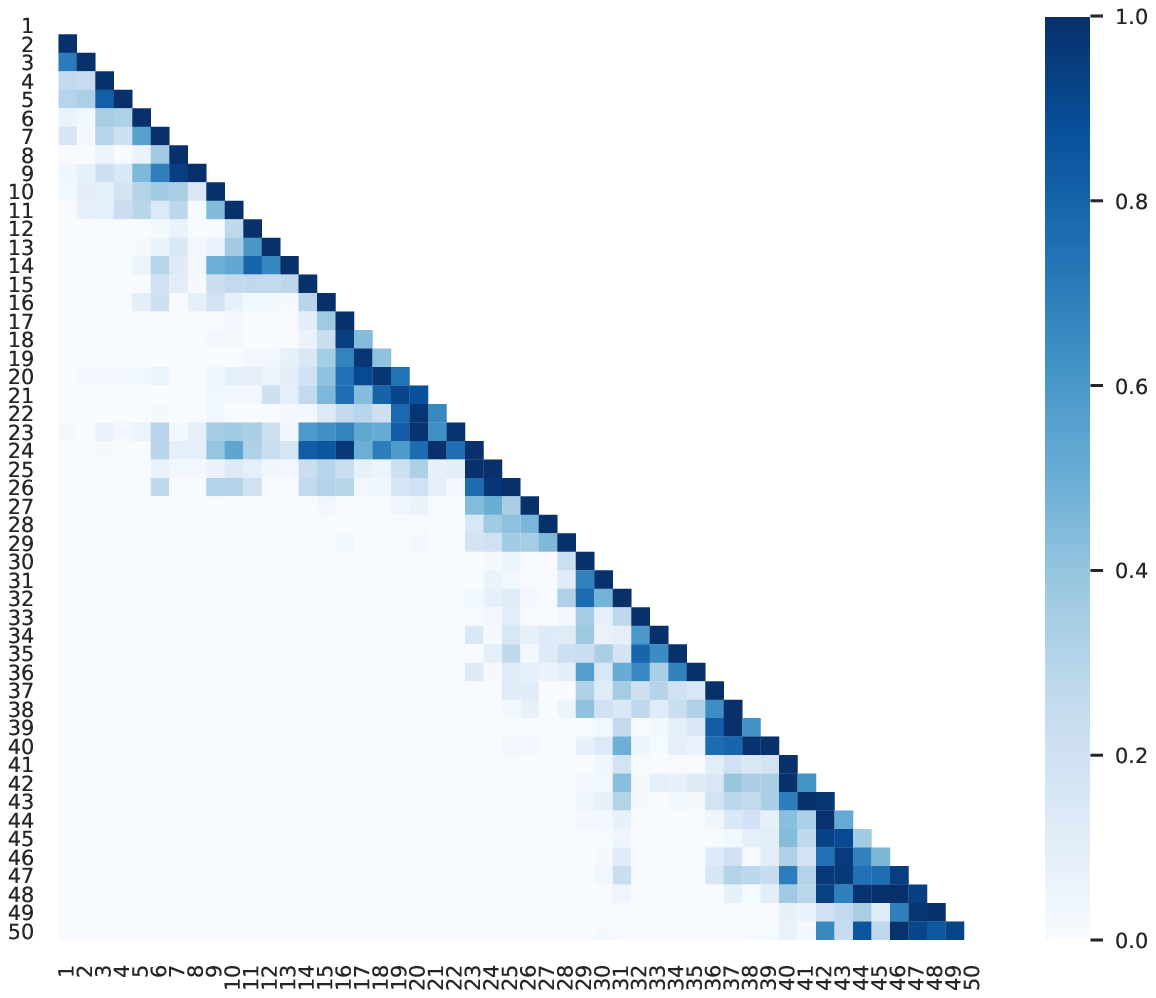}
    \end{minipage}

    \begin{minipage}[t]{0.38\textwidth}
        \includegraphics[width=\textwidth,trim={3cm 0 3cm, 0},clip]{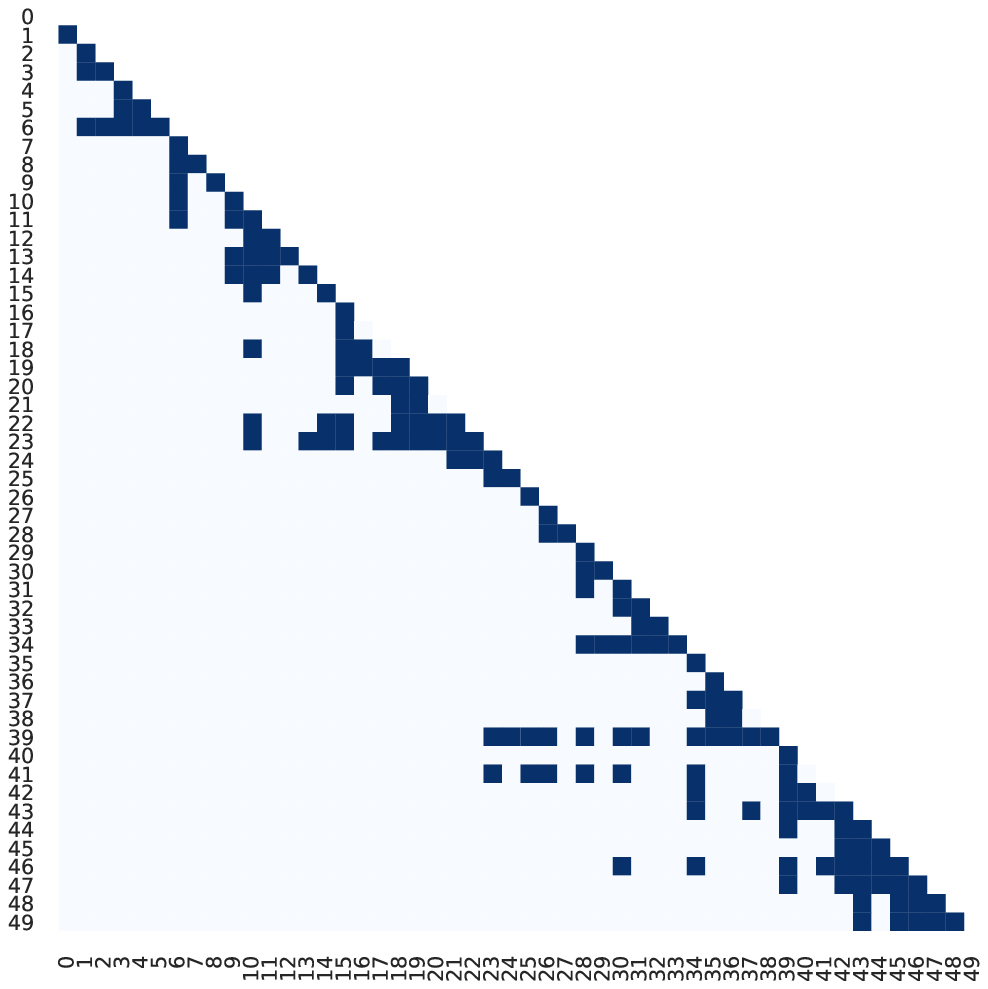}
    \end{minipage}
    ~
    \begin{minipage}[t]{0.38\textwidth}
        \includegraphics[width=\textwidth,trim={3cm 0 3cm 0},clip]{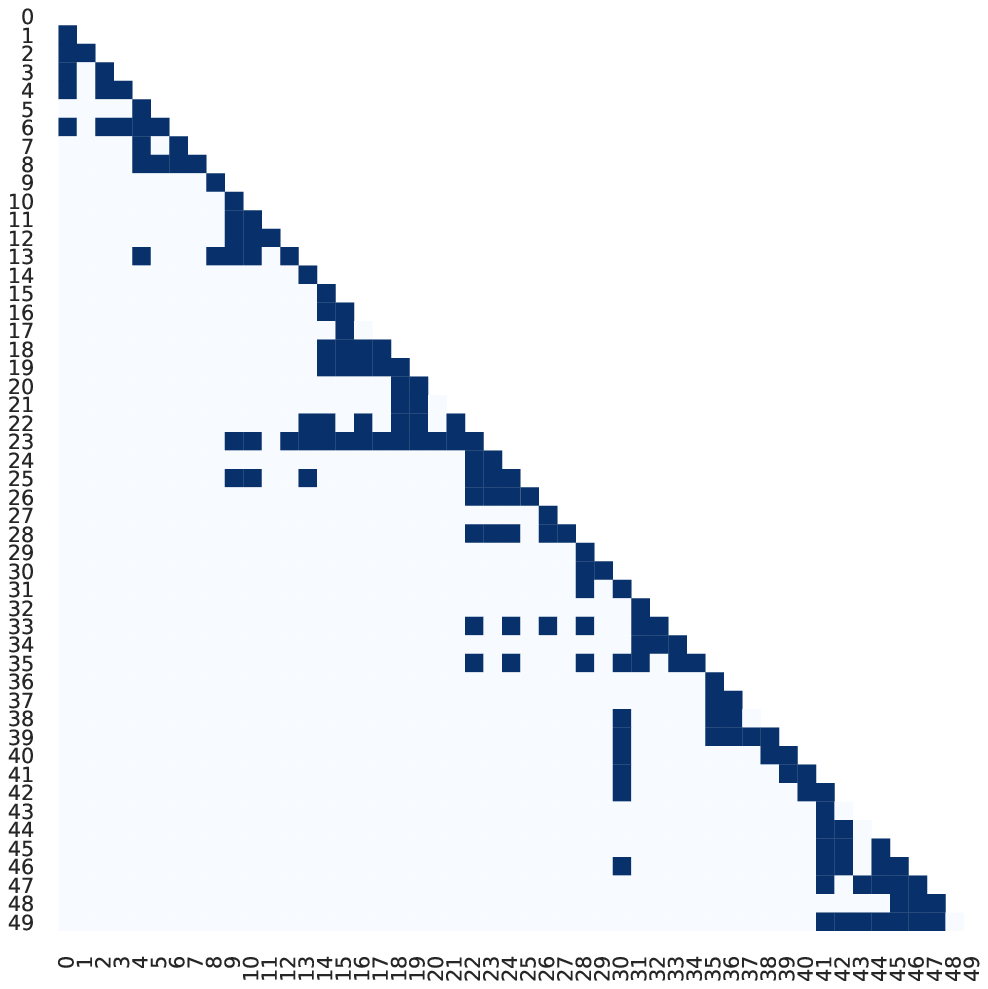}
    \end{minipage}
    \caption{Heatmaps (top row) and MAP graph estimates (bottom row) for the PG-sampler with $\delta=5$ (left panel) and $\delta=50$ (right panel).}
    \label{fig:AR1-5_est}
\end{figure}

\begin{figure}[!htbp]
    \centering

    \begin{minipage}[t]{0.45\textwidth}
        \includegraphics[width=\textwidth]{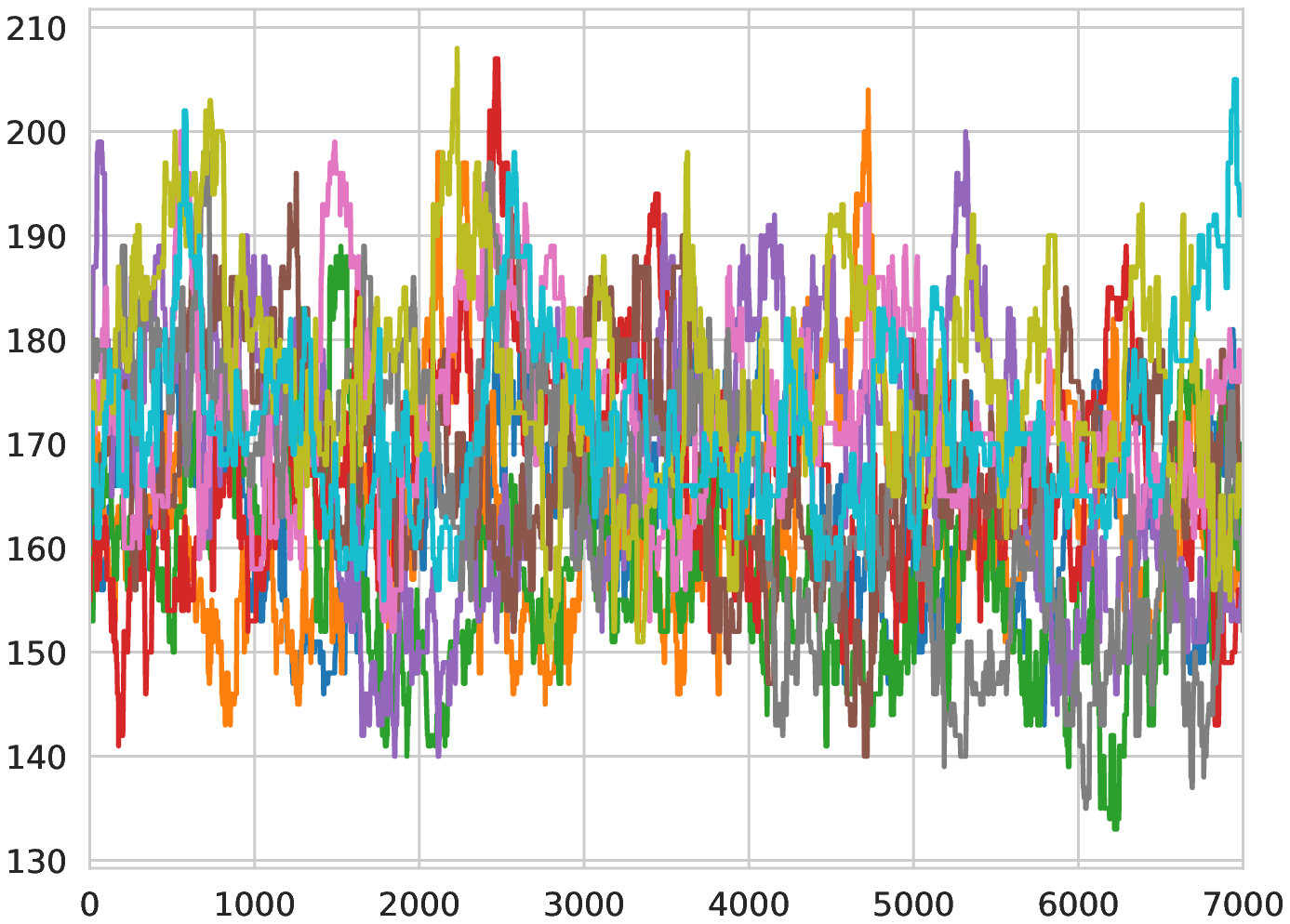}
    \end{minipage}
    ~
    \begin{minipage}[t]{0.45\textwidth}
        \includegraphics[width=\textwidth]{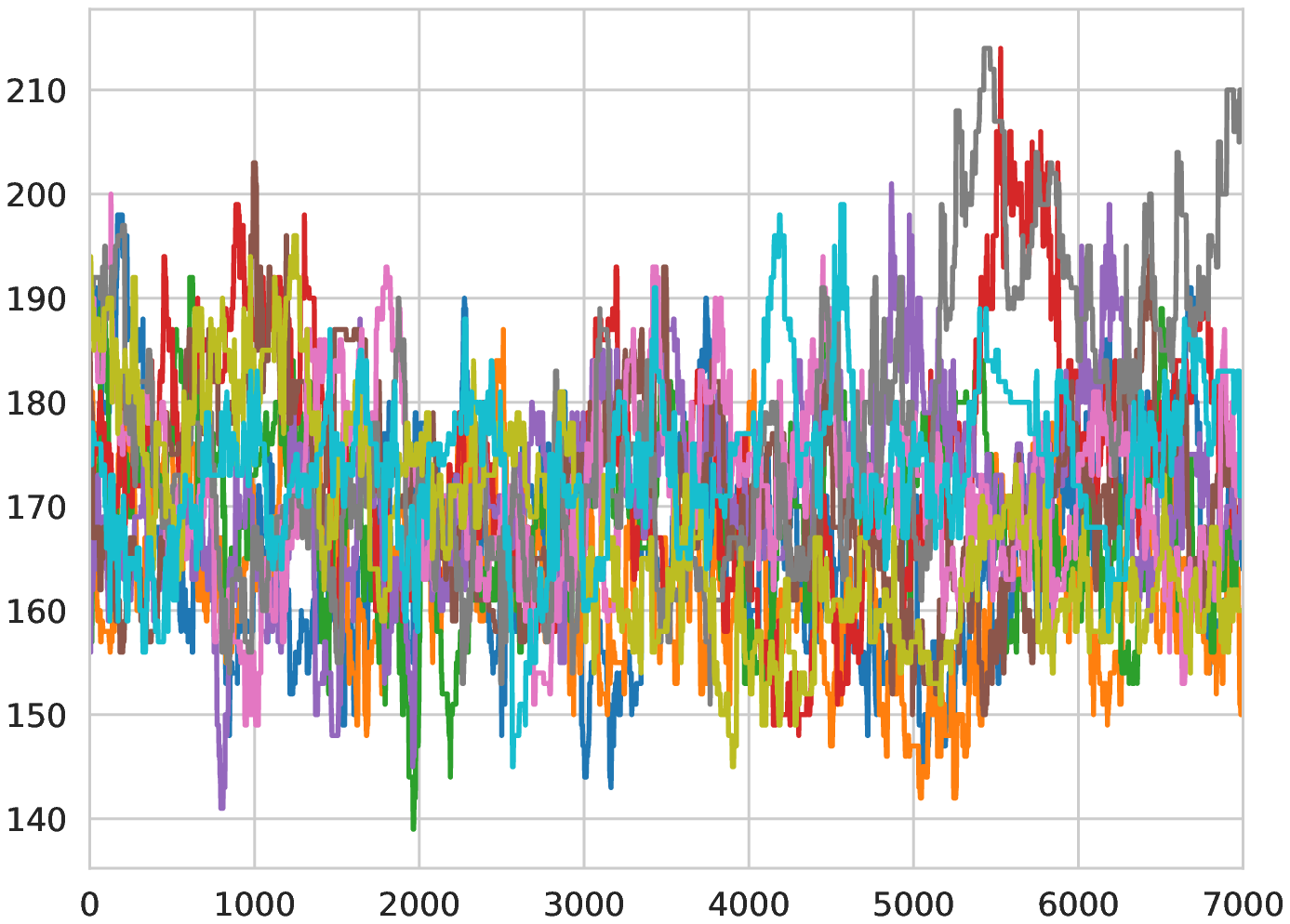}
    \end{minipage}

    \begin{minipage}[t]{0.45\textwidth}
        \includegraphics[width=\textwidth]{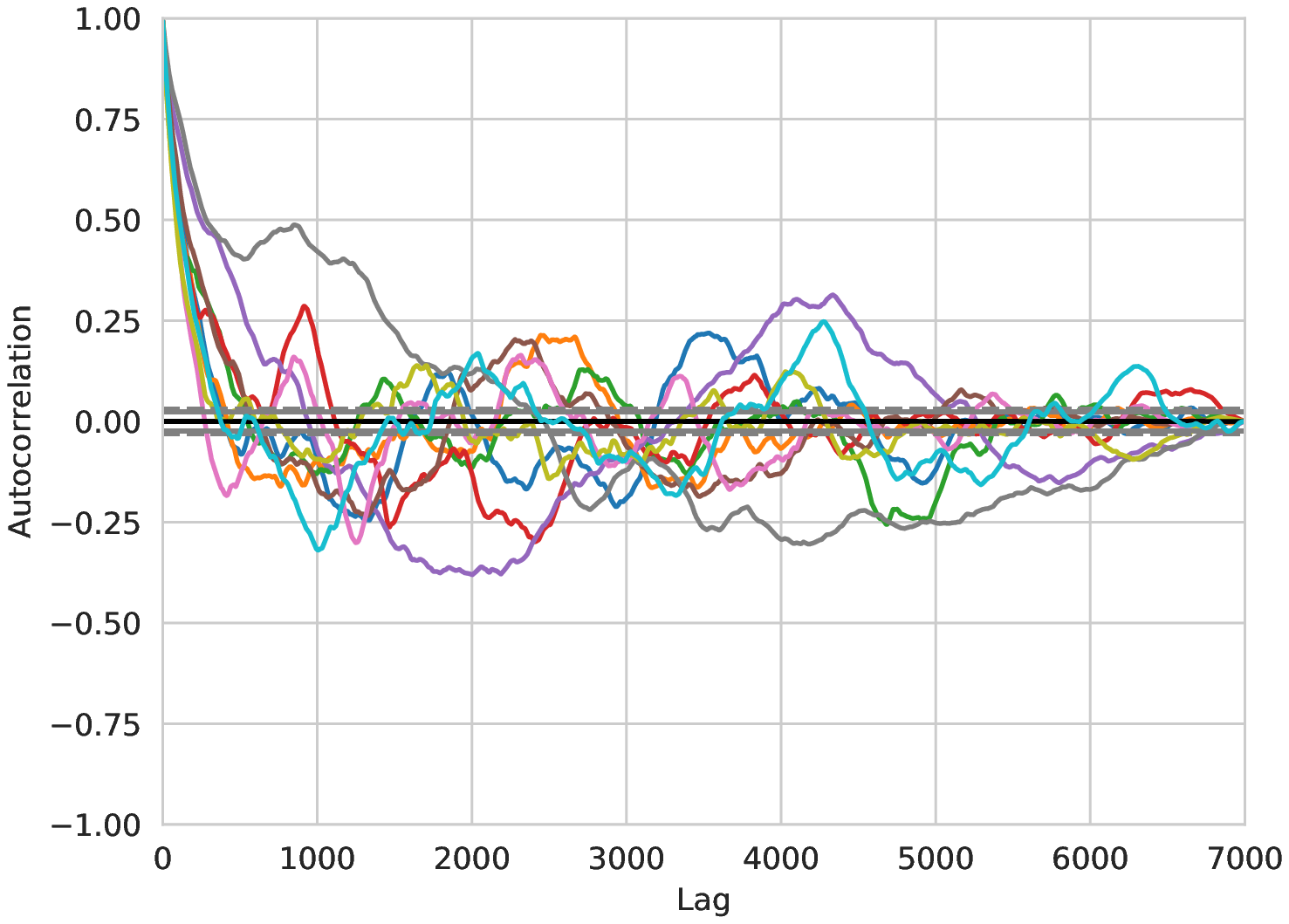}
    \end{minipage}
    ~
    \begin{minipage}[t]{0.45\textwidth}
        \includegraphics[width=\textwidth]{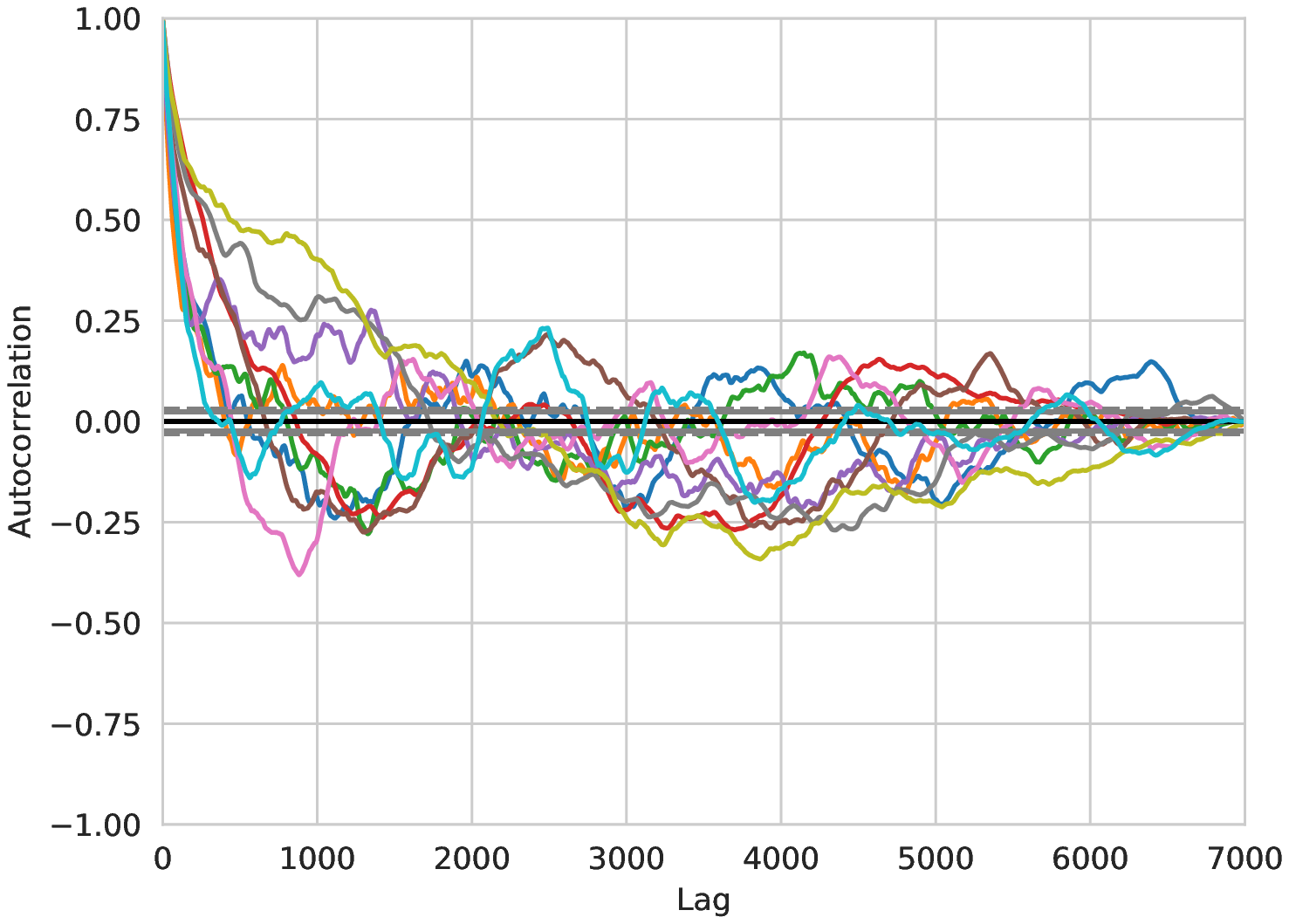}
    \end{minipage}

    \begin{minipage}[t]{0.45\textwidth}
        \includegraphics[width=\textwidth]{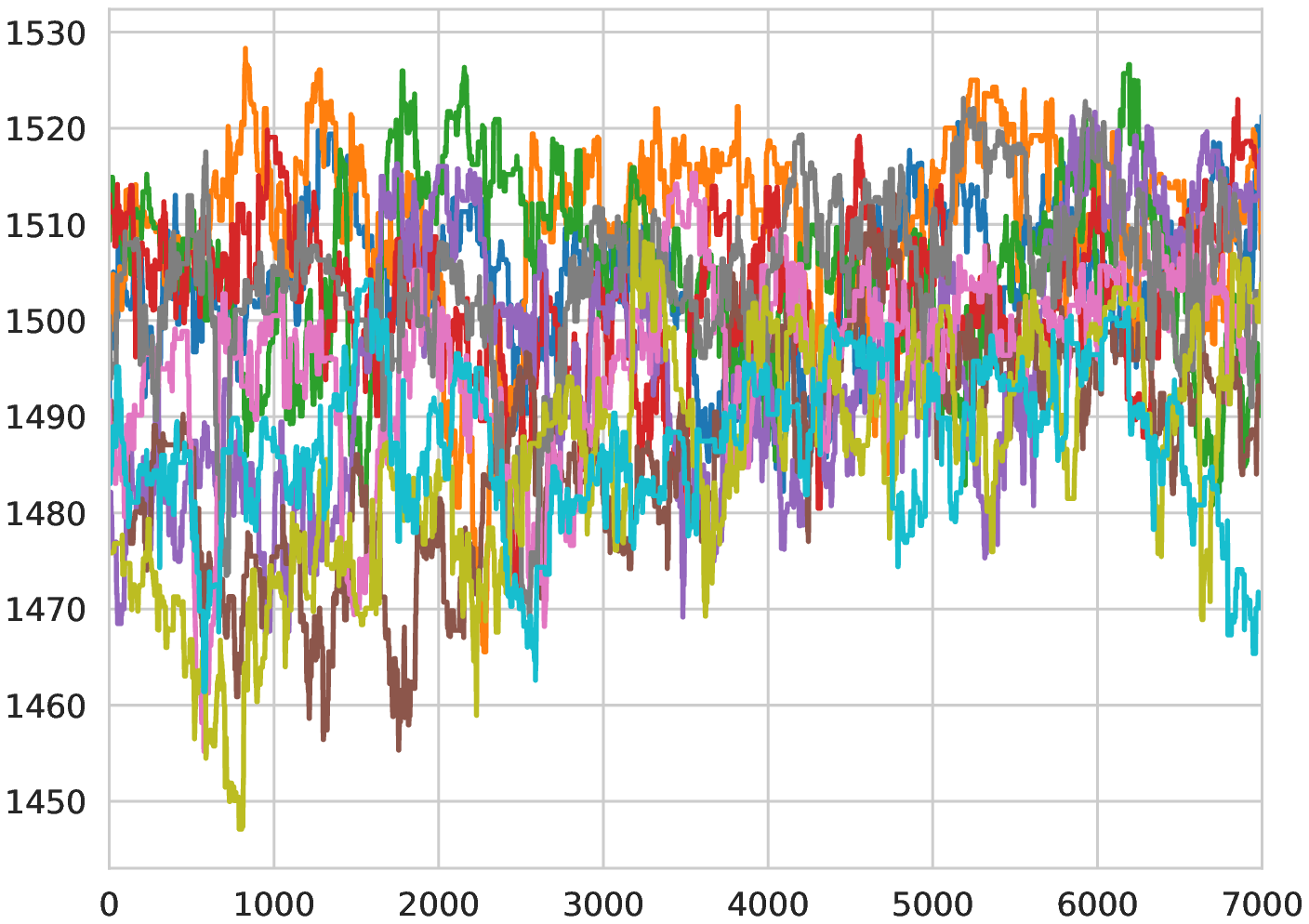}
    \end{minipage}
    ~
    \begin{minipage}[t]{0.45\textwidth}
        \includegraphics[width=\textwidth]{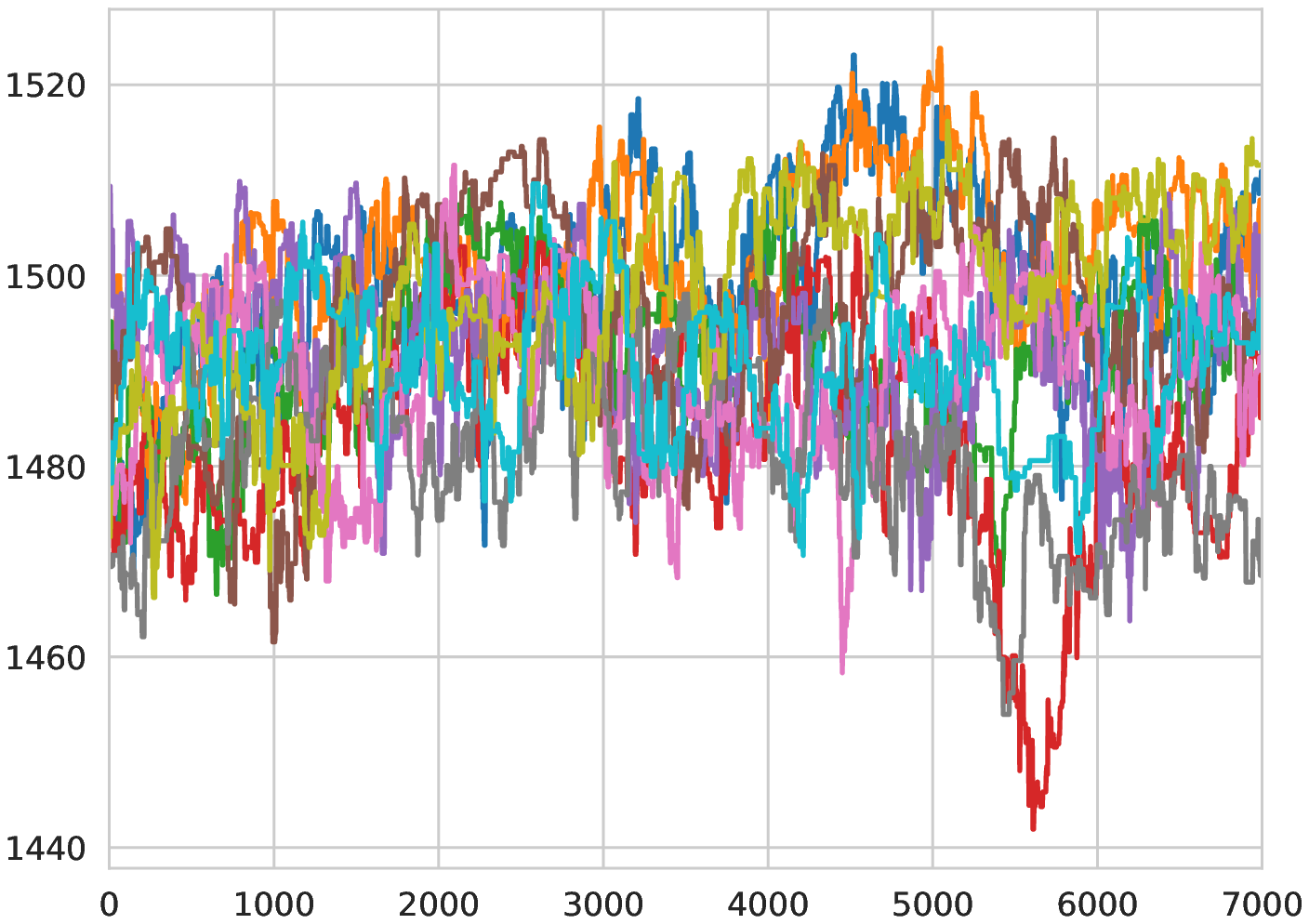}
    \end{minipage}
\caption{Ten size trajectories (top row), estimated size auto-correlations (middle row) and the graph log-likelihoods (bottom row) for the PG-sampler with $\delta=5$ (left panel) and $\delta=50$ (right panel).
}
    \label{fig:AR1-5_est_traj}
\end{figure}

\begin{figure}[ht!]
    \centering

    \begin{minipage}[l]{0.45\textwidth}
        \includegraphics[width=\textwidth,trim={3cm 0 0 0},clip]{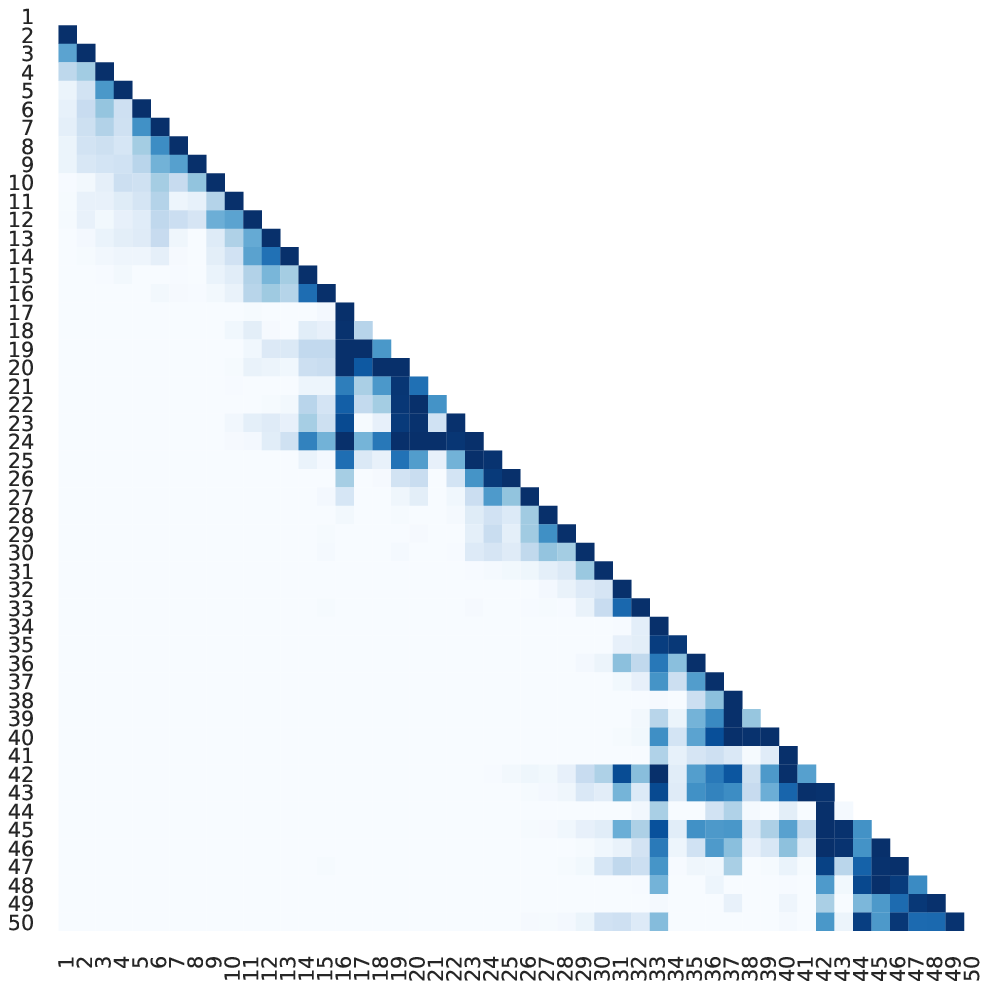}
    \end{minipage}
    ~
    \begin{minipage}[l]{0.45\textwidth}
        \includegraphics[width=\textwidth,trim={3cm 0 0 0},clip]{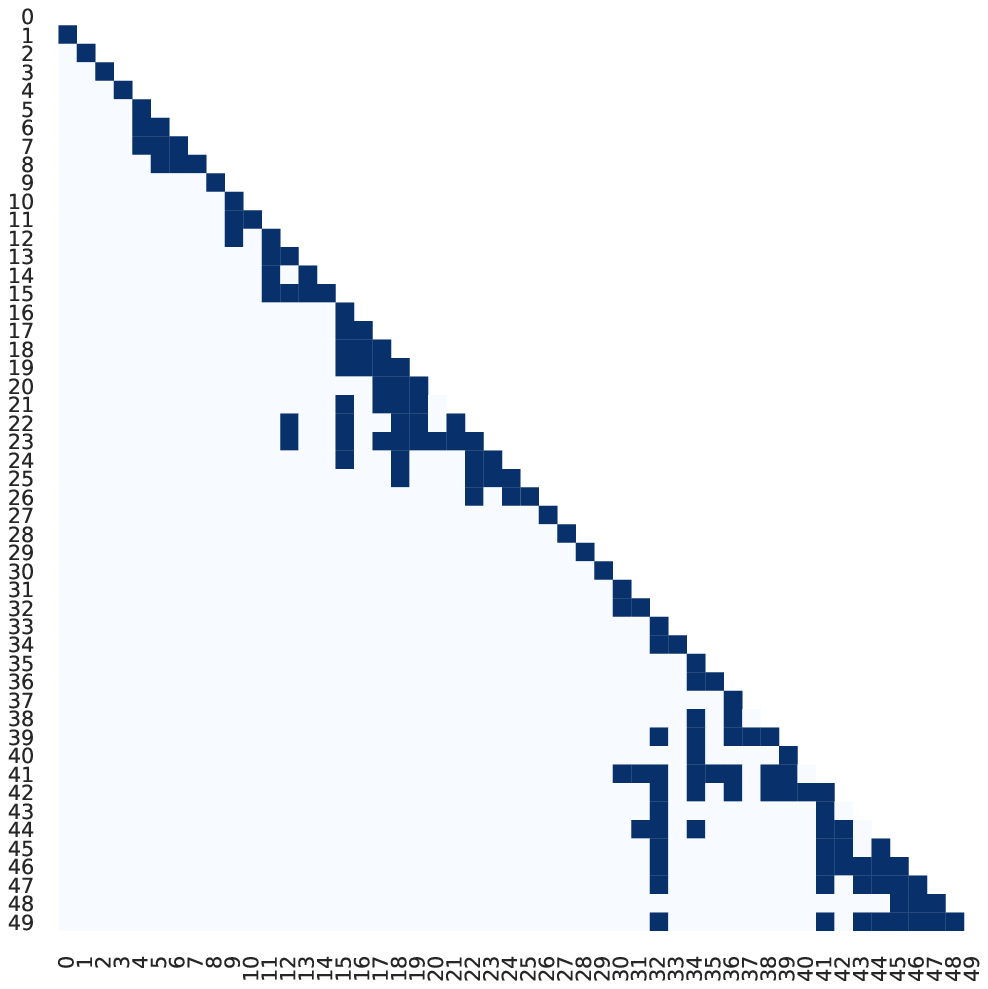}
    \end{minipage}
    \caption{Heatmaps (left panel) and MAP graph estimates (right panel) for the MH-sampler with junction tree randomization at every $\lambda=100$ iteration.
    }
    \label{fig:green_hm}
\end{figure}

\begin{figure}[ht!]
    \centering
      \begin{minipage}[t]{0.45\textwidth}
       \includegraphics[width=\textwidth]{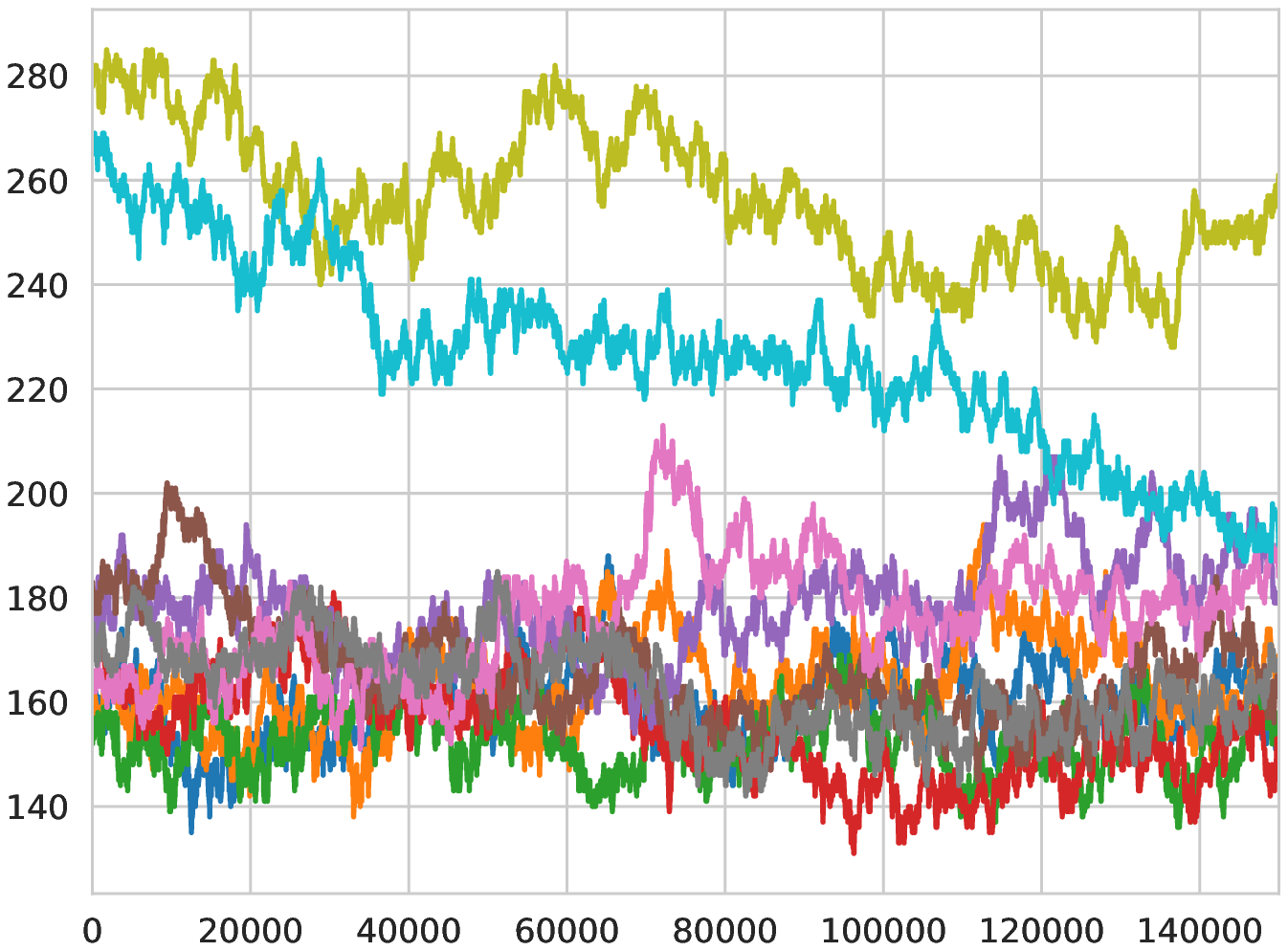}
   \end{minipage}

   \begin{minipage}[t]{0.45\textwidth}
       \includegraphics[width=\textwidth]{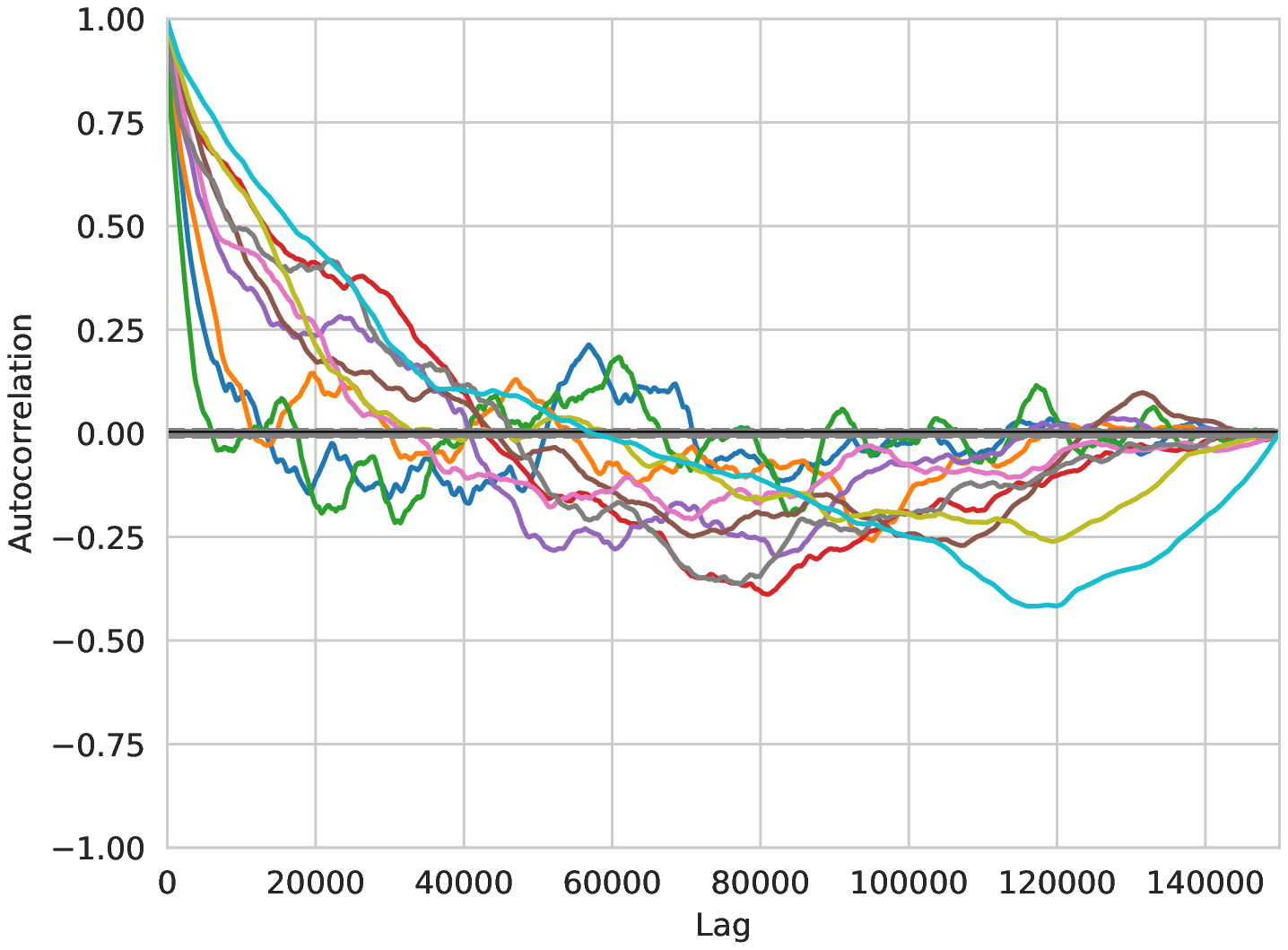}
   \end{minipage}

    \begin{minipage}[b]{0.45\textwidth}
        \includegraphics[width=\textwidth]{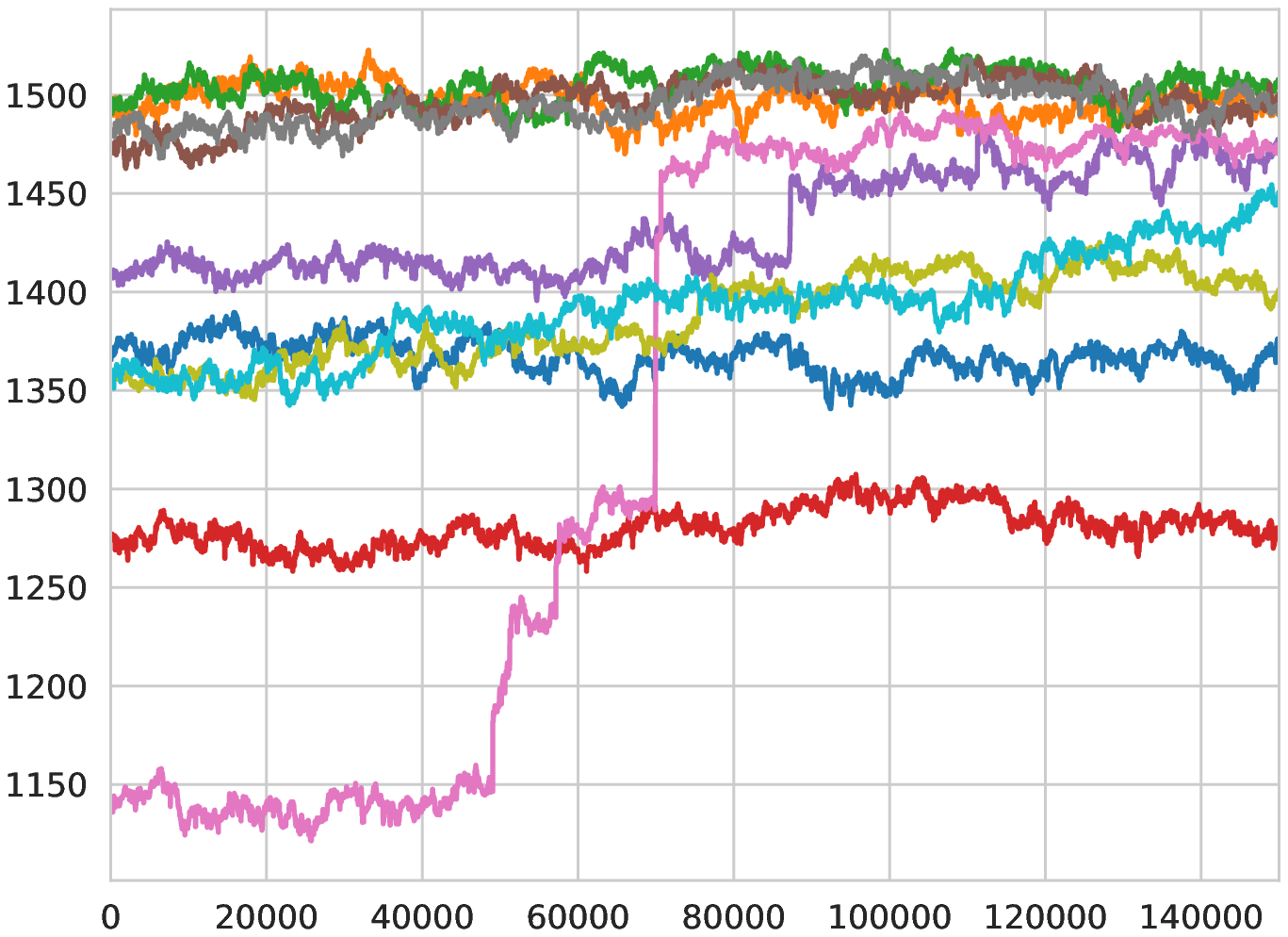}
    \end{minipage}
    \caption{Ten size trajectories (top panel), estimated size auto-correlations (middle panel) and the graph log-likelihoods (bottom panel) for the MH-sampler with junction tree randomization at each $\lambda=100$ iteration.}
    \label{fig:green_stat}
\end{figure}

\section*{Acknowledgements}\label{sec:acknowledgements}
We are grateful to the editor and two anonymous referees for the very helpful and constructive comments, which led to substantial improvement of the paper.
We would also like to thank the PDC Center for High Performance Computing at the KTH Royal Institute of Technology, for providing computing resources for our numerical study.
\clearpage

\clearpage
\appendix
\section{}\label{sec:appendix}

\subsection{Graph theory}\label{sec:graphtheory}
Given a set \(V = \{a_1, \ldots, a_p\}\) of $p \in \nsetpos$ distinct elements,
an \emph{undirected graph} $\graph$ with \emph{nodes} $\graphnodeset$
is specified by a set of \emph{edges} $\graphedgeset \subseteq \graphnodeset
\times \graphnodeset$, and we write $\graph = (\graphnodeset ,\graphedgeset)$.
In addition, we say that $\graph' = (\graphnodeset', \graphedgeset')$ is a
\emph{subgraph} of $\graph$ if $\graphnodeset' \subseteq \graphnodeset$
and $\graphedgeset' \subset \graphedgeset$. For any pair $(a, b)$ of nodes
in $\graph$, a \emph{path} from $a$ to $b$, denoted by $a \sim b$, is a
sequence $\{ a_{n_k} \}_{k = 1}^{\ell + 1}$, with $\ell \in \intvect{1}{p - 1}$, of
distinct nodes such that $a_{n_1} = a$, $a_{n_{\ell + 1}} = b$, and
$(a_{n_k}, a_{n_{k + 1}}) \in \graphedgeset$ for all $k \in \intvect{1}{\ell}$.
Here $\ell$ is called the \emph{length} of the path.  A graph is called a
\emph{tree} if there is a unique path between any pair of nodes.
A graph is \emph{connected} when there is a path between every pair
of nodes, and a \emph{subtree} is a connected subgraph of a tree.
Let $\graph =  (\graphnodeset ,\graphedgeset)$ be a graph and $A$, $B$,
and $S$ subsets of $\graphnodeset$; then $S$ is said to \emph{separate}
$A$ from $B$ if for all $a \in A$ and $b \in B$, every path $a \sim b$ intersects $S$.
This is denoted by $A \perp_{\graph} B \mid S$. A graph is \emph{complete}
if $\graphedgeset = \graphnodeset \times \graphnodeset$. Let $\graphnodeset'
\subseteq \graphnodeset$; then the \emph{induced subgraph}
$\graph {[\graphnodeset']} = (\graphnodeset', \graphedgeset')$ is the subgraph
of $\graph$ with nodes $\graphnodeset'$ and edges $\graphedgeset'$ given
by the subset of edges in $\graphedgeset$ having both endpoints in $\graphnodeset'$.
We write $\graph' = (\graphnodeset', \graphedgeset') \leq \graph = (\graphnodeset,
\graphedgeset)$ to indicate that $\graph' = \graph {[\graphnodeset']}$. A subset
$W \subseteq \graphnodeset$ is a \emph{complete set} if it induces a complete
subgraph. A complete subgraph is called a \emph{clique} if it is not an induced
subgraph of any other complete subgraph. We denote by $\clset{G}$ the family
of cliques formed by a graph $G$.\footnote{We use calligraphy
uppercase to denote families of graphs, or, more generally, families of sets
(as a graph is, given the nodes, specified through the edge set). Consequently,
calligraphy uppercase will also used to denote $\sigma$-fields.} A triple $(A, B, S)$
of disjoint subsets of $\graphnodeset$ is a \emph{decomposition} of
$\graph = (\graphnodeset, \graphedgeset)$ if $A \cup B \cup S=\graphnodeset$,
$A \neq \varnothing$, $B \neq \varnothing$, $S$ is complete, and it holds that
$A \perp_{\graph} B \mid S$.

\subsection{Proofs and lemmas}\label{sec:proofslemmas}
The following lemma, proved in a slightly different form in \cite{doi:10.1198/jcgs.2009.07129}, establishes that each extension
\eqref{def:extended:tree:meas} has the correct marginal
w.r.t. the graph, i.e., that $\graphtarg[\A]$ is the distribution
of $\trgr(\tau)$ when $\tau \sim \trtarg[\A]$.
\begin{lemma} \label{lemma:preservation:marginal}
	For all $\A \subseteq \graphnodeset$ and $h \in \mf{\graphfd}$,
	$$
		\E_{\trtarg[\A]} \left[ h \circ \trgr(\tau) \right] = \graphtarg[\A] h,
	$$
	where $\trtarg[\A]$ is defined in \eqref{def:extended:tree:meas}.
\end{lemma}
\begin{proof}
	\label{lemma:preservation:marginal:proof}
	Since
	\begin{multline*}
		\untrtarg[\A] \1_{\trsp[\A]}
		= \sum_{\tree \in \trsp[\A]} \untrtarg[\A](\tree)
		= \sum_{\graph \in \graphsp[\A]} \sum_{\tree \in \grtr(\graph)}
		\frac{\ungraphtarg[\A] \circ \trgr(\tree)}{\mu \circ \trgr(\tree)} \\
		= \sum_{\graph \in \graphsp[\A]} \mu(\graph)
		\frac{\ungraphtarg[\A](\graph)}{\mu(\graph)}
		= \ungraphtarg[\A] \1_{\graphsp[\A]},
	\end{multline*}
	it holds that
	\begin{equation} \label{eq:trtarg:alt:representation}
		\trtarg[\A](\rmd \tree) =
		\frac{\graphtarg[\A] \circ \trgr(\tree)}
		{\mu \circ \trgr(\tree)} \, \cm{\rmd \tree}.
	\end{equation}
	Now, let $h \in \mf{\graphfd[\A]}$; then by a similar computation,
	\begin{multline*}
		\E_{\trtarg[\A]} \left[ h \circ \trgr(\tree) \right]
		= \sum_{\graph \in \graphsp[\A]} \sum_{\tree \in \grtr(\graph)}
		h \circ \trgr(\tree) \frac{\graphtarg[\A] \circ \trgr(\tree)}
		{\mu \circ \trgr(\tree)} \\
		= \sum_{\graph \in \graphsp[\A]} \mu(\graph) h(\graph)
		\frac{\graphtarg[\A](\graph)}{\mu(\graph)} = \graphtarg[\A] h,
	\end{multline*}
	which completes the proof.
\end{proof}

\begin{proof}[Proof of Theorem \ref{thm:asymvar}]
\label{proof:asymvar}
	First, as established in \cite[Proposition~8]{chopin:singh:2015}, the standard
	PG kernel $\PG{p}$ is $\targ{1:p}$-reversible. As mentioned above, the kernel
	$\G{p}$ is straightforwardly $\targ{1:p}$-reversible as a standard Gibbs substep.
	Moreover, for all $\x{1:p} \in \xsp{1:p}$, $\G{p}(\x{1:p}, \xsp{1:p - 1} \times \{\x{p} \}) = 1$
	and $\G{p}$ dominates trivially the Dirac mass on the \emph{off-diagonal},
	in the sense that for all $A \in \xfd{1:p}$ and $\x{1:p} \in \xsp{1:p}$,
	$\G{p}(\x{1:p}, A \setminus \x{1:p}) \geq \delta_{\x{1:p}}(A \setminus \x{1:p}) = 0$.
	The assumptions of \cite[Lemma~18]{maire:douc:olsson:2014} are thus fulfilled,
	and the proof is concluded through application of the latter.
\end{proof}

\bibliographystyle{imsart-nameyear}
\bibliography{allbib}

\begin{thebibliography}{24}

\bibitem[\protect\citeauthoryear{Andrieu, Doucet and
  Holenstein}{2010}]{andrieu2010particle}
\begin{barticle}[author]
\bauthor{\bsnm{Andrieu},~\bfnm{Christophe}\binits{C.}},
  \bauthor{\bsnm{Doucet},~\bfnm{Arnaud}\binits{A.}} \AND
  \bauthor{\bsnm{Holenstein},~\bfnm{Roman}\binits{R.}}
(\byear{2010}).
\btitle{{Particle Markov chain Monte Carlo methods}}.
\bjournal{Journal of the Royal Statistical Society: Series B (Statistical
  Methodology)}
\bvolume{72}
\bpages{269--342}.
\end{barticle}
\endbibitem

\bibitem[\protect\citeauthoryear{Bornn et~al.}{2011}]{bornn2011bayesian}
\begin{barticle}[author]
\bauthor{\bsnm{Bornn},~\bfnm{Luke}\binits{L.}},
  \bauthor{\bsnm{Caron},~\bfnm{Fran{\c{c}}ois}\binits{F.}} \betal{et~al.}
(\byear{2011}).
\btitle{Bayesian clustering in decomposable graphs}.
\bjournal{Bayesian Analysis}
\bvolume{6}
\bpages{829--846}.
\end{barticle}
\endbibitem

\bibitem[\protect\citeauthoryear{Capp\'e, Moulines and
  Ryd\'en}{2005}]{cappe:moulines:ryden:2005}
\begin{bbook}[author]
\bauthor{\bsnm{Capp\'e},~\bfnm{O}\binits{O.}},
  \bauthor{\bsnm{Moulines},~\bfnm{E.}\binits{E.}} \AND
  \bauthor{\bsnm{Ryd\'en},~\bfnm{T.}\binits{T.}}
(\byear{2005}).
\btitle{Inference in hidden {M}arkov models}.
\bpublisher{Springer New York}.
\end{bbook}
\endbibitem

\bibitem[\protect\citeauthoryear{Chopin and Singh}{2015}]{chopin:singh:2015}
\begin{barticle}[author]
\bauthor{\bsnm{Chopin},~\bfnm{Nicolas}\binits{N.}} \AND
  \bauthor{\bsnm{Singh},~\bfnm{Sumeetpal~S.}\binits{S.~S.}}
(\byear{2015}).
\btitle{On particle {Gibbs} sampling}.
\bjournal{Bernoulli}
\bvolume{21}
\bpages{1855--1883}.
\bdoi{10.3150/14-BEJ629}
\end{barticle}
\endbibitem

\bibitem[\protect\citeauthoryear{Dawid and Lauritzen}{1993}]{dawid1993}
\begin{barticle}[author]
\bauthor{\bsnm{Dawid},~\bfnm{A.~Philip}\binits{A.~P.}} \AND
  \bauthor{\bsnm{Lauritzen},~\bfnm{S.~L.}\binits{S.~L.}}
(\byear{1993}).
\btitle{Hyper {Markov} laws in the statistical analysis of decomposable
  graphical models}.
\bjournal{The Annals of Statistics}
\bvolume{21}
\bpages{1272--1317}.
\bdoi{10.1214/aos/1176349260}
\end{barticle}
\endbibitem

\bibitem[\protect\citeauthoryear{{Del Moral}, Doucet and
  Jasra}{2006}]{10.2307/3879283}
\begin{barticle}[author]
\bauthor{\bsnm{{Del Moral}},~\bfnm{Pierre}\binits{P.}},
  \bauthor{\bsnm{Doucet},~\bfnm{Arnaud}\binits{A.}} \AND
  \bauthor{\bsnm{Jasra},~\bfnm{Ajay}\binits{A.}}
(\byear{2006}).
\btitle{Sequential {M}onte {C}arlo samplers}.
\bjournal{Journal of the Royal Statistical Society. Series B (Statistical
  Methodology)}
\bvolume{68}
\bpages{411-436}.
\end{barticle}
\endbibitem

\bibitem[\protect\citeauthoryear{Dellaportas and
  Forster}{1999}]{10.2307/2673658}
\begin{barticle}[author]
\bauthor{\bsnm{Dellaportas},~\bfnm{Petros}\binits{P.}} \AND
  \bauthor{\bsnm{Forster},~\bfnm{Jonathan~J.}\binits{J.~J.}}
(\byear{1999}).
\btitle{{Markov} chain {Monte Carlo} model determination for hierarchical and
  graphical log-linear models}.
\bjournal{Biometrika}
\bvolume{86}
\bpages{615-633}.
\end{barticle}
\endbibitem

\bibitem[\protect\citeauthoryear{Diestel}{2005}]{diestel2005graph}
\begin{bbook}[author]
\bauthor{\bsnm{Diestel},~\bfnm{Reinhard}\binits{R.}}
(\byear{2005}).
\btitle{Graph theory (Graduate texts in mathematics)}
\bvolume{173}.
\bpublisher{Springer Heidelberg}.
\end{bbook}
\endbibitem

\bibitem[\protect\citeauthoryear{Edwards and
  Havr{\'a}nek}{1985}]{10.2307/2336086}
\begin{barticle}[author]
\bauthor{\bsnm{Edwards},~\bfnm{David}\binits{D.}} \AND
  \bauthor{\bsnm{Havr{\'a}nek},~\bfnm{Tom{\'a}{\v s}}\binits{T.}}
(\byear{1985}).
\btitle{A fast procedure for model search in multidimensional contingency
  tables}.
\bjournal{Biometrika}
\bvolume{72}
\bpages{339-351}.
\end{barticle}
\endbibitem

\bibitem[\protect\citeauthoryear{Elmasri}{2017a}]{elmasri2017decomposable}
\begin{barticle}[author]
\bauthor{\bsnm{Elmasri},~\bfnm{Mohamad}\binits{M.}}
(\byear{2017}a).
\btitle{On decomposable random graphs}.
\bjournal{ArXiv e-prints}.
\end{barticle}
\endbibitem

\bibitem[\protect\citeauthoryear{Elmasri}{2017b}]{elmasri2017sub}
\begin{barticle}[author]
\bauthor{\bsnm{Elmasri},~\bfnm{Mohamad}\binits{M.}}
(\byear{2017}b).
\btitle{Sub-clustering in decomposable graphs and size-varying junction trees}.
\bjournal{ArXiv e-prints}.
\end{barticle}
\endbibitem

\bibitem[\protect\citeauthoryear{Giudici and Green}{1999}]{Giudici01121999}
\begin{barticle}[author]
\bauthor{\bsnm{Giudici},~\bfnm{P.}\binits{P.}} \AND
  \bauthor{\bsnm{Green},~\bfnm{Peter~J.}\binits{P.~J.}}
(\byear{1999}).
\btitle{{Decomposable graphical Gaussian model determination}}.
\bjournal{Biometrika}
\bvolume{86}
\bpages{785-801}.
\bdoi{10.1093/biomet/86.4.785}
\end{barticle}
\endbibitem

\bibitem[\protect\citeauthoryear{Green and Thomas}{2013}]{Green01032013}
\begin{barticle}[author]
\bauthor{\bsnm{Green},~\bfnm{Peter~J.}\binits{P.~J.}} \AND
  \bauthor{\bsnm{Thomas},~\bfnm{Alun}\binits{A.}}
(\byear{2013}).
\btitle{Sampling decomposable graphs using a {M}arkov chain on junction trees}.
\bjournal{Biometrika}
\bvolume{100}
\bpages{91-110}.
\bdoi{10.1093/biomet/ass052}
\end{barticle}
\endbibitem

\bibitem[\protect\citeauthoryear{Green and Thomas}{2017}]{green2017structural}
\begin{barticle}[author]
\bauthor{\bsnm{Green},~\bfnm{Peter~J}\binits{P.~J.}} \AND
  \bauthor{\bsnm{Thomas},~\bfnm{Alun}\binits{A.}}
(\byear{2017}).
\btitle{{A structural Markov property for decomposable graph laws that allows
  control of clique intersections}}.
\bjournal{Biometrika}
\bvolume{105}
\bpages{19-29}.
\bdoi{10.1093/biomet/asx072}
\end{barticle}
\endbibitem

\bibitem[\protect\citeauthoryear{Jones et~al.}{2005}]{2005}
\begin{barticle}[author]
\bauthor{\bsnm{Jones},~\bfnm{Beatrix}\binits{B.}},
  \bauthor{\bsnm{Carvalho},~\bfnm{Carlos}\binits{C.}},
  \bauthor{\bsnm{Dobra},~\bfnm{Adrian}\binits{A.}},
  \bauthor{\bsnm{Hans},~\bfnm{Chris}\binits{C.}},
  \bauthor{\bsnm{Carter},~\bfnm{Chris}\binits{C.}} \AND
  \bauthor{\bsnm{West},~\bfnm{Mike}\binits{M.}}
(\byear{2005}).
\btitle{Experiments in stochastic computation for high-dimensional graphical
  models}.
\bjournal{Statistical Science}
\bvolume{20}
\bpages{388-400}.
\end{barticle}
\endbibitem

\bibitem[\protect\citeauthoryear{Lauritzen}{1996}]{lauritzen1996}
\begin{bbook}[author]
\bauthor{\bsnm{Lauritzen},~\bfnm{Steffen~L.}\binits{S.~L.}}
(\byear{1996}).
\btitle{Graphical Models}.
\bpublisher{Oxford University Press}.
\end{bbook}
\endbibitem

\bibitem[\protect\citeauthoryear{Madigan and Raftery}{1994}]{10.2307/2291017}
\begin{barticle}[author]
\bauthor{\bsnm{Madigan},~\bfnm{David}\binits{D.}} \AND
  \bauthor{\bsnm{Raftery},~\bfnm{Adrian~E.}\binits{A.~E.}}
(\byear{1994}).
\btitle{Model selection and accounting for model uncertainty in graphical
  models using {Occam's} window}.
\bjournal{Journal of the American Statistical Association}
\bvolume{89}
\bpages{1535-1546}.
\end{barticle}
\endbibitem

\bibitem[\protect\citeauthoryear{Maire, Douc and
  Olsson}{2014}]{maire:douc:olsson:2014}
\begin{barticle}[author]
\bauthor{\bsnm{Maire},~\bfnm{F.}\binits{F.}},
  \bauthor{\bsnm{Douc},~\bfnm{R.}\binits{R.}} \AND
  \bauthor{\bsnm{Olsson},~\bfnm{J.}\binits{J.}}
(\byear{2014}).
\btitle{Comparison of asymptotic variances of inhomogeneous {M}arkov chains
  with application to {M}arkov chain {M}onte {C}arlo methods}.
\bjournal{The Annals of Statistics}
\bvolume{42}
\bpages{1483--1510}.
\end{barticle}
\endbibitem

\bibitem[\protect\citeauthoryear{Markenzon, Vernet and
  Araujo}{2008}]{randchord}
\begin{barticle}[author]
\bauthor{\bsnm{Markenzon},~\bfnm{Lilian}\binits{L.}},
  \bauthor{\bsnm{Vernet},~\bfnm{Oswaldo}\binits{O.}} \AND
  \bauthor{\bsnm{Araujo},~\bfnm{LuizHenrique}\binits{L.}}
(\byear{2008}).
\btitle{Two methods for the generation of chordal graphs}.
\bjournal{Annals of Operations Research}
\bvolume{157}
\bpages{47-60}.
\bdoi{10.1007/s10479-007-0190-4}
\end{barticle}
\endbibitem

\bibitem[\protect\citeauthoryear{Massam, Liu and
  Dobra}{2009}]{10.2307/25662199}
\begin{barticle}[author]
\bauthor{\bsnm{Massam},~\bfnm{H{\'e}l{\`e}ne}\binits{H.}},
  \bauthor{\bsnm{Liu},~\bfnm{Jinnan}\binits{J.}} \AND
  \bauthor{\bsnm{Dobra},~\bfnm{Adrian}\binits{A.}}
(\byear{2009}).
\btitle{A conjugate prior for discrete hierarchical log-linear models}.
\bjournal{The Annals of Statistics}
\bvolume{37}
\bpages{3431-3467}.
\end{barticle}
\endbibitem

\bibitem[\protect\citeauthoryear{Olsson, Pavlenko and Rios}{2018}]{cta}
\begin{barticle}[author]
\bauthor{\bsnm{Olsson},~\bfnm{Jimmy}\binits{J.}},
  \bauthor{\bsnm{Pavlenko},~\bfnm{Tatjana}\binits{T.}} \AND
  \bauthor{\bsnm{Rios},~\bfnm{Felix~Leopoldo}\binits{F.~L.}}
(\byear{2018}).
\btitle{{Sequential sampling of junction trees for decomposable graphs}}.
\bjournal{ArXiv e-prints}.
\end{barticle}
\endbibitem

\bibitem[\protect\citeauthoryear{Speed and Kiiveri}{1986}]{10.2307/2241271}
\begin{barticle}[author]
\bauthor{\bsnm{Speed},~\bfnm{T.~P.}\binits{T.~P.}} \AND
  \bauthor{\bsnm{Kiiveri},~\bfnm{H.~T.}\binits{H.~T.}}
(\byear{1986}).
\btitle{{Gaussian Markov} distributions over finite graphs}.
\bjournal{The Annals of Statistics}
\bvolume{14}
\bpages{138-150}.
\end{barticle}
\endbibitem

\bibitem[\protect\citeauthoryear{Stingo and
  Marchetti}{2015}]{stingo2015efficient}
\begin{barticle}[author]
\bauthor{\bsnm{Stingo},~\bfnm{Francesco}\binits{F.}} \AND
  \bauthor{\bsnm{Marchetti},~\bfnm{Giovanni~M}\binits{G.~M.}}
(\byear{2015}).
\btitle{Efficient local updates for undirected graphical models}.
\bjournal{Statistics and Computing}
\bvolume{25}
\bpages{159--171}.
\end{barticle}
\endbibitem

\bibitem[\protect\citeauthoryear{Thomas and
  Green}{2009}]{doi:10.1198/jcgs.2009.07129}
\begin{barticle}[author]
\bauthor{\bsnm{Thomas},~\bfnm{Alun}\binits{A.}} \AND
  \bauthor{\bsnm{Green},~\bfnm{Peter~J.}\binits{P.~J.}}
(\byear{2009}).
\btitle{Enumerating the junction trees of a decomposable graph}.
\bjournal{Journal of Computational and Graphical Statistics}
\bvolume{18}
\bpages{930-940}.
\bdoi{10.1198/jcgs.2009.07129}
\end{barticle}
\endbibitem

\end{thebibliography}

\end{document}